\documentclass{article} 
\usepackage[centertags,sumlimits,intlimits,namelimits,reqno]{amsmath}
\usepackage{latexsym,amsfonts,amssymb,exscale,enumerate,amsthm}
\usepackage{tikz}
\usetikzlibrary{backgrounds,circuits,circuits.ee.IEC,shapes,fit,matrix,cd}
\usepackage[all,cmtip,2cell]{xy}
\UseTwocells
\usepackage{lscape}

\usepackage[colorinlistoftodos,textsize=small]{todonotes}
\presetkeys{todonotes}{inline,backgroundcolor=orange!30!white}{}

\definecolor{myurlcolor}{rgb}{0.6,0,0}
\definecolor{mycitecolor}{rgb}{0,0,0.8}
\definecolor{myrefcolor}{rgb}{0,0,0.8}
\usepackage[pagebackref]{hyperref}
\hypersetup{colorlinks, linkcolor=myrefcolor, citecolor=mycitecolor, urlcolor=myurlcolor}

\newtheorem{theorem}{Theorem}[section]
\newtheorem{proposition}[theorem]{Proposition}   

\newtheorem{definition}[theorem]{Definition}   
\newtheorem{lemma}[theorem]{Lemma}   
\newtheorem{corollary}[theorem]{Corollary}   
\newtheorem*{proposition*}{Proposition}
\newtheorem*{theorem*}{Theorem}

\theoremstyle{remark}
\newtheorem{example}[theorem]{Example}
\newtheorem*{example*}{Example}

\newcommand{\maps}{\colon}
\newcommand{\define}[1]{{\bf \boldmath #1}}
\newcommand{\vect}[1]{{\F^{#1} \oplus {(\F^{#1})}^\ast}}

\newcommand{\id}{\mathrm{id}}
\newcommand{\s}{\sigma}

\newcommand{\asrelto}{\nrightarrow}

\newcommand\R{{\mathbb R}}

\newcommand\F{{\mathbb F}}

\newcommand{\C}{\mathcal{C}}
\renewcommand{\H}{\mathcal{H}}
\newcommand{\E}{\mathcal{E}}
\newcommand{\M}{\mathcal{M}}

\newcommand{\Circ}{\mathrm{Circ}}
\newcommand{\Circuit}{\mathrm{Circuit}}
\newcommand{\Lag}{\mathrm{Lag}}
\newcommand{\Dirich}{\mathrm{Dirich}}
\newcommand{\Lin}{\mathrm{Lin}}
\newcommand{\Rel}{\mathrm{Rel}}
\newcommand{\Corel}{\mathrm{Corel}}

\newcommand{\Cospan}{\mathrm{Cospan}}
\newcommand{\opp}{\mathrm{op}}
\newcommand{\idn}{\mathrm{id}}
\newcommand{\Fin}{\mathrm{Fin}}
\newcommand{\Vect}{\mathrm{Vect}}
\newcommand{\Set}{\mathrm{Set}}

\newcommand{\Inj}{\mathrm{Inj}}

\tikzcdset{every label/.append style = {font = \normalsize}}

\pgfdeclarelayer{edgelayer}
\pgfdeclarelayer{nodelayer}
\pgfsetlayers{edgelayer,nodelayer,main}

\tikzset{font=\footnotesize}
\tikzstyle{none}=[inner sep=0pt]
\tikzstyle{connection}=[circuit symbol open,
    circuit symbol size=width .5 height .5,
    shape=circle ee,
    inner sep=0.25\pgflinewidth]
\tikzstyle{bdot}=[circle, fill=black, draw, inner sep=1.5pt, anchor=center]
\tikzstyle{circ}=[circle,fill=black,draw,inner sep=3pt]

\newcommand{\mult}[1]
{
\begin{aligned}
    \resizebox{#1}{!}{
\begin{tikzpicture}
	\begin{pgfonlayer}{nodelayer}
		\node [style=none] (0) at (1, -0) {};
		\node [style=circ] (1) at (0.125, -0) {};
		\node [style=none] (2) at (-1, 0.5) {};
		\node [style=none] (3) at (-1, -0.5) {};
	\end{pgfonlayer}
	\begin{pgfonlayer}{edgelayer}
		\draw[line width=2pt] (0.center) to (1.center);
		\draw[line width=2pt] [in=0, out=120, looseness=1.20] (1.center) to (2.center);
		\draw[line width=2pt] [in=0, out=-120, looseness=1.20] (1.center) to (3.center);
	\end{pgfonlayer}
      \end{tikzpicture}}
\end{aligned}
}

\newcommand{\unit}[1]
{
  \begin{aligned}
    \resizebox{#1}{!}{
\begin{tikzpicture}
	\begin{pgfonlayer}{nodelayer}
		\node [style=none] (spaceup) at (0, 0.5) {};
		\node [style=none] (spacedown) at (0, -0.5) {};
		\node [style=none] (0) at (1, -0) {};
		\node [style=none] (1) at (-1, -0) {};
		\node [style=circ] (2) at (0, -0) {};
	\end{pgfonlayer}
	\begin{pgfonlayer}{edgelayer}
		\draw[line width=2pt] (0.center) to (2);
	\end{pgfonlayer}
      \end{tikzpicture}}
  \end{aligned}
}

\newcommand{\comult}[1]
{
\begin{aligned}
    \resizebox{#1}{!}{
\begin{tikzpicture}
	\begin{pgfonlayer}{nodelayer}
		\node [style=none] (0) at (-1, -0) {};
		\node [style=circ] (1) at (-0.125, -0) {};
		\node [style=none] (2) at (1, 0.5) {};
		\node [style=none] (3) at (1, -0.5) {};
	\end{pgfonlayer}
	\begin{pgfonlayer}{edgelayer}
		\draw[line width=2pt] (0.center) to (1.center);
		\draw[line width=2pt] [in=180, out=60, looseness=1.20] (1.center) to (2.center);
		\draw[line width=2pt] [in=180, out=-60, looseness=1.20] (1.center) to (3.center);
	\end{pgfonlayer}
      \end{tikzpicture}}
\end{aligned}
}

\newcommand{\counit}[1]
{
  \begin{aligned}
    \resizebox{#1}{!}{
\begin{tikzpicture}
	\begin{pgfonlayer}{nodelayer}
		\node [style=none] (spaceup) at (0, 0.5) {};
		\node [style=none] (spacedown) at (0, -0.5) {};
		\node [style=none] (0) at (-1, -0) {};
		\node [style=none] (1) at (1, -0) {};
		\node [style=circ] (2) at (0, -0) {};
	\end{pgfonlayer}
	\begin{pgfonlayer}{edgelayer}
		\draw[line width=2pt] (0.center) to (2);
	\end{pgfonlayer}
      \end{tikzpicture}}
  \end{aligned}
}

\newcommand{\idone}[1]
{
  \begin{aligned}
    \resizebox{#1}{!}{
     \begin{tikzpicture}
	\begin{pgfonlayer}{nodelayer}
		\node [style=none] (0) at (-1, -0) {};
		\node [style=none] (1) at (1, -0) {};
		\node [style=none] (2) at (0, 0.5) {};
		\node [style=none] (3) at (0, -0.5) {};
	\end{pgfonlayer}
	\begin{pgfonlayer}{edgelayer}
		\draw[line width=2pt] (1.center) to (0.center);
	\end{pgfonlayer}
\end{tikzpicture} 
    }
  \end{aligned}
}

\newcommand{\swap}[1]
{
  \begin{aligned}
    \resizebox{#1}{!}{
\begin{tikzpicture}
	\begin{pgfonlayer}{nodelayer}
		\node [style=none] (2) at (-0.5, -0.5) {};
		\node [style=none] (3) at (-2, 0.5) {};
		\node [style=none] (4) at (-0.5, 0.5) {};
		\node [style=none] (5) at (-2, -0.5) {};
	\end{pgfonlayer}
	\begin{pgfonlayer}{edgelayer}
		\draw[line width=2pt] [in=180, out=0, looseness=1.00] (3.center) to (2.center);
		\draw[line width=2pt] [in=0, out=180, looseness=1.00] (4.center) to (5.center);
	\end{pgfonlayer}
\end{tikzpicture}
    }
  \end{aligned}
}
\newcommand{\assocl}[1]
{
  \begin{aligned}
    \resizebox{#1}{!}{
\begin{tikzpicture}
	\begin{pgfonlayer}{nodelayer}
		\node [style=circ] (0) at (0.125, -0) {};
		\node [style=none] (1) at (-1, 0.5) {};
		\node [style=none] (2) at (-1, -0.5) {};
		\node [style=none] (3) at (0, -1) {};
		\node [style=none] (4) at (2.25, -0.5) {};
		\node [style=none] (5) at (0.25, -0) {};
		\node [style=circ] (6) at (1.25, -0.5) {};
		\node [style=none] (7) at (-1, -1) {};
	\end{pgfonlayer}
	\begin{pgfonlayer}{edgelayer}
		\draw[line width=2pt] [in=0, out=120, looseness=1.20] (0.center) to (1.center);
		\draw[line width=2pt] [in=0, out=-120, looseness=1.20] (0.center) to (2.center);
		\draw[line width=2pt] (4.center) to (6);
		\draw[line width=2pt] [in=0, out=120, looseness=1.20] (6) to (5.center);
		\draw[line width=2pt] [in=0, out=-120, looseness=1.20] (6) to (3.center);
		\draw[line width=2pt] (3.center) to (7.center);
	\end{pgfonlayer}
      \end{tikzpicture}}
  \end{aligned}
}

\newcommand{\assocr}[1]
{
  \begin{aligned}
    \resizebox{#1}{!}{
\begin{tikzpicture}
	\begin{pgfonlayer}{nodelayer}
		\node [style=circ] (0) at (0.125, -0.5) {};
		\node [style=none] (1) at (-1, -1) {};
		\node [style=none] (2) at (-1, 0) {};
		\node [style=none] (3) at (0, 0.5) {};
		\node [style=none] (4) at (2.25, 0) {};
		\node [style=none] (5) at (0.25, -0.5) {};
		\node [style=circ] (6) at (1.25, 0) {};
		\node [style=none] (7) at (-1, 0.5) {};
	\end{pgfonlayer}
	\begin{pgfonlayer}{edgelayer}
		\draw[line width=2pt] [in=0, out=-120, looseness=1.20] (0.center) to (1.center);
		\draw[line width=2pt] [in=0, out=120, looseness=1.20] (0.center) to (2.center);
		\draw[line width=2pt] (4.center) to (6);
		\draw[line width=2pt] [in=0, out=-120, looseness=1.20] (6) to (5.center);
		\draw[line width=2pt] [in=0, out=120, looseness=1.20] (6) to (3.center);
		\draw[line width=2pt] (3.center) to (7.center);
	\end{pgfonlayer}
      \end{tikzpicture}}
  \end{aligned}
}

\newcommand{\unitl}[1]
{
  \begin{aligned}
    \resizebox{#1}{!}{
\begin{tikzpicture}
	\begin{pgfonlayer}{nodelayer}
		\node [style=none] (0) at (1, -0) {};
		\node [style=circ] (1) at (0.125, -0) {};
		\node [style=circ] (2) at (-1, 0.5) {};
		\node [style=none] (3) at (-1, -0.5) {};
		\node [style=none] (4) at (-2, -0.5) {};
	\end{pgfonlayer}
	\begin{pgfonlayer}{edgelayer}
		\draw[line width=2pt] (0.center) to (1.center);
		\draw[line width=2pt] [in=0, out=120, looseness=1.20] (1.center) to (2.center);
		\draw[line width=2pt] [in=0, out=-120, looseness=1.20] (1.center) to (3.center);
		\draw[line width=2pt] (4.center) to (3.center);
	\end{pgfonlayer}
\end{tikzpicture}
    }
  \end{aligned}
}

\newcommand{\commute}[1]
{
  \begin{aligned}
    \resizebox{#1}{!}{
\begin{tikzpicture}
	\begin{pgfonlayer}{nodelayer}
		\node [style=none] (0) at (1.25, -0) {};
		\node [style=circ] (1) at (0.375, -0) {};
		\node [style=none] (2) at (-0.5, -0.5) {};
		\node [style=none] (3) at (-2, 0.5) {};
		\node [style=none] (4) at (-0.5, 0.5) {};
		\node [style=none] (5) at (-2, -0.5) {};
	\end{pgfonlayer}
	\begin{pgfonlayer}{edgelayer}
		\draw[line width=2pt] (0.center) to (1.center);
		\draw[line width=2pt] [in=0, out=-120, looseness=1.20] (1.center) to (2.center);
		\draw[line width=2pt] [in=180, out=0, looseness=1.00] (3.center) to (2.center);
		\draw[line width=2pt] [in=0, out=120, looseness=1.20] (1.center) to (4.center);
		\draw[line width=2pt] [in=0, out=180, looseness=1.00] (4.center) to (5.center);
	\end{pgfonlayer}
\end{tikzpicture}
    }
  \end{aligned}
}

\newcommand{\frobs}[1]
{
  \begin{aligned}
    \resizebox{#1}{!}{
\begin{tikzpicture}
	\begin{pgfonlayer}{nodelayer}
		\node [style=none] (0) at (-1.5, 0.5) {};
		\node [style=circ] (1) at (-0.75, 0.5) {};
		\node [style=none] (2) at (0.25, -0) {};
		\node [style=none] (3) at (0.25, 1) {};
		\node [style=circ] (4) at (1, -0.5) {};
		\node [style=none] (5) at (0, -0) {};
		\node [style=none] (6) at (1.75, -0.5) {};
		\node [style=none] (7) at (0, -1) {};
		\node [style=none] (8) at (1.75, 1) {};
		\node [style=none] (9) at (-1.5, -1) {};
	\end{pgfonlayer}
	\begin{pgfonlayer}{edgelayer}
		\draw[line width=2pt] [in=180, out=-60, looseness=1.20] (1) to (2.center);
		\draw[line width=2pt] [in=180, out=60, looseness=1.20] (1) to (3.center);
		\draw[line width=2pt] (0.center) to (1);
		\draw[line width=2pt] (6.center) to (4);
		\draw[line width=2pt] [in=0, out=120, looseness=1.20] (4) to (5.center);
		\draw[line width=2pt] [in=0, out=-120, looseness=1.20] (4) to (7.center);
		\draw[line width=2pt] (3.center) to (8.center);
		\draw[line width=2pt] (7.center) to (9.center);
	\end{pgfonlayer}
\end{tikzpicture}
    }
  \end{aligned}
}

\newcommand{\frobx}[1]
{
  \begin{aligned}
    \resizebox{#1}{!}{
\begin{tikzpicture}
	\begin{pgfonlayer}{nodelayer}
		\node [style=circ] (0) at (-0.5, -0) {};
		\node [style=none] (1) at (-1.5, -0.5) {};
		\node [style=none] (2) at (-1.5, 0.5) {};
		\node [style=circ] (3) at (0.5, -0) {};
		\node [style=none] (4) at (1.5, -0.5) {};
		\node [style=none] (5) at (1.5, 0.5) {};
	\end{pgfonlayer}
	\begin{pgfonlayer}{edgelayer}
		\draw[line width=2pt] [in=0, out=-120, looseness=1.20] (0.center) to (1.center);
		\draw[line width=2pt] [in=0, out=120, looseness=1.20] (0.center) to (2.center);
		\draw[line width=2pt] [in=180, out=-60, looseness=1.20] (3) to (4.center);
		\draw[line width=2pt] [in=180, out=60, looseness=1.20] (3) to (5.center);
		\draw[line width=2pt] (0) to (3);
	\end{pgfonlayer}
\end{tikzpicture}
    }
  \end{aligned}
}

\newcommand{\spec}[1]
{
  \begin{aligned}
    \resizebox{#1}{!}{
\begin{tikzpicture}
	\begin{pgfonlayer}{nodelayer}
		\node [style=none] (0) at (1.75, -0) {};
		\node [style=circ] (1) at (0.75, -0) {};
		\node [style=none] (2) at (0, -0.5) {};
		\node [style=none] (3) at (0, 0.5) {};
		\node [style=circ] (4) at (-0.75, -0) {};
		\node [style=none] (5) at (0, -0.5) {};
		\node [style=none] (6) at (-1.75, -0) {};
		\node [style=none] (7) at (0, 0.5) {};
	\end{pgfonlayer}
	\begin{pgfonlayer}{edgelayer}
		\draw[line width=2pt] (0.center) to (1.center);
		\draw[line width=2pt] [in=0, out=-120, looseness=1.20] (1.center) to (2.center);
		\draw[line width=2pt] [in=0, out=120, looseness=1.20] (1.center) to (3.center);
		\draw[line width=2pt] (6.center) to (4);
		\draw[line width=2pt] [in=180, out=-60, looseness=1.20] (4) to (5.center);
		\draw[line width=2pt] [in=180, out=60, looseness=1.20] (4) to (7.center);
	\end{pgfonlayer}
\end{tikzpicture}
    }
  \end{aligned}
}

\newcommand{\cuptwo}[1]
{
\begin{aligned}
\begin{tikzpicture}[circuit ee IEC, set resistor graphic=var resistor IEC
      graphic,scale=.5]
\scalebox{1}{
		\node [style=none] (0) at (1, -0) {};
		\node [style=none] (1) at (0.125, -0) {};
		\node [style=none] (2) at (-1, 0.5) {};
		\node [style=none] (3) at (-1, -0.5) {};
		\node [style=none] (4) at (1, -0) {};
	\draw[line width = 1.5pt] (2) to [in=0, out=0,looseness = 4] (3);
}
      \end{tikzpicture}
\end{aligned}
}

\newcommand{\captwo}[1]
{
\begin{aligned}
\begin{tikzpicture}
[circuit ee IEC, set resistor graphic=var resistor IEC
      graphic,scale=.5]
\scalebox{1}{
		\node [style=none] (0) at (-.75, -0) {};
		\node [style=none] (1) at (-0.125, -0) {};
		\node [style=none] (2) at (1, 0.5) {};
		\node [style=none] (3) at (1, -0.5) {};
		\node [style=none] (4) at (-1, -0) {};
		\node [style=none] (5) at (-0.5, -.5) {};
		\node [style=none] (6) at (-0.5, .5) {};
	\draw[line width = 1.5pt] (2) to [in=180, out=180,looseness = 4] (3);
}
      \end{tikzpicture}
\end{aligned}
}

\begin{document}   

\begin{center}   
  {\bf A Compositional Framework for Passive Linear Networks \\}   
  \vspace{0.3cm}
  {\em John\ C.\ Baez \\}
  \vspace{0.3cm}
  {\small
 Department of Mathematics \\
    University of California \\
  Riverside CA, USA 92521 \\ and \\
 Centre for Quantum Technologies  \\
    National University of Singapore \\
    Singapore 117543  \\    } 
  \vspace{0.4cm}
  {\em Brendan Fong \\}
  \vspace{0.3cm}
  {\small Department of Mathematics  \\
    Massachusetts Institute of Technology  \\
   Cambridge MA, USA 02139  \\ }
  \vspace{0.3cm}   
  {\small email:  baez@math.ucr.edu, bfo@mit.edu\\} 
  \vspace{0.3cm}   
  {\small \today}
  \vspace{0.3cm}   
\end{center}   

\begin{abstract}
\noindent Passive linear networks are used in a wide variety of engineering applications, but
the best studied are electrical circuits made of resistors, inductors and capacitors.  
We describe a category where a morphism is a circuit of this sort with marked input and output terminals.  In this category, composition describes the process of attaching the outputs of one circuit to the inputs of another.   We construct a functor, dubbed the `black box functor', that takes a circuit, forgets its internal structure, and remembers only its external behavior.  Two circuits have the same external behavior if and only if they impose same relation between currents and potentials at their terminals.   The space of these currents and potentials naturally has the structure of a symplectic vector space, and the relation imposed by a circuit is a Lagrangian linear relation.  Thus, the black box functor goes from our category of circuits to a category with Lagrangian linear relations
as morphisms.   We prove that this functor is symmetric monoidal and indeed a hypergraph functor.   We assume the reader is familiar with category theory, but not with circuit theory or symplectic linear algebra.  
\end{abstract}

\tableofcontents


\section{Introduction}\label{sec:intro}
In the late 1940s, just as Feynman was developing his diagrams for processes in particle physics, Eilenberg and Mac Lane initiated their work on category theory.  Over the subsequent decades, and especially in the work of Joyal and Street in the 1980s \cite{JS1,JS2}, it became clear that these developments were profoundly linked: monoidal categories have a precise graphical representation in terms of string diagrams, and conversely monoidal categories provide an algebraic foundation for the intuitions behind Feynman diagrams.  The key insight is the use of categories where morphisms describe physical processes, rather than structure-preserving maps between mathematical objects \cite{AC,BaezStay,CP,Se}.   More recently, the same techniques have filtered into other applications.  This paper is part of a program of applying string diagrams to engineering, with the aim of giving diverse diagram languages a unified foundation based on category theory \cite{BCR,BE,BSZ,CF,Er,Fong16,FS18,KSW,Sp}. 

Indeed, even before physicists began using Feynman diagrams, various branches of engineering were using diagrams that in retrospect are closely related.   Foremost among these are the ubiquitous electrical circuit diagrams. Although less well-known, similar diagrams are used to describe networks consisting of mechanical, hydraulic, thermodynamic and chemical systems.   Further work, pioneered in particular by 
Forrester \cite{Fo} and Odum \cite{Od}, applies similar diagrammatic methods to biology, ecology, and economics.

As discussed in detail by Olsen \cite{Ol}, Paynter \cite{Pa} and others \cite{Brown,KRM}, there are mathematically precise analogies between these different systems.  In each case, the system's state is described by variables that come in pairs, with one variable in each pair playing the role of  `displacement' and the other playing the role of `momentum'.  In engineering, the time derivatives of these variables are sometimes called `flow' and `effort'.    In classical mechanics, this pairing of variables is well understood using
symplectic geometry.  Thus, any mathematical formulation of the diagrams used to
describe networks in engineering needs to take symplectic geometry as well as category
theory into account. 

\vskip 1em
\begin{small}
\begin{center}
\begin{tabular}{|c||c|c|c|c|}
\hline
& displacement  &  flow & momentum & effort \\
& $q$ & $\dot{q}$ & $p$ & $\dot{p}$ \\
\hline\hline
Electronics & charge & current & flux linkage & voltage\\
\hline
Mechanics (translation) & position & velocity & momentum & force\\
\hline
Mechanics (rotation) & angle & angular velocity & angular momentum & torque\\
\hline
Hydraulics & volume & flow & pressure momentum & pressure\\
\hline
Thermodynamics & entropy & entropy flow & temperature momentum & temperature \\
\hline
Chemistry & moles & molar flow & chemical momentum & chemical potential \\
\hline
\end{tabular}
\end{center}
\end{small}


Although we shall keep the broad applicability of network diagrams in the back of our minds, we couch our discussion in terms of electrical circuits, for the sake of familiarity. In this paper our goal is somewhat limited.  We only study circuits built from passive components: that is, those that do not produce energy.  Thus, we exclude batteries and current sources.  We only consider components that respond linearly to an applied voltage.   Thus, we exclude components such as nonlinear resistors or diodes.  Finally, we only consider components with one input and one output, so that a circuit can be described as a graph with edges labelled by components.  Thus, we also exclude transformers.  The most familiar components our framework covers are linear resistors, capacitors and inductors.

While we treat more general circuits in a companion paper \cite{BCR}, the class of circuits considered here has appealing mathematical properties, and is worthy of deep study.  Indeed, this class has been studied intensively for many decades by electrical engineers \cite{AV,Budak,Slepian}.  Even circuits made exclusively of resistors have inspired work by mathematicians of the caliber of Weyl \cite{Weyl} and Smale \cite{Smale}.  

Our work relies on this research.  All we are adding is an emphasis on symplectic geometry and an explicitly compositional framework, which clarifies the way a larger circuit can be built from smaller pieces.  This is where monoidal categories become important: the main operations for building circuits from pieces are composition and tensoring.
 
Our strategy is most easily illustrated for circuits made of linear resistors.  Such a resistor dissipates power, turning useful energy into heat at a rate determined by the voltage across the resistor.  However, a remarkable fact is that a circuit made of these resistors always acts to \emph{minimize} the power dissipated this way.  This `principle of minimum power' can be seen as the reason symplectic geometry becomes important in understanding circuits made of resistors, just as the principle of least action leads to the role of symplectic geometry in classical mechanics.  

Here is a circuit made of linear resistors:
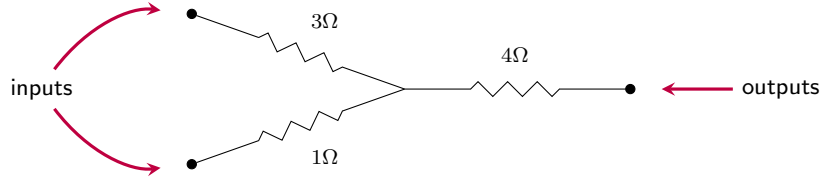
\[
\begin{tikzpicture}[circuit ee IEC, set resistor graphic=var resistor IEC graphic]
\node[contact] (I1) at (0,2) {};
\node[contact] (I2) at (0,0) {};
\node[contact] (O1) at (5.83,1) {};
\node(input) at (-2,1) {\small{\textsf{inputs}}};
\node(output) at (7.83,1) {\small{\textsf{outputs}}};
\draw (I1) 	to [resistor] node [label={[label distance=2pt]85:{$3\Omega$}}] {} (2.83,1);
\draw (I2)	to [resistor] node [label={[label distance=2pt]275:{$1\Omega$}}] {} (2.83,1)
				to [resistor] node [label={[label distance=3pt]90:{$4\Omega$}}] {} (O1);
\path[color=purple, very thick, shorten >=10pt, ->, >=stealth, bend left] (input) edge (I1);		\path[color=purple, very thick, shorten >=10pt, ->, >=stealth, bend right] (input) edge (I2);		
\path[color=purple, very thick, shorten >=10pt, ->, >=stealth] (output) edge (O1);
\end{tikzpicture}
\]
The wiggly lines are resistors, and their resistances are written beside them: for example,
$3\Omega$ means 3 ohms, an ohm being a unit of resistance.  To formalize this, define an
open circuit to consist of
\begin{itemize}
\item a set $N$ of nodes,
\item a set $E$ of edges, 
\item maps $s,t \maps E \to N$ sending each edge to its source and target node,
\item a map $r\maps E \to (0,\infty)$ specifying the resistance of the resistor 
labelling each edge, 
\item maps $i \maps X \to N$, $o \maps Y \to N$ specifying the
inputs and outputs of the circuit.
\end{itemize}

When we run electric current through such a circuit, each node $n \in N$ gets
a potential $\phi(n) \in \R$.  The voltage across an edge $e \in E$ is defined as the 
change in potential as we move from to the source of $e$ to its target, $\phi(t(e)) - 
\phi(s(e))$, and the power dissipated by the resistor on this edge equals
\[      
\frac{1}{r(e)}\big(\phi(t(e))-\phi(s(e))\big)^2. 
\]
The total power dissipated by the circuit is therefore twice
\[   
P(\phi) = \frac{1}{2}\sum_{e \in E} \frac{1}{r(e)}\big(\phi(t(e))-\phi(s(e))\big)^2.
\]
The factor of $\frac{1}{2}$ is convenient in some later calculations.  
Note that $P$ is a nonnegative quadratic form on the vector space $\R^N$.
However, not every nonnegative quadratic form on $\R^N$ arises in this way from some circuit of linear resistors with $N$ as its set of nodes.  The quadratic forms that do arise are called `Dirichlet forms'.  They have been extensively investigated \cite{Fukushima,MR,Sabot1997,Sabot2004}, and they play a major role in our work.

We write $\partial N = i(X) \cup o(Y)$ for the set of terminals: that is,
nodes corresponding to inputs and outputs.    The principle of minimum
power says that if we fix the potential at the terminals, the circuit will choose
the potential at other nodes to minimize the total power dissipated.   
An element $\psi$ of the vector space $\R^{\partial N}$ assigns a potential 
to each terminal.   Thus, if we fix $\psi$, the total power dissipated will be twice
\[
  Q(\psi) = \min_{\substack{ \phi \in \R^N \\ \phi\vert_{\partial N} = \psi}} \; P(\phi)  
\]
The function $Q \maps \R^{\partial N} \to \R$ is again a Dirichlet form.  We call it the `power functional' of the circuit.  

Now, suppose we are unable to see the internal workings of a circuit, and can only observe its external behavior: that is, the potentials at its terminals and the currents flowing into or out of these terminals.   As we shall see, this behavior is completely determined by the power functional $Q$.  The reason is that the current at any terminal can be obtained by differentiating $Q$ with respect to the potential at this terminal, and relations of this form are \emph{all} the relations that hold between potentials and currents at the terminals.   

The Laplace transform allows us to generalize these facts to circuits that also contain 
linear inductors and capacitors, simply by changing the field we work over and speaking 
of `impedance' rather than resistance.   Indeed, we can define open circuits for any 
field $\F$ with a well-behaved subset $\F^+$ of `positive elements': instead of
a map $r \maps E \to (0,\infty)$ labelling each edge with a resistance, such a circuit
has a map $Z \maps E \to \F^+$ labelling each edge with an impedance.   

We shall construct a category $\Circ$ where, roughly speaking, an object is a finite set, a
morphism $f \maps X \to Y$ is an open circuit with input set $X$ and output set $Y$, and
composition is given by identifying the outputs of one circuit with the inputs
of the next, and taking the resulting union of labelled graphs.  Each
such circuit gives rise to a Dirichlet form, and we prove that this Dirichlet form 
completely describes the externally observable behavior of the circuit.  

Given this, it would be nice to have
a category with Dirichlet forms as morphisms, and a functor from $\Circ$ to this category.  But
although there is a notion of composition for Dirichlet forms, we show that it lacks
identity morphisms, or equivalently, it lacks morphisms representing ideal wires
of zero impedance. To address this, we turn to Lagrangian subspaces of
symplectic vector spaces.  These generalize quadratic forms via the map
\[
  \Big(Q\maps \F^{\partial N} \to \F\Big) \longmapsto \Big(\mathrm{Graph}(dQ) =
  \{(\psi, dQ_\psi) \mid \psi \in \F^{\partial N} \} \subseteq \F^{\partial
  N} \oplus (\F^{\partial N})^\ast\Big)
\]
taking a quadratic form $Q$ on the vector space $\F^{\partial N}$
over a field $\F$ to the graph of its differential $dQ$. Here we think of the
symplectic vector space $\F^{\partial N} \oplus (\F^{\partial N})^\ast$ as the
state space of the circuit, and the subspace $\mathrm{Graph}(dQ)$ as the
subspace of attainable states, with $\psi \in \F^{\partial N}$ describing the
potentials at the terminals, and $dQ_\psi \in (\F^{\partial N})^\ast$ the
currents. 

The advantage of Lagrangian subspaces is that if we take a circuit made of
parallel resistors and let their resistances tend to zero, while the limit does not
give a Dirichlet form, it gives a well-defined Lagrangian subspace.
Indeed, there is a category $\Lag\Rel$ with finite sets as objects and
Lagrangian relations as morphisms from a finite set $X$ to a finite set $Y$:
that is, Lagrangian subspaces of $\overline{\vect{X}} \oplus \vect{Y} $, where
$\overline{V}$ is the symplectic vector space conjugate to $V$.   

To move from the Lagrangian subspace defined by the graph of the differential of
the power functional to a morphism in the category $\Lag\Rel$---that
is, to a Lagrangian relation---we must treat seriously the input and output
functions of the circuit. These express the circuit as built upon a cospan   
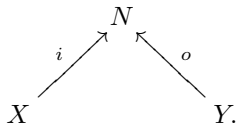
\[
  \xymatrix{
    & N \\
    X \ar[ur]^{i} && Y. \ar[ul]_o
  }
\]
Cospans model systems with two `ends', an input and output end, but without any
connotation of directionality: we might just as well exchange the role of the
inputs and outputs by taking the mirror image of the above diagram. The 
input and output functions simply mark the terminals
we may glue to the terminals of another circuit, and the pushout of cospans
gives formal precision to this gluing construction.

One upshot of this cospan framework is that we may consider circuits with elements
of $N$ that are both inputs and outputs, such as this one:
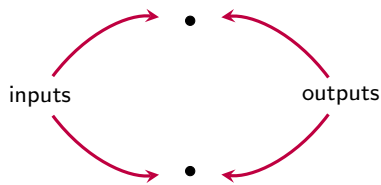
\[
  \begin{tikzpicture}[circuit ee IEC, set resistor graphic=var resistor iec graphic]
    \node[contact] (c1) at (0,2) {};
    \node[contact] (c2) at (0,0) {};
    \node(input) at (-2,1) {\small{\textsf{inputs}}};
    \node(output) at (2,1) {\small{\textsf{outputs}}};
    \path[color=purple, very thick, shorten >=10pt, ->, >=stealth, bend left]
    (input) edge (c1);		
    \path[color=purple, very thick, shorten >=10pt, ->, >=stealth, bend right]
    (input) edge (c2);	
    \path[color=purple, very thick, shorten >=10pt, ->, >=stealth, bend right]
    (output) edge (c1);
    \path[color=purple, very thick, shorten >=10pt, ->, >=stealth, bend left]
    (output) edge (c2);
  \end{tikzpicture}
\]
This corresponds to the identity morphism on the finite set with two elements.
Another is that some points may be considered an input or output multiple
times; we draw this:
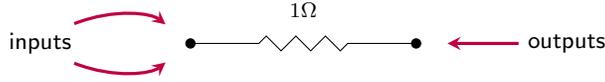
\[
  \begin{tikzpicture}[circuit ee IEC, set resistor graphic=var resistor IEC graphic]
    \node[contact] (I1) at (0,0) {};
    \node[contact] (O1) at (3,0) {};
    \node(input) at (-2,0) {\small{\textsf{inputs}}};
    \node(output) at (5,0) {\small{\textsf{outputs}}};
    \draw (I1) 	to [resistor] node [label={[label distance=3pt]90:{$1\Omega$}}]
    {} (O1);
    \path[color=purple, very thick, shorten >=10pt, ->, >=stealth, bend left] (input)
    edge (I1);		
    \path[color=purple, very thick, shorten >=10pt, ->,
    >=stealth, bend right] (input) edge (I1);		
    \path[color=purple, very thick, shorten >=10pt, ->, >=stealth] (output) edge (O1);
  \end{tikzpicture}
\]
This allows us to connect two distinct outputs to the above double input.

Given a set $X$ of inputs or outputs, we understand the electrical behavior on this set 
by considering the symplectic vector space $\vect{X}$, the direct sum of the space
$\F^X$ of potentials and the space ${(\F^X)}^\ast$ of currents at these points.
A Lagrangian relation specifies which states of the output space $\vect{Y}$ are
allowed for each state of the input space $\vect{X}$.
Turning the Lagrangian subspace $\mathrm{Graph}(dQ)$ of a circuit into this
information requires that we understand the `symplectification' 
\[  Sf\maps \vect{B} \to \vect{A} \] 
and `twisted symplectification'
\[  S^tf\maps \vect{B} \to \overline{\vect{A}}\]
of a function $f\maps A \to B$ between finite sets.  In particular we need to understand how these apply to the input and output functions with codomain restricted to $\partial N$; abusing notation, we also write these $i\maps X \to \partial N$ and $o\maps Y \to \partial N$.

The symplectification is a Lagrangian relation, and the catch
phrase is that it `copies voltages' and `splits currents'.  More precisely,
for any given potential-current pair $(\psi,\iota)$ in $\vect{B}$, its image
under $Sf$ comprises all elements of $(\psi', \iota') \in \vect{A}$ such that
the potential at $a \in A$ is equal to the potential at $f(a) \in B$, and such
that, for each fixed $b \in B$, collectively the currents at the $a \in
f^{-1}(b)$ sum to the current at $b$.  We use the symplectification $So$ of the
output function to relate the state on $\partial N$ to that on the
outputs $Y$. As our current framework is set up to report the current \emph{out}
of each node, to describe input currents we define the twisted symplectification
$S^tf\maps \vect{B} \to \overline{\vect{A}}$ almost identically to the above, 
except that we flip the sign of the currents $\iota' \in (\F^A)^\ast$.  We use the 
twisted symplectification $S^ti$ of the input function to relate the state on $\partial N$
to that on the inputs.

We shall see that the Lagrangian relation corresponding to a circuit is the set of all
potential--current pairs that are possible at the inputs and outputs of that circuit.   
For instance, consider a resistor of resistance $r$, with one end considered as an
input and the other as an output:
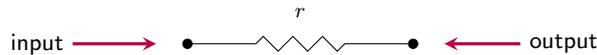
\[
  \begin{tikzpicture}[circuit ee IEC, set resistor graphic=var resistor IEC graphic]
    \node[contact] (I1) at (0,0) {};
    \node[contact] (O1) at (3,0) {};
    \node(input) at (-2,0) {\small{\textsf{input}}};
    \node(output) at (5,0) {\small{\textsf{output}}};
    \draw (I1) 	to [resistor] node [label={[label distance=3pt]90:{$r$}}] {} (O1);
    \path[color=purple, very thick, shorten >=10pt, ->, >=stealth] (input)
    edge (I1);
    \path[color=purple, very thick, shorten >=10pt, ->, >=stealth] (output) edge (O1);
  \end{tikzpicture}
\]
To obtain the corresponding Lagrangian relation, we must first specify domain and
codomain symplectic vector spaces. In this case, as the input and output sets
each consist of a single point, these vector spaces are both $\F \oplus \F^\ast$,
where the first summand is understood as the space of potentials, and the second
the space of currents.

Now, the resistor has power functional $Q\maps \F^2 \to \F$ given by
\[   Q(\psi_1,\psi_2) = \frac1{2r}(\psi_2-\psi_1)^2, \]
and the graph of the differential of $Q$ is
\[
  \mathrm{Graph}(dQ) = \big\{\big(\psi_1,\psi_2,
  \tfrac1r(\psi_1-\psi_2),\tfrac1r(\psi_2-\psi_1)\big) \,\big|\, \psi_1,\psi_2 \in
  \F\big\} \subseteq \F^2 \oplus (\F^2)^\ast.
\]
In this example the input and output functions $i,o$ are simply the identity
functions on a one element set, so the symplectification of the output function
is simply the identity linear transformation, and the twisted symplectification
of the input function is the isomorphism  between conjugate
symplectic vector spaces $\F\oplus\F^\ast \to \overline{\F\oplus\F^\ast}$ mapping $(\phi,i)$ to $(\phi,-i)$ This implies that the behavior associated to this
circuit is the Lagrangian relation
\[
  \big\{(\psi_1,i,\psi_2,i) \,\big|\, \psi_1,\psi_2 \in \F, i =
  \tfrac1r(\psi_2-\psi_1)\big\}\subseteq \overline{\F \oplus \F^\ast} \oplus \F
    \oplus \F^\ast.
\]
This is precisely the set of potential-current pairs that are allowed at the
input and output of a resistor of resistance $r$.  In particular, the relation
$i = \tfrac1r(\psi_2-\psi_1)$ is well known in electrical engineering: it is
`Ohm's law'.

A crucial fact is that the process of mapping a circuit to its corresponding
Lagrangian relation identifies distinct circuits.  For example, a single 2-ohm resistor:
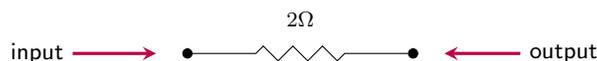
\[
  \begin{tikzpicture}[circuit ee IEC, set resistor graphic=var resistor IEC graphic]
    \node[contact] (I1) at (0,0) {};
    \node[contact] (O1) at (3,0) {};
    \node(input) at (-2,0) {\small{\textsf{input}}};
    \node(output) at (5,0) {\small{\textsf{output}}};
    \draw (I1) 	to [resistor] node [label={[label distance=3pt]90:{$2\Omega$}}] {} (O1);
    \path[color=purple, very thick, shorten >=10pt, ->, >=stealth] (input)
    edge (I1);
    \path[color=purple, very thick, shorten >=10pt, ->, >=stealth] (output) edge (O1);
  \end{tikzpicture}
\]
has the same Lagrangian relation as two 1-ohm resistors in series:
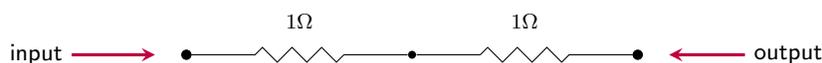
\[
  \begin{tikzpicture}[circuit ee IEC, set resistor graphic=var resistor IEC graphic]
    \node[contact] (I1) at (0,0) {};
    \node[circle, minimum width = 3pt, inner sep = 0pt, fill=black] (int) at
    (3,0) {};
    \node[contact] (O1) at (6,0) {};
    \node(input) at (-2,0) {\small{\textsf{input}}};
    \node(output) at (8,0) {\small{\textsf{output}}};
    \draw (I1) 	to [resistor] node [label={[label distance=3pt]90:{$1\Omega$}}] {} (int)
    to [resistor] node [label={[label distance=3pt]90:{$1\Omega$}}] {} (O1);
    \path[color=purple, very thick, shorten >=10pt, ->, >=stealth] (input)
    edge (I1);
    \path[color=purple, very thick, shorten >=10pt, ->, >=stealth] (output) edge (O1);
  \end{tikzpicture}
\]
The Lagrangian relation does not shed any light on the internal workings of a
circuit.  Thus, we call the process of computing this relation `black boxing':
it is like encasing the circuit in an opaque box, leaving only its terminals
accessible. Fortunately, the Lagrangian relation of a circuit is enough to
completely characterize its external behavior, including how it interacts when
connected with other circuits. 

Put more precisely, the black boxing process is \emph{functorial}: we can 
compute the black boxed version of a circuit made of parts by computing the
black boxed versions of the parts and then composing them.   In fact we shall 
prove that $\Circ$ and $\Lag\Rel$ are symmetric monoidal categories with
some extra structure, known as hypergraph categories (see Sec.~\ref{subsec:hypergraph}), 
and the black box functor preserves this structure:

\begin{theorem} \label{thm:main}
  There exists a hypergraph functor, the \define{black box functor}   
  \[ \blacksquare\maps \Circ \to \Lag\Rel, \]
  mapping a finite set $X$ to the symplectic vector space $\vect{X}$ it
  generates, and an open circuit $(N,E,s,t,r,i,o)$ to the Lagrangian
  relation 
  \[
    \bigcup_{v \in \mathrm{Graph}(dQ)} S^ti(v) \times So(v)
    \subseteq \overline{\F^X \oplus (\F^X)^\ast} \oplus \F^Y \oplus (\F^Y)^\ast,
  \]
  where $Q$ is the circuit's power functional.
\end{theorem}

The goal of this paper is to prove and explain this result. The proof itself is
more tricky than one might first expect, but our approach introduces various
concepts that are useful throughout the study of networks, such as `hypergraph categories',
`decorated cospans' and `corelations'.  These provide a general framework for
discussing open networked systems---and not only the passive linear systems
discussed here, but also others, such as Markov processes \cite{BFP} and chemical
reaction networks \cite{BP}.

\eject  

\subsection{Plan of the paper}
This paper is split into three parts, addressing in turn the questions:
\begin{enumerate}[I.]
  \item What do circuit diagrams mean?
  \item How do we interact with circuit diagrams?
  \item How is meaning preserved under these interactions?
\end{enumerate}

We begin Part \ref{part:circuits}, on the semantics of circuit diagrams, with
a discussion of circuits of linear resistors, developing the intuition for the
governing laws of passive linear circuits---Ohm's law, Kirchhoff's voltage law,
and Kirchhoff's current law---in a time-independent setting (Section
\ref{sec:resistors}). This allows us to develop the concept of
Dirichlet form as a representation of power consumption, and understand their
composition as minimizing power, an expression of the current law. In Section
\ref{sec:plcs}, the Laplace transform then allows us to generalize these ideas
to inductors and capacitors, speaking of impedance where we
formerly spoke of resistance, and generalizing Dirichlet forms from the field
$\R$ to the field $\R(s)$ of real rational functions.   In this setting the
principle of minimum power is replaced by a variational principle,
but the intuitions gained from circuits of resistors still remain useful. 

Part \ref{part:categories} introduces the syntax and semantics of open circuits.
In Section \ref{sec:cospans}, we develop machinery to construct categories of decorated
cospans. These are categories where the objects are finite sets and the morphisms
are cospans of finite sets equipped with some extra structure on
the apex.  Open circuits, as defined above, are an example of this
construction.   In Section \ref{sec:circsemantics} we define the category $\Circ$ whose
morphisms are open circuits.  We also construct a functor from $\Circ$ to the
category $\Lag\Cospan$, where a morphism is a cospan of finite sets with the apex
decorated by a Lagrangian relation.

Having introduced these prerequisites, we turn to black-boxing in Part \ref{part:blackbox}.
We start by introducing decorated corelation categories in Section \ref{sec:deccorel}.
In Section \ref{sec:blackbox} we show that $\Lag\Rel$ is isomorphic to a decorated corelation
category and use this to construct a functor from $\Lag\Cospan$ to $\Lag\Rel$.   Composing
this with the previous functor from $\Circ$ to $\Lag\Cospan$ we obtain 
the black box functor $\blacksquare \maps \Circ \to \Lag\Rel$.   Finally, in Section 
\ref{sec:main}, we use the tools we have developed prove our main result.

\subsection*{Acknowledgements}
We thank Jamie Vicary for useful conversations, an anonymous referee for a
careful reading and detailed comments, and Omar Camarena, Brandon Coya and
Bernhard Reinke for catching errors. BF would like to thank the Clarendon Fund;
Hertford College, Oxford; the Centre for Quantum Technologies, Singapore; and
USA AFOSR grants FA9550-14-1-0031 and FA9550-17-1-0058 for their support.  Much
of the material here appears in BF's Ph.D. thesis \cite{Fong16}.

\part{Passive Linear Circuits} \label{part:circuits}
In this part we begin by reviewing the properties of resistors, inductors,
capacitors, and circuits built from these.   Then we show that the behavior of
a circuit is determined by its power functional.    This is a quadratic form of 
a special sort called a `Dirichlet form'.  We define Dirichlet forms over any field
with a well-behaved set of positive elements, and in Thm.~\ref{thm:dirichlet_problem}
we establish the properties of Dirichlet forms that we need in the rest of our work.

\section{Circuits of linear resistors} \label{sec:resistors}
In order to let physical intuition lead the way, we begin by considering the
case of linear resistors. In this section we describe how to find the behavior 
of a circuit from its physical form, advocating in particular the perspective of
the principle of minimum power. This allows us to identify the external behavior
of a circuit with a so-called Dirichlet form representing the dependence of its
power consumption on potentials at its terminals.

\subsection{Circuits as labelled graphs}

The concept of an abstract open electrical circuit made of linear resistors is
well known in electrical engineering, but we shall need to formalize it with
more precision than usual.  The basic idea is that a circuit of linear resistors
is a graph whose edges are labelled by positive real numbers called
`resistances', and whose set of nodes is equipped with a subset of terminals. 
This unfolds as follows.

A (closed) circuit of resistors looks like this: 
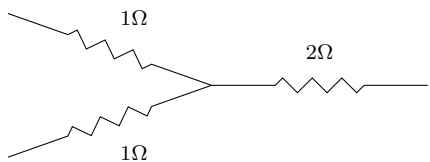
\[
\begin{tikzpicture}[circuit ee IEC, set resistor graphic=var resistor IEC graphic]
\node (I1) at (0,0) {};
\node (I2) at (0,2) {};
\node (O1) at (5.83,1) {};
\draw (I1) 	to [resistor] node [label={[label distance=2pt]275:{$1\Omega$}}] {} (2.83,1);
\draw (I2)	to [resistor] node [label={[label distance=2pt]85:{$1\Omega$}}] {} (2.83,1)
				to [resistor] node [label={[label distance=3pt]90:{$2\Omega$}}] {} (O1);
\end{tikzpicture}
\]
We can consider this a labelled graph, with each resistor an edge of the graph,
its resistance its label, and the nodes of the graph the points at which
resistors are connected. 

A circuit is `open' if it can be connected to other circuits. To do this we
first mark points at which connections can be made by denoting some nodes as
terminals:
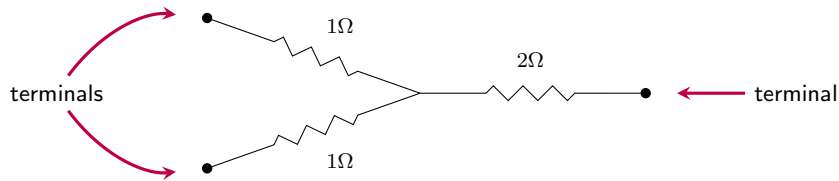
\[
\begin{tikzpicture}[circuit ee IEC, set resistor graphic=var resistor IEC graphic]
\node[contact] (I1) at (0,2) {};
\node[contact] (I2) at (0,0) {};
\node[contact] (O1) at (5.83,1) {};
\node(input) at (-2,1) {\small{\textsf{terminals}}};
\node(output) at (7.83,1) {\small{\textsf{terminal}}};
\draw (I1) 	to [resistor] node [label={[label distance=2pt]85:{$1\Omega$}}] {} (2.83,1);
\draw (I2)	to [resistor] node [label={[label distance=2pt]275:{$1\Omega$}}] {} (2.83,1)
				to [resistor] node [label={[label distance=3pt]90:{$2\Omega$}}] {} (O1);
\path[color=purple, very thick, shorten >=10pt, ->, >=stealth, bend left] (input) edge (I1);		\path[color=purple, very thick, shorten >=10pt, ->, >=stealth, bend right] (input) edge (I2);		
\path[color=purple, very thick, shorten >=10pt, ->, >=stealth] (output) edge (O1);
\end{tikzpicture}
\]

More formally, we define a \define{graph}
 to be a pair of functions $s,t\maps E \to N$ where $E$ and $N$ are finite sets.  We call elements of $E$ \define{edges} and elements of $N$ 
 \define{nodes}.  We say that the edge $e \in E$ has \define{source} $s(e)$ and
\define{target} $t(e)$, and also say that $e$ is an edge \define{from} $s(e)$
\define{to} $t(e)$.

To study circuits we need graphs with labelled edges:

\begin{definition}
Given a set $L$ of \define{labels}, an \define{$L$-graph} is a graph equipped with a function $r\maps E \to L$:
\[
\xymatrix{
L & E \ar@<2.5pt>[r]^{s} \ar@<-2.5pt>[r]_{t} \ar[l]_{r} & N.
}
\]
\end{definition}

For circuits made of resistors we take $L = (0,\infty)$, but we shall later
generalize this. In either case, a circuit will be an $L$-graph
with some extra structure:

\begin{definition} \label{def_circuit}
Given a set $L$, a \define{circuit with boundary over $L$} is an $L$-graph 
\[  \xymatrix{
L & E \ar@<2.5pt>[r]^{s} \ar@<-2.5pt>[r]_{t} \ar[l]_{r} & N} \]
together with a subset $\partial N \subseteq N$. We call $\partial N$ the 
\define{boundary} of the circuit, and elements of $\partial N$ \define{terminals}.
\end{definition}


We will later make use of the notion of connectedness in graphs. Recall that
given two nodes $v, w \in N$ of a graph, a \define{path from $v$ to $w$} is a
finite sequence of nodes $v = v_0, v_1, \dots , v_n = w$ and edges $e_1,
\dots , e_n$ such that for each $1 \le i \le n$, either $e_i$ is an edge from
$v_i$ to $v_{i+1}$, or an edge from $v_{i+1}$ to $v_i$. A subset $S$ of the
nodes of a graph is \define{connected} if, for each pair of nodes in $S$,
there is a path from one to the other. A \define{connected component} of a graph
is a maximal connected subset of its nodes.\footnote{In the theory of
directed graphs the qualifier `weakly' is commonly used before the word
`connected' in these two definitions, in distinction from a stronger notion of
connectedness requiring paths to respect edge directions. As we never consider
any other sort of connectedness, we omit this qualifier.}

In the rest of this section we take $L = \R^+ = (0,\infty)$ and fix a circuit over 
this label set.  The edges of this circuit should be thought of as `wires'.  The label 
$r(e) \in (0,\infty)$ stands for the \define{resistance} of the resistor on the wire $e$.   
There will also be a voltage and current on each wire.  In this section, these will
be specified by functions $V \in \R^E$ and $I \in \R^E$.  Here, as customary in
engineering, we use $I$ for `intensity of current', following Amp\`ere.  

\subsection{Ohm's law, Kirchhoff's laws, and the principle of minimum power}

In 1827, Georg Ohm published a book which included a linear relation between 
the voltage and current for circuits made of resistors \cite{O}.  
We thus say that \define{Ohm's law} holds if for each edge $e \in
E$ the voltage and current obey
\[ 
V(e) = r(e) I(e)  \label{ohm}  
\]
where $r(e)$ is the resistance of that edge.
Kirchhoff's laws date to Gustav Kirchhoff in 1845, generalizing Ohm's work. 
We say
\define{Kirchhoff's voltage law} holds if there exists $\phi \in \R^N$ such that
\[
V(e) = \phi(t(e)) - \phi(s(e)).
\]
We call the function $\phi$ a \define{potential}, and think of it as assigning
an electrical potential to each node in the circuit. The voltage then arises as
the differences in potentials between adjacent nodes. 

A \define{boundary potential} is a function in $\R^{\partial N}$, thought
of as specifying potentials on the terminals of a circuit. As our circuits are
`open', with the terminals serving as points of interaction with the external world,
we think of these potentials as variables that are free for us to choose.
We say \define{Kirchhoff's current law} holds if for all nonterminal nodes $n
\in N\setminus \partial N$ we have
\[ 
\sum_{s(e) = n} I(e) = \sum_{t(e) = n} I(e).  \label{kcl}  
\]  
This is an expression of conservation of charge within the circuit; it says that
the total current flowing in or out of any nonterminal node is zero. Even when
Kirchhoff's current law is obeyed, terminals need not be sites of zero net
current; we call the function $\iota \in \R^{\partial N}$ that takes a terminal
to the difference between the outward and inward flowing currents,
\begin{align*}
\iota \maps \partial N &\longrightarrow \R \\
n &\longmapsto \sum_{t(e) = n} I(e) -\sum_{s(e) = n} I(e),
\end{align*}
the \define{boundary current} for $I$.

In Section~\ref{ssec:dirichlet_problem} we show that the above three
principles---Ohm's law, Kirchhoff's voltage law, and Kirchhoff's current
law---imply that choosing a boundary potential determines unique voltage and
current functions on that circuit. 
The so-called `principle of minimum power' gives insight into how this occurs,
by describing the way boundary potentials determine potentials. From this,
Kirchhoff's voltage law then gives rise to a voltage function on the edges, and
Ohm's law gives us a current function too. We shall show, in fact, that a
potential satisfies the principle of minimum power for a given boundary
potential if and only if this current obeys Kirchhoff's current law.

What is this power that we minimize? Power is the rate at which the circuit
dissipates energy. A circuit with current $I$ and voltage $V$ dissipates energy
at a rate equal to
\[
 \sum_{e \in E} I(e)V(e).
\]  
Ohm's law allows us to rewrite $I(e)$ as $V(e)/r(e)$, while Kirchhoff's voltage law gives
us a potential $\phi$ such that $V(e)$ can be written as
$\phi(t(e))-\phi(s(e))$, so for a circuit obeying these two laws the power can
also be expressed in terms of this potential. We thus arrive at a functional
mapping each potential $\phi$ to the power dissipated by the circuit when Ohm's law
and Kirchhoff's voltage law are obeyed for $\phi$. 

\begin{definition}
The \define{extended power functional} $P\maps \R^N \to \R$ of a circuit is
defined by
\[
P(\phi) =\frac{1}{2} \sum_{e \in E} \frac{1}{r(e)}\big(\phi(t(e))-\phi(s(e))\big)^2.
\]
\end{definition}

\noindent
The factor of $\frac{1}{2}$ is inserted to cancel the factor of 2 that appears
when we differentiate this expression.  We call $P$ the \emph{extended} power
functional as it is defined even on potentials that are not
compatible with the three governing laws of electric circuits. We shall later
restrict the domain of this functional so that it is defined precisely on those
potentials that \emph{are} compatible with the governing laws. Note that $P$
does not depend on the directions chosen for the edges of the circuit.

This expression lets us formulate the `principle of minimum power', which gives
us information about the potential $\phi$ given its restriction to the boundary 
$\partial N \subseteq N$. Call a potential $\phi \in \R^N$ an \define{extension} of a
boundary potential $\psi \in \R^{\partial N}$ if $\phi$ is equal to $\psi$ when
restricted to $\R^{\partial N}$---that is, if $\phi|_{\partial N} = \psi$. 

\begin{definition}
We say a potential $\phi \in \R^{N}$ \define{obeys the principle of minimum
power} for a boundary potential $\psi \in \R^{\partial N}$ if $\phi$ minimizes
the extended power functional $P$ subject to the constraint that  $\phi$ is an
extension of $\psi$. 
\end{definition}

As promised, in the presence of Ohm's law and Kirchhoff's voltage law, the
principle of minimum power is equivalent to Kirchhoff's current law.

\begin{proposition} \label{minimum_power_implies_kirchhoff_current}
Let $\phi$ be a potential extending some boundary potential $\psi$. Then $\phi$
obeys the principle of minimum power for $\psi$ if and only if the 
current 
\[  I(e) = \frac1{r(e)}(\phi(t(e))-\phi(s(e))) \] 
obeys Kirchhoff's current law.
\end{proposition}

\begin{proof}
Fixing the potentials at the terminals to be those given by the boundary
potential $\psi$, the power is a nonnegative quadratic function of the
potentials at the nonterminals. This implies that an extension $\phi$ of $\psi$
minimizes $P$ precisely when 
\[ \left. \frac{\partial P(\varphi)}{\partial \varphi(n)}\right|_{\varphi = \phi} = 0 \]
for all nonterminals $n \in N \setminus \partial N$. Note that the
partial derivative of the power with respect to the potential at $n$ is given by 
\begin{align*}
  \frac{\partial P}{\partial \varphi(n)}\bigg|_{\varphi = \phi} 
  &= \sum_{t(e) = n} \frac1{r(e)}\big(\phi(t(e))-\phi(s(e))\big) - \sum_{s(e) =
  n} \frac1{r(e)}\big(\phi(t(e))-\phi(s(e))\big) \\
  &= \sum_{t(e) = n} I(e) - \sum_{s(e) = n} I(e).
\end{align*}
Thus $\phi$ obeys the principle of minimum power for $\psi$ if and only if
\[ \sum_{s(e) = n} I(e) = \sum_{t(e) = n} I(e)\] 
for all $n \in N \setminus \partial N$, and so if and only if Kirchhoff's current law holds.
\end{proof}

\subsection{A Dirichlet problem} \label{ssec:dirichlet_problem}

We are studying circuits as objects that define
relationships between boundary potentials and boundary currents. This
relationship is defined by the stipulation that voltage--current pairs on a
circuit must obey Ohm's law and Kirchhoff's laws---or equivalently, Ohm's law,
Kirchhoff's voltage law, and the principle of minimum power. In this subsection
we show these laws imply that for each boundary potential $\psi$ there exists a
potential $\phi$ extending $\psi$, unique up to what may be interpreted as a
choice of reference potential on each connected component of the circuit.  From
this potential $\phi$ we can then compute the unique voltage, current, and
boundary current functions compatible with the given boundary potential.

Fix a circuit with extended power functional $P\maps \R^N \to \R$. Let
$\Delta\maps \R^N \to \R^N$ be the operator that maps a potential 
$\phi \in \R^N$ to the function from $N$ to $\R$ given by
\[
n \longmapsto \frac{\partial P}{\partial \varphi(n)}\bigg|_{\varphi = \phi} \;.
\]
As we have seen, this function takes potentials to the pointwise currents that
they induce. We have also seen, in Prop.~\ref{minimum_power_implies_kirchhoff_current}, 
that a potential $\phi$ is compatible with the governing laws of circuits if and only if
\begin{equation}
\Delta \phi \big|_{\R^{N\setminus \partial N}} = 0 .\label{dirichlet}
\end{equation}
The operator $\Delta$ acts as a discrete analogue of the Laplacian, so we call this operator the \define{Laplacian}, and say that  equation \eqref{dirichlet} is a version of Laplace's equation. We then say that the problem of finding an extension $\phi$ of some fixed boundary potential $\psi$ that solves this Laplace's equation---or, equivalently, the problem of finding a $\phi$ that obeys the principle of minimum power for $\psi$---is a discrete version of the \define{Dirichlet problem}. 

As we shall see, this version of the Dirichlet problem always has a solution.  However, the solution is not necessarily unique.  If we take a solution $\phi$ and some $\alpha \in \R^N$ that is constant on each connected component and vanishes on the boundary $\partial N$, it is clear that $\phi+\alpha$ is still an extension of $\psi$ and that 
\[
\left.\frac{\partial P(\varphi)}{\partial \varphi(n)}\right|_{\varphi = \phi} = 
\left.\frac{\partial P(\varphi)}{\partial \varphi(n)}\right|_{\varphi = \phi + \alpha},
\] 
so $\phi + \alpha$ is another solution. We say that a connected component of a circuit \define{touches the boundary} if it contains a node in $\partial N$. Note that $\alpha$ as above must vanish on all connected components touching the boundary.

With these preliminaries in hand, we can solve the Dirichlet problem.
\begin{proposition} \label{prop:dirichlet_problem}
Let $\psi \in \R^{\partial N}$ be a boundary potential. Then:
\begin{enumerate}[(i)]
\item There exists a potential $\phi \in \R^N$ obeying the principle of minimum power for $\psi$.
\item If $\phi$ and $\phi'$ both obey the principle of minimum power for
$\psi$, then $P(\phi) = P(\phi')$.  
\item If $\phi$ and $\phi'$ both obey the principle of minimum power for
$\psi$, then $\phi-\phi'$ is constant on every connected component of the graph 
$ \xymatrix{ E \ar@<2.5pt>[r]^{s} \ar@<-2.5pt>[r]_{t} & N}$.
\item There exists a unique potential $f(\psi) \in \R^N$ 
that obeys the principle of minimum power for $\psi$ and vanishes on every 
connected component of  
$ \xymatrix{ E \ar@<2.5pt>[r]^{s} \ar@<-2.5pt>[r]_{t} & N}$
not touching the boundary.
\item The potential $f(\psi)$ depends linearly on $\psi$.
\end{enumerate}
\end{proposition}

\begin{proof} We prove this more generally in Thm.~\ref{thm:dirichlet_problem}.  
\end{proof}

\subsection{The power functional}

We have seen that boundary potentials determine, essentially uniquely, the value
of all the electric properties across the entire circuit. But from the
compositional perspective, this internal structure is irrelevant: we can only
access the circuit at its terminals, and hence only need concern ourselves with
what can be witnessed there: the relationship between boundary potentials and
boundary currents. In this section we 
state the precise way in which boundary currents depend on boundary potentials.
In particular, we shall show that the relationship is completely captured by the
circuit's `power functional': the function taking any boundary potential to the 
minimum possible power used by any extension of that boundary potential. 
Furthermore, circuits with different power functionals have different relations between
boundary potentials and boundary currents.  So, two circuits are equivalent, as
far as what can be observed `from outside', if they have the same
power functional.

\begin{definition}
The \define{power functional} $Q \maps \R^{\partial N} \to \R$ of a circuit with extended power functional $P$ is given by
\[
 Q(\psi) = \min_{\phi|_{\R^{\partial N}} = \psi } P(\phi).
\]
\end{definition}

Prop.~\ref{prop:dirichlet_problem}(i) shows that the minimum above exists, so the power
functional is well defined.  Prop.~\ref{prop:dirichlet_problem}(iv) implies that 
$Q(\psi) = P(f(\psi))$, where $f$ is the map sending $\psi$ to the unique potential obeying
the principle of minimum power for $\psi$.   Prop.~\ref{prop:dirichlet_problem}(v)
says that $f$ is linear; since $P$ is a nonnegative quadratic form, it follows that $Q$ is as well.   
We call $Q$ the `power functional' because $Q(\psi)$ equals 
$\frac{1}{2}$ times the power dissipated by the circuit when the boundary voltage is $\psi$. 

Since $Q$ is a smooth real-valued function on $\R^{\partial N}$, its
differential $d Q$ at any given point $\psi \in \R^{\partial N}$ defines an
element of the dual space $(\R^{\partial N})^\ast$, which we denote by $d
Q_\psi$.  In fact, this element is equal to the boundary current $\iota$
corresponding to the boundary potential $\psi$:

\begin{proposition} \label{boundary_current_determines_boundary_voltage}
Suppose $\psi \in \R^{\partial N}$.  Suppose $\phi$ is any extension of $\psi$ minimizing the power. Then $dQ_\psi \in (\R^{\partial N})^\ast \cong \R^{\partial N}$ gives the boundary current of the current induced by the potential $\phi$.
\end{proposition}

\begin{proof}
Note first that while there may be several choices of $\phi$ minimizing the
power subject to the constraint that $\phi|_{\R^{\partial N}} = \psi$,
Prop.~\ref{prop:dirichlet_problem} says that there is a unique choice $\phi = f(\psi)$ 
vanishing on all components not toucing the boundary of $\Gamma$, and that
\[
f\maps \R^{\partial N} \longrightarrow \R^N
\]
is linear.   We thus have
\[
Q(\psi) = P(f(\psi)).
\]
Write $\overline\iota\maps N \to \R$ for the extension of $\iota\maps \partial N \to \R$ to $N$ taking value $0$ on $N \setminus \partial N$. Given any $\psi' \in \R^{\partial N}$, we thus have
\begin{align*}
dQ_\psi(\psi') &= \frac{d}{d\lambda}Q(\psi +\lambda\psi') \bigg|_{\lambda=0} \\
&= \frac{d}{d\lambda}P(f(\psi+\lambda\psi'))\bigg|_{\lambda=0} \\
&= \frac{1}{2} \frac{d}{d\lambda}\sum_{e \in E} \frac1{r(e)}\bigg(f(\psi+\lambda\psi')(t(e))-f(\psi+\lambda\psi')(s(e))\bigg)^2 \bigg|_{\lambda=0} \\
&= \sum_{e \in E} \frac1{r(e)}\big((f(\psi)(t(e))- f(\psi)(s(e))\big)\, \big(f(\psi')(t(e))- f(\psi')(s(e))\big) \\
&= \sum_{e \in E} I(e) \, \big(f(\psi')(t(e))- f(\psi')(s(e))\big) \\
&= \sum_{n \in N}\left(\sum_{t(e) = n} I(e) - \sum_{s(e) = n} I(e)\right) \, f(\psi')(n) \\
&= \sum_{n \in N}\overline \iota(n) \, f(\psi')(n) \\
&= \sum_{n \in \partial N}\iota(n) \, \psi'(n).
\end{align*}
This shows that $dQ_\psi \in (\R^{\partial N})^\ast$ maps to $\iota \in
\R^{\partial N}$ under the canonical isomorphism $(\R^{\partial N})^\ast \cong
\R^{\partial N}$, as claimed.  Note that this calculation explains why we
inserted a factor of $\frac{1}{2}$ in the definition of $P$: it cancels the
factor of $2$ obtained from differentiating a square. 
\end{proof}

Note this only depends on $Q$, which makes no mention of the potentials at
nonterminals. This is fundamental: the way power depends on boundary potentials
completely characterizes the way boundary currents depend on boundary
potentials. In particular, in Part \ref{part:blackbox} we shall see that this
allows us to define a composition rule for behaviors of circuits.

To demonstrate these notions, we give a basic example of equivalent circuits.

\begin{example}[Resistors in series] \label{resistors_in_series}
Resistors are said to be placed in \define{series} if they are placed end to end or, more
precisely, if they form a path with no self-intersections. It is well known that
resistors in series are equivalent to a single resistor with resistance equal to
the sum of their resistances. To prove this, consider the following circuit
comprising two resistors in series, with input $A$ and output $C$:
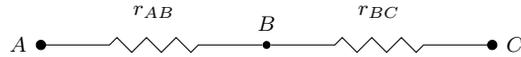
\[
  \begin{tikzpicture}[circuit ee IEC, set resistor graphic=var resistor IEC graphic]
    \node[contact] (I1) at (0,0) [label=left:$A$] {};
    \node[circle, minimum width = 3pt, inner sep = 0pt, fill=black] (int) at (3,0) [label=above:$B$] {};
    \node[contact] (O1) at (6,0) [label=right:$C$] {};
    \draw (I1) 	to [resistor] node [label={[label distance=3pt]90:{$r_{AB}$}}] {} (int)
    to [resistor] node [label={[label distance=3pt]90:{$r_{BC}$}}] {} (O1);
  \end{tikzpicture}
\]
Now, the extended power functional $P\maps \R^{\{A,B,C\}} \to \R$ for this circuit is
\[
P(\phi) = \frac12\left(\frac1{r_{AB}}\big(\phi(A)-\phi(B)\big)^2 +
\frac1{r_{BC}}\big(\phi(B)-\phi(C)\big)^2\right),
\]
while the power functional $Q\maps \R^{\{A,C\}} \to \R$ is given by minimization
over values of $\phi(B) = x$:
\[
Q(\psi) = \min_{x \in \R} \frac12 \left(\frac1{r_{AB}}\big(\psi(A)-x\big)^2 + \frac1{r_{BC}}\big(x-\psi(C)\big)^2 \right). 
\]
Differentiating with respect to $x$, we see that this minimum occurs when
\[
\frac1{r_{AB}}\big(x-\psi(A)\big) + \frac1{r_{BC}}\big(x-\psi(C)\big) = 0,
\]
and hence when $x$ is the $r$-weighted average of $\psi(A)$ and $\psi(C)$:
\[
x = \frac{r_{BC}\psi(A) + r_{AB}\psi(C)}{r_{BC}+ r_{AB}}.
\]
Substituting this value for $x$ into the expression for $Q$ above and simplifying gives
\[
Q(\psi) = \frac12\cdot\frac1{r_{AB}+r_{BC}}\big(\psi(A)-\psi(C)\big)^2. 
\]
This is also the power functional of the circuit
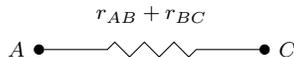
\[
\begin{tikzpicture}[circuit ee IEC, set resistor graphic=var resistor IEC graphic]
\node[contact] (I1) at (0,0) [label=left:$A$] {};
\node[contact] (O1) at (3,0) [label=right:$C$] {};
\draw (I1) 	to [resistor] node [label={[label distance=3pt]90:{$r_{AB}+r_{BC}$}}] {} (O1);
\end{tikzpicture}
\]
and so these two circuits are equivalent.
\end{example}

\subsection{Dirichlet forms}

In the previous subsection we claimed that power functionals are quadratic forms
on the boundary of the circuit whose behavior they represent. They comprise, in
fact, precisely those quadratic forms known as Dirichlet forms \cite{Fukushima,MR, Sabot1997,Sabot2004}.  Dirichlet forms are usually defined over $\R$, but we shall find 
it useful to work over any field with a notion of positive elements.

\begin{definition} \label{def:positive}
  We define a \define{field with positive elements} $(\F,\F^+)$ to be a field
  $\F$ equipped with a subset $\F^+$ that contains $c^2$ for every nonzero $c
  \in \F$, and is closed under addition, multiplication, and division.
\end{definition}

For example, $\R^+=(0,\infty)$ is a set of positive elements for $\R$.   Since $\F^+$ is closed under squaring we must have $1 \in \F^+$.   Since $\F^+$ is closed 
under division we must have $0 \notin \F^+$.  Since $\F^+$ is closed under addition
it follows that $-1 \notin \F^+$, and any field with positive elements must have characteristic 
zero.  We say $c$ is \define{nonnegative}, and write $c \ge 0$, if $c \in \F^+ \cup \{0\}$.   

\begin{definition} Given a finite set $S$ and a field with positive elements $(\F,\F^+)$, 
a \define{Dirichlet form} over $(\F,\F^+)$ on $S$ is a quadratic form $Q\maps \F^S \to \F$ 
given by the formula 
  \[ 
    Q(\psi) = \sum_{i,j \in S} c_{ij} (\psi_i - \psi_j)^2,
  \]
for some choice of nonnegative elements $c_{i j} \in \F$, where we have written
$\psi_i=\psi(i) \in \F$.   
\end{definition}

Every Dirichlet form is nonnegative: $Q(\psi)
\ge 0$ for all $\psi \in \F^S$.  However, not every nonnegative quadratic
form is a Dirichlet form.  For example, taking $S = \{1,2\}$, the quadratic
form $Q(\psi) = (\psi_1 + \psi_2)^2$ is nonnegative but not a Dirichlet form.    

Real Dirichlet forms are precisely the power functionals of circuits:

\begin{proposition} \label{prop.power_dirichlet}
  A function $Q \maps \R^{\partial N} \to \R$ is the power functional for some 
  circuit if and only if it is a Dirichlet form over $(\R,\R^+)$. 
\end{proposition}

\begin{proof}  This is an expression of the `star-mesh transform', a well-known fact of
electrical engineering stating that every circuit of linear resistors is
equivalent to some complete graph of resistors between its terminals. For more
details see \cite{vLO}.  \end{proof}

To summarize, our work so far has shown the existence of a surjective function
\[
  \bigg\{\begin{array}{c} \mbox{circuits of linear resistors} \\ \mbox{ with
    boundary $\partial N$} \end{array} \bigg\} \longrightarrow \bigg\{
    \mbox{Dirichlet forms on $\partial N$}\bigg\}
\]
mapping two circuits to the same Dirichlet form if and only if they have the
same external behavior.  In the next section we discuss how these ideas extend
to circuits comprising inductors and capacitors too.

\section{Passive linear circuits} \label{sec:plcs}
The intuition gleaned from the study of resistors carries over to inductors and
capacitors, and provides a framework for studying what are known as passive
linear circuits. To understand inductors and capacitors in this way, however, we
must introduce a notion of time dependence and subsequently the Laplace
transform, which allows us to work in the so-called frequency domain. Here, like
resistors, inductors and capacitors simply impose a relationship of
proportionality between the voltages and currents that run across them. The
constant of proportionality is known as the impedance of the component.

As for resistors, the interconnection of such components may be understood, at
least formally, as a minimization of some quantity, and we may represent the
behaviors of this class of circuits with a more general idea of Dirichlet form.
We conclude this section by noting an obstruction to building a composition rule
for Dirichlet forms, motivating our work in Part \ref{part:categories}.

\subsection{Inductors and capacitors}
\label{ssec:freqdomain}

In broadening the class of electrical circuit components under examination, we
find ourselves dealing with components whose behaviors depend on the rates of
change of current and voltage with respect to time. We thus now consider
time-varying voltages $v \maps [0,\infty) \to \R$ and currents $i \maps
[0,\infty) \to \R$, where $t \in [0,\infty)$ is a real variable representing
time. For mathematical reasons, we
restrict these voltages and currents to only those with (i) zero initial
conditions (that is, $f(0) = 0$) and (ii) Laplace transform lying in the field
\[
  \R(s) = \left\{ Z(s) = \displaystyle{\frac{P(s)}{Q(s)}} \,\Big\vert\, P, Q \mbox{
  polynomials over $\R$ in $s$}, \, Q \ne 0 \right\}
\]
of real rational functions of one variable.  While it is possible that
physical voltages and currents might vary with time in a more general way, we
restrict to these cases as the rational functions are, crucially, well behaved
enough to form a field, and yet still general enough to provide arbitrarily
close approximations to currents and voltages found in standard applications.

An inductor is a two-terminal circuit across which the voltage is
proportional to the rate of change of the current. By convention we draw this as
follows, with the inductance $L$ the constant of proportionality:
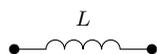
\[
  \begin{tikzpicture}[circuit ee IEC]
    \node[contact] (I1) at (0,0) {};
    \node[contact] (I2) at (1.83,0) {};
    \draw (I1) 	to [inductor] node [label={[label distance=2pt]{$L$}}]
    {} (I2);
  \end{tikzpicture}
\]
Writing $v_L(t)$ and $i_L(t)$ for the voltage and current over time $t$ across
this component respectively, and using a dot to denote the derivative with
respect to time $t$, we thus have the relationship 
\[
  v_L(t) = L\, \dot{i}_L(t).
\]
Switching the roles of current and voltage, a capacitor is a two-terminal
circuit across which the current is proportional to the rate of change of the
voltage. We draw this as follows, with the capacitance $C$ the constant of
proportionality:
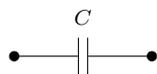
\[
  \begin{tikzpicture}[circuit ee IEC]
    \node[contact] (I1) at (0,0) {};
    \node[contact] (I2) at (1.83,0) {};
    \draw (I1) 	to [capacitor] node [label={[label distance=5pt]{$C$}}]
    {} (I2);
  \end{tikzpicture}
\]
Writing $v_C(t)$, $i_C(t)$ for the voltage and current across the capacitor,
this gives the equation
\[
  i_C(t) = C\, \dot{v}_C(t).
\]
We assume here that inductances $L$ and capacitances $C$ are positive real numbers.

Although inductors and capacitors impose a linear relationship if we involve the
derivatives of current and voltage, to mimic the above work on resistors we wish
to have a constant of proportionality between functions representing the current
and voltage themselves. Various integral transforms perform just this role; electrical
engineers typically use the Laplace transform. This lets us write a function of time $t$ instead as a function of frequencies $s$, and in doing so turns differentiation with respect to $t$ into multiplication by $s$, and integration with respect to $t$ into
division by $s$.  

In detail, given a function $f(t)\maps [0, \infty) \to \R$, we define the
\define{Laplace transform} of $f$
\[
  \mathfrak{L}\{f\}(s) = \int_{0}^\infty f(t) e^{-st} dt.
\]
We also use the notation $\mathfrak{L}\{f\}(s) = F(s)$, denoting the Laplace
transform of a function in upper case, and refer to the Laplace transforms as
lying in the \define{frequency domain} or \define{$s$-domain}. For us, the three
crucial properties of the Laplace transform are then: 
\begin{enumerate}[(i)]
  \item linearity: $\mathfrak{L}\{af+bg\}(s) = aF(s)+bG(s)$ for $a,b\in \R$;
  \item differentiation: $\mathfrak{L}\{\dot{f}\}(s) = s F(s) - f(0)$;
  \item integration: if $g(t) = \int_0^t f(\tau)d\tau$ then 
 $G(s) = \frac{1}{s} F(s)$.
\end{enumerate}
Writing $V(s)$ and $I(s)$ for the Laplace transform of the voltage $v(t)$ and
current $i(t)$ across a component respectively, and recalling that by assumption
$v(t) = i(t) = 0$ for $t \le 0$, the $s$-domain behaviors of components become,
for a resistor of resistance $R$:
\[
  V(s) = RI(s),
\]
for an inductor of inductance $L$:
\[
  V(s) = sLI(s),
\]
and for a capacitor of capacitance $C$:
\[
  V(s) = \frac1{sC} I(s). 
\]

Note that for each component the voltage equals the current times a rational function of
the real variable $s$, called the \define{impedance} and in general denoted by $Z$.
Note also that in each case the impedance lies in the set
\[         
 \R(s)^+ = \{ Z \in \R(s) : Z \ne 0 \textrm{ and } s > 0 \implies  Z(s) \ge 0 \} . 
\]
It is easy to check that $\R(s)^+$ forms a set 
of positive elements for the field $\R(s)$ according to Definition \ref{def:positive}.
We warn the reader that this concept of positivity differs from the one more commonly
used in circuit theory \cite{Brune}, which gives a set that is not closed under multiplication.

The examples of resistances in $\R$ and impedances in $\R(s)$ 
motivate our algebraic approach to passive linear circuits.  In what follows, we fix an
arbitrary field with positive elements $(\F,\F^+)$, and consider circuits
made from components obeying this generalization of Ohm's law:
\[
  V=ZI
\]
where $I \in \F$ is the current, $V \in \F$ is the voltage, and $Z \in
\F^+$ is the impedance of the component. 

\begin{definition} \label{def_circuit_2}
A \define{passive linear circuit} is a circuit with over $\F^+$: that is, a diagram
\[  \xymatrix{
\F^+ & E \ar@<2.5pt>[r]^{s} \ar@<-2.5pt>[r]_{t} \ar[l]_{Z} & N.} \]
We call $Z(e)$ the \define{impedance} of the edge $e \in E$, and call  
$ \xymatrix{ E \ar@<2.5pt>[r]^{s} \ar@<-2.5pt>[r]_{t} & N}$
the circuit's \define{underlying graph}.  We often abbreviate a passive linear
circuit as $(N,E,s,t,Z)$.   
\end{definition}

\begin{definition} 
A \define{passive linear circuit with boundary} is a passive linear circuit 
together with a subset $\partial N \subseteq N$ called the \define{boundary}.
\end{definition}

\subsection{The power functional as a Dirichlet form} \label{sec:generalized}

Generalizing from circuits of linear resistors to passive linear circuits is mainly a matter of formally replacing resistances by impedances.   However, we need a purely algebraic formulation of the principle of minimum power.   In what follows we fix a field with positive elements $(\F,\F^+)$ and a passive linear circuit  $(N,E,s,t,Z)$.   We define the \define{extended power functional} 
$P\maps \F^N \to \F$ by 
\[
  P(\phi) = \frac{1}{2} \sum_{e \in E} \frac1{Z(e)}\big(\phi(t(e))-\phi(s(e))\big)^2.
\]
Note that $2 \in \F^+$, so dividing by 2 is permitted, and $P$ is a Dirichlet
form.  

Although it is not clear what it means to minimize over the field $\F$, we
can use formal derivatives to formulate a version of the principle of minimum
power.   This is actually a `variational principle', saying the derivative
of the power functional vanishes with respect to certain variations in the
potential.

In detail, the extended power functional can be reinterpreted as giving an element 
$P(\varphi)$ of the polynomial ring $\F[\{\varphi(n)\}_{n \in N}]$ generated by formal
variables $\varphi(n)$ corresponding to potentials at the nodes $n \in N$. We
may thus take formal derivatives of the extended power functional with respect
to the $\varphi(n)$.    For any set $R \subseteq N$, we say $\phi \in \F^N$ is
an \define{extension} of $\psi \in \F^R$ if $\phi$ restricted to $R$ equals
$\psi$.  We say such an extension $\phi$ \define{obeys the principle of minimum 
power} for $\psi$ if 
\[
    \frac{\partial P}{\partial \varphi(n)}\bigg\vert_{\varphi = \phi} = 0
  \]
for all $n \in N \setminus R$.   We are especially interested in the case where we
have a passive linear circuit with boundary $\partial N \subseteq N$ and $R = \partial N$.
However, we need other cases too.

We now generalize Prop.~\ref{prop:dirichlet_problem} to any field with
positive elements, strengthen it, and provide a proof.  In this stronger version
we show that minimizing the extended power functional over all extensions of a
given $\psi \in \F^R$ gives a Dirichlet form $Q \maps \F^R \to \F$, the `power 
functional'.  The key point is that when power is minimized, the potential at every 
node can be chosen to depend linearly on $\psi$, in fact forming a weighted average, 
and that substituting these weighted averages into a Dirichlet form gives another Dirichlet 
form.

\begin{theorem} \label{thm:dirichlet_problem}
Given a passive linear circuit $(N,E,s,t,Z)$, suppose 
$R \subseteq N$ and $\psi \in \F^R$.  Then:
\begin{enumerate}[(i)]
\item There exists a potential $\phi \in \F^N$ obeying the principle of minimum power for $\psi$.
\item If $\phi$ and $\phi'$ both obey the principle of minimum power for
$\psi$, then $P(\phi) = P(\phi')$.  
\item If $\phi$ and $\phi'$ both obey the principle of minimum power for
$\psi$, then $\phi-\phi'$ is constant on every connected component of the circuit's underlying
graph.
\item There exists a unique potential $f(\psi) \in \F^N$ 
that obeys the principle of minimum power for $\psi$ and vanishes on every 
connected component not containing a node in $R$.
\item The potential $f(\psi)$ depends linearly on $\psi$.
\item If $P \maps \F^N \to \F$ is the extended power functional, then the \define{power
functional} $Q \maps \F^R \to \F$ given by
\[    Q(\psi) = P(f(\psi)) \]
is a Dirichlet form on $R$.
\end{enumerate}
\end{theorem}

\begin{proof}  

(i) The extended power functional $P$ is a Dirichlet form on $\F^N$.   Writing 
$\F^N = \F^{N \setminus R} \oplus \F^R$, the extensions of
$\psi \in \F^R$ form the affine subspace 
\[              S = \{ (\theta, \psi) \mid \theta \in \F^{N \setminus R} \} \] 
of $\F^N$.  It suffices
to show that  the differential of $P|_S \maps S \to \F$ vanishes at some point 
$\phi \in S$, since then $\phi$ obeys the principle of minimum power for $\psi$.

Write $s = \lvert N\setminus R\rvert$, $r = \lvert R \rvert$.  We
can express the quadratic form $P$ in terms of a symmetric matrix, which we
write in block form as
\[   \left( \begin{array}{cc} A & B^\top \\ B & C \end{array} \right) \]
where $A$ is a symmetric $s \times s$ matrix, $B$ is an $r \times s$ matrix, and
$C$ is a symmetric $r \times r$ matrix. This gives
\[   P(\theta,\xi) =  A \theta \cdot \theta + 2 B^\top \xi \cdot  \theta + \xi
\cdot C\xi    \]
for all $\theta \in \F^{N \setminus R}$ and $\xi \in \F^R$, where the dot products
are defined as usual on $\F^{N \setminus R}$ and $\F^R$.

For any $\phi = (\theta,\psi) \in S$, the
restriction of $P$ to $S$ then has the following differential at $\phi$:
\[     (dP|_S)_{(\theta,\psi)} = 2 (A \theta + B^\top \psi). \]
We must show that for some choice of $\theta$ this 
differential vanishes. Thus, it suffices to show that 
$\mathrm{im} B^\top \subseteq \mathrm{im} A$.  This is equivalent to  
$\mathrm{ker}A^\top \subseteq \mathrm{ker}B$, but $A = A^\top$, so it suffices to show that
\[         A \theta = 0 \implies B \theta = 0  \]
for all $\theta \in \F^{N\setminus R}$.

Proceeding by contradiction, suppose this is false.  Thus there exists $\theta$ with 
$A\theta = 0$ but $\xi \cdot B \theta \ne 0$ for some $\xi \in \F^R$.   For this
choice of $\theta$ and $\xi$
we have
\[   P(\theta,\xi) = 2 \xi \cdot B\theta  + \xi \cdot C \xi .\]
Since $\xi \cdot B\theta \ne 0$ we can choose $c \in \F$ such that 
\[         P(c\theta,\xi) = 2 c \xi \cdot B\theta + \xi \cdot C \xi = -1  \]
(simply solve the second equation for $c$).  This contradicts the fact that $P$ is
nonnegative, since we cannot have $-1 \ge 0$ in a field with positive elements.

(ii) This follows from the formal version of the multivariable Taylor theorem for
polynomial rings over a field of characteristic zero. Let $\phi, \phi' \in \F^N$ 
be extensions of $\psi$ obeying the principle of minimum power and note that
 $dP_{\phi}(\phi - \phi')=0$, since for all $n \in \partial N$ we have
$\phi(n) -\phi'(n) =0$, and for all $n \in N \setminus R$ we have
  \[
    \frac{\partial P}{\partial \varphi(n)}\bigg\vert_{\varphi = \phi}=0. 
  \]
We may take the Taylor expansion of $P$ around $\phi$ and evaluate at
$\phi'$. As $P$ is a quadratic form, this gives
\begin{align*}
    P(\phi') &=
    P(\phi)+dP_{\phi}(\phi'-\phi)+P(\phi'-\phi)
    \\
    & = P(\phi)+P(\phi'-\phi).
  \end{align*}
Similarly, we arrive at  
\[
    P(\phi)= P(\phi')+P(\phi-\phi').
\]
  But again as $P$ is a quadratic form, we then see that 
\[
    P(\phi')-P(\phi) = P(\phi'-\phi) =
    P(\phi-\phi') = P(\phi)-P(\phi').
\]
  This implies that $P(\phi')-P(\phi) = 0$, as required.
  
(iii) Suppose that $\phi$ and $\phi'$ both obey the principle of minimum power for
$\psi$.   Define $\alpha = \phi' - \phi$.    Then for each $\lambda \in \F$, $\phi + \lambda \alpha$ is an extension of $\psi$, and it obeys the principle of minimum power for $\psi$ when
$\lambda = 0$ and $\lambda = 1$.   Thus $f(\lambda) = P(\phi + \lambda \alpha)$ 
is a quadratic function whose derivative with respect to $\lambda$ 
vanishes at both $\lambda = 0$ and $\lambda = 1$, so it is constant.     On the other 
hand, when expanding
\[
f(\lambda) = \frac12\sum_{e \in E}
\frac1{Z(e)}\big((\phi+\lambda\alpha)(t(e))-(\phi+\lambda\alpha)(s(e))\big)^2,
\]
as a Taylor series, the coefficient of $\lambda^2$ is
\[
0 = \frac{1}{2}\sum_{e \in E} \frac1{Z(e)}\big(\alpha(t(e))-\alpha(s(e))\big)^2.
\]
Since $Z(e) \in \F^+$ for each edge $e$, the sum can only vanish if $\alpha$ is constant
on each connected component of the circuit's underlying graph.

(iv) By (i) there exists a potential $\phi$ that obeys the principle of minimum power
for $\psi$.    To obtain a potential $\phi'$ that obeys this principle and also vanishes
on every connected component not containing any node in $R$, simply set $\phi'(n) = 
\phi(n)$ when $n$ is in a connected component that contains a node in $R$, 
and $\phi(n)=0$ otherwise.  Thus $\phi'$ is
an extension of $\psi$, so we need only show that $\phi'$ obeys the principle
of minimum power for $\psi$.   For this we must show that any  
$n \in N \setminus R$ we have
\[ \left. \frac{\partial P(\varphi)}{\partial \varphi(n)}\right|_{\varphi = \phi'} = 0. \]

Note that
\begin{equation}
\label{eq.1}
  \frac{\partial P}{\partial \varphi(n)}\bigg|_{\varphi = \phi'} 
  = \sum_{t(e) = n} \frac1{Z(e)}\big(\phi'(t(e))-\phi'(s(e))\big) - \sum_{s(e) =
  n} \frac1{Z(e)}\big(\phi'(t(e))-\phi'(s(e))\big) 
\end{equation}
and since $\phi$ obeys the principle of minimum power for $\psi$ we have
\begin{equation}
\label{eq.2}
0 =  \frac{\partial P}{\partial \varphi(n)}\bigg|_{\varphi = \phi} 
  = \sum_{t(e) = n} \frac1{Z(e)}\big(\phi(t(e))-\phi(s(e))\big) - \sum_{s(e) =
  n} \frac1{Z(e)}\big(\phi(t(e))-\phi(s(e))\big). 
\end{equation}
If $n$ is in a connected component that contains a node in $R$, the right-hand sides of 
(\ref{eq.1}) and (\ref{eq.2}) are equal, so 
\[      \frac{\partial P}{\partial \varphi(n)}\bigg|_{\varphi = \phi'} = 0.\]
If $n$ is in a connected component that contains no nodes in $R$, 
each term in the right-hand side of (\ref{eq.1}) vanishes, so we obtain the same result.  

This shows existence of a potential $\phi$ that obeys the principle of minimum power
for $\psi$ and vanishes on every connected component not containing a node in $R$.
For uniqueness, note that given two such $\phi$ their difference is constant on each
connected component by (iii).  This constant must be zero for all the connected components
not containing a node in $R$, but also for those that do, since
we require $\phi|_R = \psi$.

(v) Fix $\psi, \psi' \in \F^R$, and suppose $\phi, \phi' \in \F^N$ obey the principle of minimum power for $\psi,\psi'$, respectively, and that both $\phi$ and $\phi'$ vanish
on every connected component containing no nodes in $R$.  For any $\lambda 
\in \F$, the function $\phi + \lambda \phi'$ vanishes on every connected 
component containing no nodes in $R$.   So, to prove the linearity
of $f$, it suffices to show that $\phi + \lambda \phi'$ obeys the principle of minimum power for
$\psi + \lambda \psi'$.

Since $P$ is a quadratic form, there is a linear operator $T \maps \F^N \to 
\F^{N \setminus R}$ that maps any potential $\zeta \in \F^N$ 
to the function from $N \setminus R$ to $\F$ given by
\[
n \longmapsto \frac{\partial P}{\partial \varphi(n)}\bigg|_{\varphi = \zeta} \;. \]
A potential $\zeta \in \F^N$ obeys the principle of minimum power for $\mu \in \F^R$
if and only if $\zeta|_R = \mu$ and $T\zeta = 0$.  Since both these equations
are linear, and $\phi, \phi' \in \R^N$ obey the principle of minimum power for $\psi,\psi'$, respectively, it follows that $\phi + \lambda \phi'$ obeys the principle of minimum power
for $\psi + \lambda \psi'$.

(vi) By the arguments so far, if we have a Dirichlet form $P$ on any set $S$, and any subset
$R \subseteq S$, there is a linear map $f \maps \F^R \to \F^S$
sending any $\psi \in \F^R$ to the unique extension $\phi \in \F^S$ obeying the
principle of minimum power for $P$.    This gives a new quadratic form which we call
$\min_{S \setminus R} P$ on $\F^R$, defined by
\[      (\min_{S \setminus R} P) (\psi) = P(f(\psi) \\) .\]
To complete the theorem it suffices to show that $\min_{S \setminus R} P$ is a Dirichlet form.
Given Lemma \ref{lem:onestepdirichletmin} below, this is an easy inductive argument
where we `minimize' over one node at a time.
\end{proof}

\begin{lemma} \label{lem:onestepdirichletmin}
  Let $P$ be a Dirichlet form on a set $S$ and let $s \in S$. 
  Then the quadratic form $\min_{\{s\}} P$ is a Dirichlet form on $S \setminus \{s\}$.
\end{lemma}
\begin{proof}
  Write $P(\phi) = \sum_{i,j} c_{ij}(\phi_i -\phi_j)^2$. Since
  $(\phi_i-\phi_j)^2 = (\phi_j-\phi_i)^2$, we may assume without loss of
  generality that $c_{sk} =0$ for all $k$; indeed, if $c_{sk} \ne 0$, simply
  choose new coefficients $c'$ such that $c'_{ks}:=c_{ks}+c_{sk}$, $c'_{sk}:=0$,
  and $c'_{ij}:=c_{ij}$ when $i,j \ne s$. 

  We then have
  \[
    \frac{\partial P}{\partial \varphi(s)}\bigg\vert_{\varphi = \phi} = \sum_k
    2c_{ks}(\phi_s-\phi_k),
  \]
  and this is equal to zero when
  \[
    \phi_s = \frac{\sum_k c_{ks}\phi_k}{\sum_k c_{ks}}.
  \]
  Note that the $c_{ij}$ lie in $\F^+$, and $\F^+$ is closed under addition, so
  $\sum_k c_{ks} \ne 0$. Thus $\min_{\{s\}}P$ may be given explicitly by the
  expression
  \[
    \min_{\{s\}} P(\psi) = \sum_{i,j \in S \setminus \{s\}} c_{ij}(\psi_i -\psi_j)^2 +
    \sum_{\ell \in S \setminus \{s\}} c_{\ell s}\left(\psi_\ell - \frac{\sum_k
      c_{ks} \psi_k}{\sum_k c_{ks}}\right)^2.
  \]
  We must show this is a Dirichlet form on $S \setminus \{s\}$. 
  
  As the sum of Dirichlet forms is evidently Dirichlet, it suffices to check that the expression 
  \[
    \sum_\ell c_{\ell s}\left(\psi_\ell - \frac{\sum_k c_{ks} \psi_k}{\sum_k
      c_{ks}}\right)^2
  \]
  is Dirichlet on $S \setminus \{s\}$. Multiplying through by the constant
  $(\sum_k c_{ks})^2 \in \F^+$, it further suffices to check
  \begin{align*}
    \sum_\ell c_{\ell s}\left(\sum_k c_{ks} \psi_\ell - \sum_k c_{ks}
    \psi_k\right)^2 &= \sum_\ell c_{\ell s} \left(\sum_k c_{ks} (\psi_\ell -
    \psi_k)\right)^2 \\
    &= \sum_\ell c_{\ell s} \left(2 \sum_{\substack{k,m \\ k \ne m}} c_{k s} c_{ms}
    (\psi_\ell-\psi_k)(\psi_\ell - \psi_m) + \sum_{k} c_{k
    s}^2(\psi_\ell-\psi_k)^2\right) \\
    &= 2\sum_{\substack{k,\ell,m \\ k \ne \ell, k \ne m, \ell \ne m}} c_{\ell s}
    c_{k s} c_{ms}
    (\psi_\ell-\psi_k)(\psi_\ell - \psi_m) + \sum_{k, \ell} c_{\ell s}c_{k
    s}^2(\psi_\ell-\psi_k)^2
  \end{align*}
  is a Dirichlet form. Since we can rewrite the first sum using the fact that
  \begin{align*}
    &\quad (\psi_k - \psi_\ell)(\psi_k - \psi_m)+(\psi_\ell - \psi_k)(\psi_\ell -
    \psi_m) + (\psi_m-\psi_k)(\psi_m-\psi_\ell) \\ 
    &= \psi_k^2+\psi_\ell^2+\psi_m^2-\psi_k\psi_\ell- \psi_k\psi_m -
    \psi_\ell\psi_m \\
    &= \tfrac12\big( (\psi_k-\psi_\ell)^2 +(\psi_k-\psi_m)^2
    +(\psi_\ell-\psi_m)^2\big),
  \end{align*}
  this expression is indeed a Dirichlet form. Indeed, pasting these computations
  together shows that
  \[
    \min_{\{s\}}P(\psi) = \sum_{i,j} \left(c_{ij}+\frac{c_{is}c_{js}}{ \sum_k
    c_{ks}}\right)(\psi_i-\psi_j)^2. \qedhere
  \]
\end{proof}

It is now straightforward to generalize
Prop.~\ref{minimum_power_implies_kirchhoff_current} and
Prop.~\ref{boundary_current_determines_boundary_voltage} to the present
context.   Briefly, the principle of minimum power is equivalent to
Kirchhoff's current law, and the differential of the power functional determines 
the current flowing in or out of a circuit's terminals:

\begin{theorem} \label{thm:circuit_behavior_from_power}  
Given a passive linear circuit $(N,E,s,t,Z)$ with boundary
$\partial N \subseteq N$, suppose $\psi \in \R^{\partial N}$.    Then $\phi$ 
obeys the principle of minimum power for $\psi$ if and only if the current 
$I \in \R^N$ given by  
\[  I(e) = \frac1{Z(e)}(\phi(t(e))-\phi(s(e))) \] 
obeys Kirchhoff's current law on $N \setminus \partial N$.   Moreover, if $\phi$ 
obeys the principle of minimum power for $\psi$, then 
$dQ_\psi \in (\F^{\partial N})^\ast \cong \F^{\partial N}$ equals the boundary current
$\iota \in \F^{\partial N}$ given by 
\begin{align*}
\iota \maps \partial N &\longrightarrow \R \\
n &\longmapsto \sum_{t(e) = n} I(e) -\sum_{s(e) = n} I(e).
\end{align*}
\end{theorem}

\begin{proof} The proofs are the same as before, with formal derivatives
replacing derivatives.
\end{proof}

\subsection{Composition of Dirichlet forms}
\label{subsec:dirichlet_composition}

It would be nice to have a category in which circuits are morphisms, and a
category in which Dirichlet forms are morphisms, such that the map sending
a circuit to its behavior is a functor.  Here we present a na\"ive attempt to
construct the category with Dirichlet forms as morphisms, using the principle
of minimum power to compose these morphisms.  There is a hitch:
the proposed category does not include identity morphisms.  However, our construction
points in the right direction, and underlines the importance of the cospan
formalism we develop in the next Part~\ref{part:categories}.

We can define a composition rule for Dirichlet forms that reflects
composition of circuits.  Given finite sets $S$ and $T$, write $S+T$ for their
disjoint union.  Let $D(S,T)$ be the set of Dirichlet forms over $(\F,\F^+)$ on
$S+T$. There is a way to compose these Dirichlet forms
\[ 
\circ \maps D(T,U) \times D(S,T) \to D(S,U) 
\]
defined as follows.  Given $P \in D(T,U)$ and $Q \in D(S,T)$, let
\[ 
  (P \circ Q)(\alpha, \gamma) = \min_{T} Q(\alpha, \beta) + P(\beta, \gamma),
\]
where $\alpha \in F^S, \gamma \in F^U$. This operation has a clear
interpretation in terms of electrical circuits: the power used by the entire
circuit is just the sum of the power used by its parts. 

It is immediate from Thm.~\ref{thm:dirichlet_problem} 
that this composition rule is well defined: the composite of two Dirichlet
forms is again a Dirichlet form. Moreover, this composition is associative.
However, it fails to provide the structure of a category, as there is typically
no Dirichlet form $1_S \in D(S,S)$ playing the role of the identity for this
composition. For an indication of why this is so, let $S$ be a set
with one element, and suppose that some Dirichlet form $I(\beta,\gamma) =
k(\beta-\gamma)^2 \in D(S,S)$ acts as an identity on the
right for this composition. Then for all $Q(\alpha,\beta) = c(\alpha-\beta)^2
\in D(S,S)$, we must have
\[
  c\alpha^2 = Q(\alpha,0) = (I \circ Q)(\alpha,0) = \min_{\beta \in \F}
  Q(\alpha, \beta) + I(\beta,0) = \min_{\beta \in \F} k(\alpha-\beta)^2 +
  c\beta^2 = \frac{kc}{k+c}\alpha^2.
\]
But this
equality only holds when $c = 0$, so no such Dirichlet form exists. Note,
however, that in the real case, if $k \gg c$ we have 
$c\alpha^2 \approx \frac{kc}{k+c}\alpha^2$, so
Dirichlet forms with large values of $k$---corresponding to resistors with
resistance close to zero---act as `approximate identities'.

We might thus interpret any hypothetical identities in this category as
behaviors of idealized components with zero resistance, or perfectly conductive
wires. Unfortunately, the power functional of a purely conductive wire is
undefined; the formula for it involves division by zero.  In real life, coming
close to this situation leads to the disaster that electricians call a `short
circuit': a huge amount of power dissipated for even a small voltage.  

Nonetheless, we have most of the structure required for a category. A `category
without identity morphisms' is called a \define{semicategory}, so we see
\begin{proposition}
There is a semicategory where:
\begin{itemize}
\item the objects are finite sets,

\item a morphism from $T$ to $S$ is a Dirichlet form $Q \in D(S,T)$.  

\item composition of morphisms is given by 
\[
(R \circ Q)(\gamma, \alpha) = \min_{T} Q(\gamma, \beta) + R(\beta, \alpha).
\]

\end{itemize}
\end{proposition}

We would like to make this into a category. One easy way to do this is to
formally adjoin identity morphisms; this trick works for any semicategory.
However, we obtain a better category if we include \emph{more} morphisms 
corresponding to circuits made of perfectly conductive wires. 

As the expression for the extended power functional includes the reciprocals of
impedances, such circuits cannot be expressed within the framework we have
developed thus far. Indeed, for these idealized circuits there is no function
taking boundary potentials to boundary currents: the vanishing impedance would
imply that any difference in potentials at the boundary induces `infinite'
currents. In the next part, we deal with this by generalizing Dirichlet forms to
Lagrangian relations.

\part{Categories of Circuits} 
\label{part:categories}
In this part we move our focus from the semantics of circuit diagrams to the
syntax, addressing the question ``How do we interact with circuit diagrams?''.
Informally, the answer to this is that we interact with them by connecting them
to each other, perhaps after moving them into the right form by rotating or
reflecting them, or by crossing or bending some of the wires. To formalize this,
we adopt a category theoretic viewpoint, defining various categories with
circuits and their behaviors as morphisms. As we wish to capture the above
operations, these categories will be endowed with additional structure, in particular
the structure of a hypergraph category.   We claim a formal analysis of this 
structure, especially of the composition or interconnection of circuits, has been 
overlooked in analysis of circuits thus far. 

This part culminates in the definition of two important categories, the
category $\Circ$ of circuit diagrams, and the category $\Lag\Rel$ containing all
behaviors of circuits. We also develop the technical material required to
appreciate the structure of these categories, and that aids understanding of the
relationship between the two, to be addressed in Part \ref{part:blackbox}.

\section{Decorated cospans} \label{sec:cospans}
We begin this part with a technical section describing a general technique for
developing composition rules for structures on finite sets. As we have seen,
whether represented by circuit diagrams or Dirichlet forms, circuits can be
described as structures on a finite set of nodes. While this provides a good
classification of the different types of circuits that exist, it does not allow
for discussion of their composition. In this section, however, we describe a
method for taking (1) a description of a structure that can be placed on finite
sets together with (2) a description of how this structure interacts with
functions between these sets, and producing a category which describes
composition of structures. This category is built as a cospan category, with the
apex of the cospan describing some structure, such as a circuit, and the feet of
the cospan describing possible interfaces to this structure.

\subsection{Cospan categories} \label{subsec:cospans}

Recall that a \define{cospan} from $X$ to $Y$ in a category $\C$ is an object
$S \in \C$ with a pair of morphisms $f\maps X \to S$, $g\maps Y \to S$:
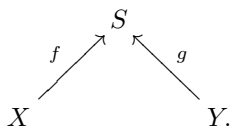
\[
  \xymatrix{
    & S \\
    X \ar[ur]^{f} && Y. \ar[ul]_g
  }
\]
We call $X$ and $Y$ the \define{feet} of the cospan and call $S$ its \define{apex}.  
When $\C$ has pushouts, cospans may be composed using the
pushout from the common foot: given cospans $X \stackrel{f}{\longrightarrow} C
\stackrel{g}{\longleftarrow} Y$ and $Y
\stackrel{f'}{\longrightarrow} C' \stackrel{g'}{\longleftarrow} Z$, their
composite cospan is $X \stackrel{i \circ f}{\longrightarrow} P
\stackrel{i'\circ g'}{\longleftarrow} Z$ where $P, i$ and $i'$ form the top half of
this pushout square:
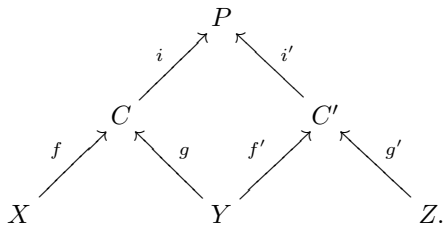
\[
  \xymatrix{
    && P \\
    & C \ar[ur]^i && C' \ar[ul]_{i'} \\
    X \ar[ur]^f && Y \ar[ul]_g \ar[ur]^{f'} && Z. \ar[ul]_{g'}
  }
\]
A \define{map of cospans} is a morphism $h\maps S \to S'$ in $\C$ between the
apices of two cospans $X \stackrel{f}{\longrightarrow} S
\stackrel{g}{\longleftarrow} Y$ and $X \stackrel{f'}{\longrightarrow} S'
\stackrel{g'}{\longleftarrow} Y$ with the same feet, such that 
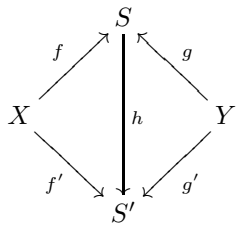
\[
  \xymatrix{
    & S \ar[dd]^h  \\
    X \ar[ur]^{f} \ar[dr]_{f'} && Y \ar[ul]_g \ar[dl]^{g'}\\
    & S'
  }
\]
commutes. Given a category $\C$ with pushouts, we may define a category
$\Cospan(\C)$ with objects of $\C$ as objects and isomorphism classes of cospans 
in $\C$ as morphisms. We will often abuse our terminology
and refer to a cospans itself as a morphisms in some category
$\Cospan(\C)$; we of course refer instead to the isomorphism class of
the said cospan. Note that there is a functor
\[
 i \maps \C \to \Cospan(\C),
\]
taking any object of $\C$ to its corresponding
object in $\Cospan(\C)$ and taking any morphism $f\maps X \to Y$ to the
cospan $X \stackrel{f}\longrightarrow Y \stackrel{1_Y}\longleftarrow Y$. 
This functor is faithful and bijective on objects.   For
this reason we often treat $\C$ as a subcategory of $\Cospan(\C)$. 

Cospan categories also come equipped with a so-called dagger structure, which
maps $X \stackrel{f}{\longrightarrow} S \stackrel{g}{\longleftarrow} Y$ to its reflection $Y
\stackrel{g}{\longrightarrow} S \stackrel{f}{\longleftarrow} X$.   Moreover,
when $\C$ has finite colimits, this dagger structure arises from a compact
closed structure, which in turn results from a hypergraph structure.  Since
these structures are important in circuit theory, they warrant a brief review.

\subsection{Hypergraph categories} \label{subsec:hypergraph}

It is useful to treat systems with inputs and outputs as morphisms in a category, so that
composition corresponds to connecting systems, with the outputs of one 
system attached to the inputs of the next.   A symmetric monoidal category gives us the further
ability to treat systems with multiple inputs and outputs as morphisms between tensor products
of objects.  The calculus of string diagrams \cite{JS1} lets us reason with such morphisms
using diagrams.   In brief, to set up our conventions, we represent a morphism $f\maps
X_1 \otimes \dots \otimes X_m \to Y_1 \otimes \dots \otimes Y_n$ as follows:
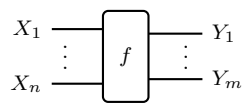
\[
  \begin{tikzpicture}
  \node[draw, outer sep=0,thick,rounded corners=2pt, minimum height=8ex, minimum
  width=4ex] (f) {$f$};
  \node at ($(f.north west)!.2!(f.south west)-(1,0)$) (x1) {$X_1$};
  \node at ($(f.north west)!.8!(f.south west)-(1.018,0)$) (xn) {$X_n$};
  \node at ($(f.north west)!.42!(f.south west)-(.5,0)$) (xl) {$\vdots$};
  \node at ($(f.north east)!.25!(f.south east)+(1,0)$) (y1) {$Y_1$};
  \node at ($(f.north east)!.75!(f.south east)+(1.052,0)$) (yn) {$Y_m$};
  \node at ($(f.north east)!.42!(f.south east)+(.5,0)$) (yl) {$\vdots$};
  \begin{scope}[thick]
  \draw (x1) to ($(x1)+(1,0)$);
  \draw (xn) to ($(xn)+(1.018,0)$);
  \draw (y1) to ($(y1)-(1,0)$);
  \draw (yn) to ($(yn)-(1.052,0)$);
  \end{scope}
  \end{tikzpicture}
\]
Composition is then represented by connecting the lines or `wires' representing
the codomain of one morphism with the domain of the other placed beside it.
The tensor product of two morphisms is represented by their side-by-side
juxtaposition, and the symmetry by crossing wires.   

Dagger compact categories \cite{AC,Se} are a special class of symmetric monoidal
categories that permit additional manipulations.   The compactness lets us
convert an individual input into an output or vice versa by bending a wire 180
degrees.   The dagger structure lets us reflect the whole diagram, interchanging
all inputs and outputs.  Dagger compact categories are widespread in
applications of category theory to open systems \cite{BaezStay,CP}.    However,
most of the categories considered in this paper have even more structure: they
are hypergraph categories \cite{Fong16,FongSpivak18}.

We recall the definition of these with the help of string diagrams.  First, recall that a 
\define{special commutative Frobenius structure} $(\mu,\eta,\delta,\epsilon)$ on 
an object $X$ in a symmetric monoidal category $(\C,  \otimes, I)$ consists of morphisms 
\[
  \xymatrixrowsep{1pt}
  \xymatrixcolsep{30pt}
  \xymatrix{
    \mult{.075\textwidth} & \unit{.075\textwidth} &
    \comult{.075\textwidth} & \counit{.075\textwidth} 
    \\
    \mu\colon X\otimes X \to X & \eta\colon I \to X &
    \delta\colon X \to X\otimes X & \epsilon\colon X \to I
  }
\]
obeying the equations
\[
  \xymatrixrowsep{1pt}
  \xymatrixcolsep{25pt}
  \xymatrix{
    \assocl{.1\textwidth} = \assocr{.1\textwidth} & \unitl{.1\textwidth} =
    \idone{.1\textwidth} & \commute{.1\textwidth} = \mult{.07\textwidth}
  }
\]
\[
  \xymatrixrowsep{1pt}
  \xymatrixcolsep{25pt}
  \xymatrix{
    \frobs{.1\textwidth} = \frobx{.1\textwidth} 
    & \spec{.1\textwidth} = \idone{.1\textwidth}
    }
\]
and their reflected versions, where $\swap{1em}$ denotes the swap  
$X \otimes X \to X \otimes X$.   The first row of equations says that $X$ is a commutative
 monoid; its mirror image says that $X$ is a cocommutative comonoid.  The first equation
 in the second row, together with its mirror image, says that the monoid and comonoid
 structures form a Frobenius monoid.  The last equation is called the \define{special}
 law.

A \define{hypergraph category} is a symmetric monoidal category
$(\mathcal{C},I,\otimes)$ in which each object $X$ is equipped with a special
commutative Frobenius structure $(\mu_X,\eta_X,\delta_X,\epsilon_X)$  in such a way 
that the structure on $X \otimes Y$ is determined by those on $X$ and $Y$ as follows
  \[
    \tikzset{every path/.style={line width=1.1pt}}
    \xymatrixrowsep{1ex}
    \xymatrix{
    \begin{aligned}
    \resizebox{.15\textwidth}{!}{
      \begin{tikzpicture}[scale=.65]
	\begin{pgfonlayer}{nodelayer}
		\node [style=none] (0) at (-0.25, 0.5) {};
		\node [style=bdot] (1) at (0.5, -0) {};
		\node [style=none] (2) at (-0.25, -0.5) {};
		\node [style=none] (3) at (1.25, -0) {};
		\node [style=none] (4) at (-1.5, 0.5) {$X \otimes Y$};
		\node [style=none] (5) at (-1.5, -0.5) {$X \otimes Y$};
		\node [style=none] (6) at (2, -0) {$X \otimes Y$};
		\node [style=none] (7) at (-0.75, 0.5) {};
		\node [style=none] (8) at (-0.75, -0.5) {};
	\end{pgfonlayer}
	\begin{pgfonlayer}{edgelayer}
		\draw [in=90, out=0, looseness=0.90] (0.center) to (1.center);
		\draw [in=-90, out=0, looseness=0.90] (2.center) to (1.center);
		\draw (1.center) to (3.center);
		\draw (7.center) to (0.center);
		\draw (8.center) to (2.center);
	\end{pgfonlayer}
\end{tikzpicture}}
\end{aligned}
=
\begin{aligned}
    \resizebox{.12\textwidth}{!}{
  \begin{tikzpicture}[scale=.65]
	\begin{pgfonlayer}{nodelayer}
		\node [style=none] (0) at (-0.25, 0.5) {};
		\node [style=bdot] (1) at (0.5, -0) {};
		\node [style=none] (2) at (-0.25, -0.5) {};
		\node [style=none] (3) at (1.25, -0) {};
		\node [style=none] (4) at (-1.75, 0.5) {$X$};
		\node [style=none] (5) at (-1.75, -1.25) {$X$};
		\node [style=none] (6) at (1.5, -0) {$X$};
		\node [style=none] (7) at (-1.5, 0.5) {};
		\node [style=none] (8) at (-1.5, -1.25) {};
		\node [style=none] (9) at (1.25, -1.75) {};
		\node [style=none] (10) at (-1.5, -2.25) {};
		\node [style=none] (11) at (1.5, -1.75) {$Y$};
		\node [style=none] (12) at (-1.5, -0.5) {};
		\node [style=none] (13) at (-0.25, -2.25) {};
		\node [style=bdot] (14) at (0.5, -1.75) {};
		\node [style=none] (15) at (-0.25, -1.25) {};
		\node [style=none] (16) at (-1.75, -0.5) {$Y$};
		\node [style=none] (17) at (-1.75, -2.25) {$Y$};
	\end{pgfonlayer}
	\begin{pgfonlayer}{edgelayer}
		\draw [in=90, out=0, looseness=0.90] (0.center) to (1.center);
		\draw [in=-90, out=0, looseness=0.90] (2.center) to (1.center);
		\draw (1.center) to (3.center);
		\draw (7.center) to (0.center);
		\draw [in=180, out=0, looseness=1.00] (8.center) to (2.center);
		\draw [in=90, out=0, looseness=0.90] (15.center) to (14);
		\draw [in=-90, out=0, looseness=0.90] (13.center) to (14);
		\draw (14) to (9.center);
		\draw [in=180, out=0, looseness=1.00] (12.center) to (15.center);
		\draw (10.center) to (13.center);
	\end{pgfonlayer}
\end{tikzpicture}}
\end{aligned}
& &   
\qquad
  \begin{aligned}
    \resizebox{.1\textwidth}{!}{
    \begin{tikzpicture}[scale=.65]
	\begin{pgfonlayer}{nodelayer}
		\node [style=bdot] (0) at (-0.5, 0.5) {};
		\node [style=none] (1) at (1.5, 0.5) {$X\otimes Y$};
		\node [style=none] (2) at (0.75, 0.5) {};
	\end{pgfonlayer}
	\begin{pgfonlayer}{edgelayer}
		\draw (2.center) to (0.center);
	\end{pgfonlayer}
\end{tikzpicture}}
  \end{aligned}
  = \qquad
  \begin{aligned}
    \resizebox{.07\textwidth}{!}{
    \begin{tikzpicture}[scale=.65]
	\begin{pgfonlayer}{nodelayer}
		\node [style=bdot] (0) at (0, 0.5) {};
		\node [style=none] (1) at (1.5, 0.5) {$X$};
		\node [style=none] (2) at (1.25, 0.5) {};
		\node [style=none] (3) at (1.25, -0.25) {};
		\node [style=bdot] (4) at (0, -0.25) {};
		\node [style=none] (5) at (1.5, -0.25) {$Y$};
	\end{pgfonlayer}
	\begin{pgfonlayer}{edgelayer}
		\draw (2.center) to (0.center);
		\draw (3.center) to (4.center);
	\end{pgfonlayer}
\end{tikzpicture}}
  \end{aligned}
\\
    \begin{aligned}
    \resizebox{.15\textwidth}{!}{
\begin{tikzpicture}[scale=.65]
	\begin{pgfonlayer}{nodelayer}
		\node [style=none] (0) at (0.75, 0.5) {};
		\node [style=bdot] (1) at (0, -0) {};
		\node [style=none] (2) at (0.75, -0.5) {};
		\node [style=none] (3) at (-0.75, -0) {};
		\node [style=none] (4) at (2, 0.5) {$X \otimes Y$};
		\node [style=none] (5) at (2, -0.5) {$X \otimes Y$};
		\node [style=none] (6) at (-1.5, -0) {$X \otimes Y$};
		\node [style=none] (7) at (1.25, 0.5) {};
		\node [style=none] (8) at (1.25, -0.5) {};
	\end{pgfonlayer}
	\begin{pgfonlayer}{edgelayer}
		\draw [in=90, out=180, looseness=0.90] (0.center) to (1.center);
		\draw [in=-90, out=180, looseness=0.90] (2.center) to (1.center);
		\draw (1.center) to (3.center);
		\draw (7.center) to (0.center);
		\draw (8.center) to (2.center);
	\end{pgfonlayer}
\end{tikzpicture}}
\end{aligned}
=
\begin{aligned}
    \resizebox{.12\textwidth}{!}{
\begin{tikzpicture}[scale=.65]
	\begin{pgfonlayer}{nodelayer}
		\node [style=none] (0) at (0, 0.5) {};
		\node [style=bdot] (1) at (-0.75, -0) {};
		\node [style=none] (2) at (0, -0.5) {};
		\node [style=none] (3) at (-1.5, -0) {};
		\node [style=none] (4) at (1.5, 0.5) {$X$};
		\node [style=none] (5) at (1.5, -1.25) {$X$};
		\node [style=none] (6) at (-1.75, -0) {$X$};
		\node [style=none] (7) at (1.25, 0.5) {};
		\node [style=none] (8) at (1.25, -1.25) {};
		\node [style=none] (9) at (-1.5, -1.75) {};
		\node [style=none] (10) at (1.25, -2.25) {};
		\node [style=none] (11) at (-1.75, -1.75) {$Y$};
		\node [style=none] (12) at (1.25, -0.5) {};
		\node [style=none] (13) at (0, -2.25) {};
		\node [style=bdot] (14) at (-0.75, -1.75) {};
		\node [style=none] (15) at (0, -1.25) {};
		\node [style=none] (16) at (1.5, -0.5) {$Y$};
		\node [style=none] (17) at (1.5, -2.25) {$Y$};
	\end{pgfonlayer}
	\begin{pgfonlayer}{edgelayer}
		\draw [in=90, out=180, looseness=0.90] (0.center) to (1.center);
		\draw [in=-90, out=180, looseness=0.90] (2.center) to (1.center);
		\draw (1.center) to (3.center);
		\draw (7.center) to (0.center);
		\draw [in=0, out=180, looseness=1.00] (8.center) to (2.center);
		\draw [in=90, out=180, looseness=0.90] (15.center) to (14);
		\draw [in=-90, out=180, looseness=0.90] (13.center) to (14);
		\draw (14) to (9.center);
		\draw [in=0, out=180, looseness=1.00] (12.center) to (15.center);
		\draw (10.center) to (13.center);
	\end{pgfonlayer}
\end{tikzpicture}}
\end{aligned}
& &
  \begin{aligned}
    \resizebox{.1\textwidth}{!}{
    \begin{tikzpicture}[scale=.65]
	\begin{pgfonlayer}{nodelayer}
		\node [style=bdot] (0) at (1.5, 0.5) {};
		\node [style=none] (1) at (-0.5, 0.5) {$X\otimes Y$};
		\node [style=none] (2) at (0.25, 0.5) {};
	\end{pgfonlayer}
	\begin{pgfonlayer}{edgelayer}
		\draw (2.center) to (0.center);
	\end{pgfonlayer}
\end{tikzpicture}}
  \end{aligned}
  \qquad \quad =
  \begin{aligned}
    \resizebox{.07\textwidth}{!}{
    \begin{tikzpicture}[scale=.65]
	\begin{pgfonlayer}{nodelayer}
		\node [style=bdot] (0) at (1.5, 0.5) {};
		\node [style=none] (1) at (0, 0.5) {$X$};
		\node [style=none] (2) at (0.25, 0.5) {};
		\node [style=none] (3) at (0.25, -0.25) {};
		\node [style=bdot] (4) at (1.5, -0.25) {};
		\node [style=none] (5) at (0, -0.25) {$Y$};
	\end{pgfonlayer}
	\begin{pgfonlayer}{edgelayer}
		\draw (2.center) to (0.center);
		\draw (3.center) to (4.center);
	\end{pgfonlayer}
\end{tikzpicture}}
\qquad \quad
  \end{aligned}
}
\]
for all $X, Y \in \C$, and such that the structure on the monoidal unit $I$ is
equal to $(\rho^{-1}_I, \id_I,\rho_I,\id_I)$.

These operations on diagrams---placing diagrams on the same page, rearranging,
splitting, combining and terminating wires, and then connecting these wires to
form a larger diagram---represent precisely the collection of operations used
for reasoning with circuit diagrams.  Thus, hypergraph categories are a good
setting for formalizing circuit diagrams.   

In a hypergraph category $\H$ each object $X$ is its own dual, with the unit $I \to X \otimes X$ 
being the composite of $\eta_X \maps I \to X$ and $\delta_X \maps X \to X \otimes X$, and the counit $X \otimes X \to I$ being the composite of $\mu_X \maps X \otimes X \to X$ and 
$\epsilon_X \maps X \to I$.   Using this duality, every morphism $f \maps X \to Y$  gives rise to a 
morphism $f^\dagger \maps Y \to X$, making $\H$ into a dagger category.  It is then a simple
computation to check that $\H$ is in fact dagger compact.

\begin{example} \label{ex.cospan_is_hypergraph}
  Whenever $\C$ is a 
  category with finite colimits, $\Cospan(\C)$ is a hypergraph category \cite{Fong15,RSW08}.   
  To understand this, note first that $\Cospan(\C)$ is symmetric monoidal when equipped
  with the tensor product arising from coproducts in the category $\C$,
  together with the structure maps inherited from viewing $(\mathrm{\C},+)$
  as a subcategory.   The Frobenius structure $(\mu_X,\eta_X,\delta_X,\epsilon_X)$ on each object $X$ is then given by the cospans
  \[
  \xymatrixrowsep{1pt}
  \xymatrixcolsep{50pt}
    \xymatrix{
    X+X \xrightarrow{[1_X,1_X]} X \stackrel{1_X}\longleftarrow X &
    0 \stackrel{!}\longrightarrow X \stackrel{1_X}\longleftarrow X \\
    X \stackrel{1_X}\longrightarrow X \xleftarrow{[1_X,1_X]} X+X &
    X \stackrel{1_X}\longrightarrow X \stackrel{!}\longleftarrow 0
    }
  \]
  where $! \maps X \to 0$ is the unique morphism to the initial object of $\C$, 
  and we write $[f,g]$ for the copairing of two morphisms $f$ and $g$ with a common codomain.
 \end{example}

In addition to hypergraph categories, we need functors between them.  Given
hypergraph categories $\H$ and $\H'$, a \define{hypergraph functor} is a strong symmetric monoidal functor $(F,\varphi) \maps \H \to \H'$ that preserves the chosen Frobenius structures
on each object $X$.   More precisely, we demand that for each $X \in \H$, the Frobenius structure on $FX$ is
\[
  (F\mu_X \circ \varphi_{X,X},\enspace  \varphi^{-1}_{X,X} \circ F\delta_X,\enspace  F\eta_X \circ
  \varphi_I,\enspace  \varphi_I^{-1} \circ F\epsilon_X)
\]
where $\varphi_{X,X} \maps FX \otimes FX \to F(X \otimes X)$ and $\varphi_I \maps I \to FI$ are
the coherence maps for $F$.   If these coherence maps are all identity morphisms, we say $(F,\varphi) \maps \H \to \H'$ is a \define{strict} hypergraph functor.   

Just as any hypergraph category can be made into a dagger compact category as explained 
above, any hypergraph functor gives a symmetric monoidal dagger functor.  Thus, the category 
of hypergraph categories and hypergraph functors can be seen as an enhanced version of the
more widely studied category of dagger compact categories and symmetric monoidal monoidal
dagger functors.

\subsection{Decorated cospan categories}
\label{ssec:decorated_cospans}

An important example of a hypergraph category is the category of cospans 
in $\Fin\Set$, the category of finite sets and functions.  However, to describe circuits
we need a more flexible class of hypergraph categories.  This is provided by the idea 
of an `$F$-decorated' cospan: a cospan in $\Fin\Set$ in which the apex $N$ is equipped with an element of some set $FN$.  We think of $F$ as describing the collection of available structures on $N$: examples include the collection of circuit diagrams or Dirichlet
forms on $N$.

\begin{lemma} 
\label{lemma:decorated_cospans}
  Let
  \[
    (F,\varphi)\maps (\Fin\Set,+) \longrightarrow (\Set, \times)
  \]
  be a lax symmetric monoidal functor.  There is a category
  $F\Cospan$, the category of \define{$F$-decorated cospans}, with
   objects being finite sets and morphisms from $X$ to $Y$ being equivalence classes of pairs 
  \[
    (X \stackrel{i}\longrightarrow N \stackrel{o}\longleftarrow Y,\enspace s)
  \]
  comprising a cospan $X \stackrel{i}\rightarrow N \stackrel{o}\leftarrow Y$ in
  $\Fin\Set$ together with an element $s \in FN$.  We call $s$ the
  \define{decoration}. The equivalence relation arises
  from isomorphism of cospans: an isomorphism of cospans induces a one-to-one
  correspondence between their decorations.

  Composition in this category is given via pushout of cospans in 
  $\Fin\Set$:
  \[
    \xymatrix{
      && N+_YM \\
      & N \ar[ur]^{j_N} && M \ar[ul]_{j_M} \\
      \quad X \quad \ar[ur]^{i_X} && Y \ar[ul]_{o_Y} \ar[ur]^{i_Y} && \quad Z \quad \ar[ul]_{o_Z}
    }
  \]
 together with applying the map
  \[
   FN \times FM
    \xrightarrow{\varphi_{N,M}} F(N+M)
    \xrightarrow{F[j_N,j_M]} F(N+_YM)
  \]
 to the pair of decorations.
\end{lemma}

\begin{proof} This follows from Thm.~3.4 of \cite{Fong15}. \end{proof}

Decorated cospan categories are so named as they generalize the category of
cospans of finite sets: the theorem just cited also shows that there is a 
functor $i \maps \Cospan(\Fin\Set) \to F\Cospan$ that is faithful and bijective 
on objects.   

The category $F\Cospan$ is monoidal, where we define the tensor 
product of finite sets to be their disjoint union $X+Y$ and define the 
tensor product of decorated cospans $(X \stackrel{i_X}\longrightarrow N
\stackrel{o_Y}\longleftarrow Y,\enspace s)$ and
$(X' \stackrel{i_{X'}}\longrightarrow N' \stackrel{o_{Y'}}\longleftarrow
Y',\enspace s')$ to be 
\[
  \Big(\, X+X' \xrightarrow{i_X+i_{X'}} N+N' \xleftarrow{o_Y+o_{Y'}} Y+Y', \quad
  \varphi_{N,N'}(s,s') \,\Big).
\]
We also write $+$ for the tensor product in $F\Cospan$.   $F\Cospan$ also inherits a 
symmetric monoidal struture from its subcategory $\Cospan(\Fin\Set)$.
In fact $F\Cospan$ is a hypergraph category, and thus in particular dagger compact:

\begin{lemma} \label{lemma:dccsarehypergraph}
Let  $(F,\varphi)\maps (\Fin\Set,+) \longrightarrow (\Set, \times)$ be a lax symmetric 
monoidal functor. Then the symmetric monoidal category $F\Cospan$ can be  equipped 
with the structure of a hypergraph category in a unique way such that the functor 
$i \maps \Cospan(\Fin\Set) \to F\Cospan$ is a hypergraph functor.  
\end{lemma}

\begin{proof}
This follows from Thm.~3.4 of \cite{Fong15}.  Uniqueness follows from the fact
that $i$ is a strict monoidal functor, bijective on objects.
\end{proof}

Decorated cospans allow us to understand the syntax of circuit diagrams. 
Equally crucial to our understanding of circuit diagrams, however, is their semantics as 
discussed in Part \ref{part:circuits}.  For this, we use a procedure to construct
functors between decorated cospan categories.

\begin{lemma} \label{lemma:decoratedfunctors}
  Let 
  \[
    (F,\varphi), (G,\gamma)\maps  (\Fin\Set,+) \longrightarrow (\Set, \times)
  \]
  be lax symmetric monoidal functors and let 
  \[
    \theta\maps (F,\varphi) \Longrightarrow (G,\gamma)
  \]
  be a monoidal natural transformation between them. Then we may define a
  functor, in fact a hypergraph functor, 
  \[    T\maps F\Cospan \longrightarrow G\Cospan
  \]
  by letting any finite set
  $X$ in $F\Cospan$ map to the same finite set as an object of
  $G\Cospan$, and letting any morphism
  \[
    (X \stackrel{i}\longrightarrow N \stackrel{o}\longleftarrow Y, 1
    \stackrel{s}\longrightarrow FN)
  \]
  map to:
  \[
    (X \stackrel{i}\longrightarrow N \stackrel{o}\longleftarrow Y, 1
    \stackrel{\theta_N\circ s}\longrightarrow GN).
  \]
\end{lemma}
\begin{proof}
This is a special case of Thm.~4.1 of \cite{Fong15}.
\end{proof}

\section{Open circuits and their semantics} \label{sec:circsemantics}
In Part \ref{part:circuits}, we defined a circuit with boundary to be a labelled
graph with a chosen subset of nodes called `terminals'. To form a category, we now more
explicitly describe these terminals using a cospan, as in the example seen earlier:
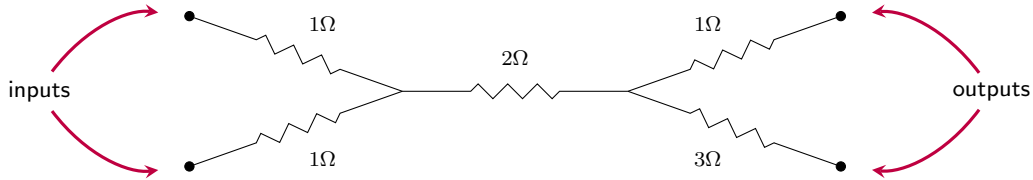
\[
\begin{tikzpicture}[circuit ee IEC, set resistor graphic=var resistor IEC graphic]
\node[contact] (I1) at (0,2) {};
\node[contact] (I2) at (0,0) {};
\coordinate (int1) at (2.83,1) {};
\coordinate (int2) at (5.83,1) {};
\node[contact] (O1) at (8.66,2) {};
\node[contact] (O2) at (8.66,0) {};
\node (input) at (-2,1) {\small{\textsf{inputs}}};
\node (output) at (10.66,1) {\small{\textsf{outputs}}};
\draw (I1) 	to [resistor] node [label={[label distance=2pt]85:{$1\Omega$}}] {} (int1);
\draw (I2)	to [resistor] node [label={[label distance=2pt]275:{$1\Omega$}}] {} (int1)
				to [resistor] node [label={[label distance=3pt]90:{$2\Omega$}}] {} (int2);
\draw (int2) 	to [resistor] node [label={[label distance=2pt]95:{$1\Omega$}}] {} (O1);
\draw (int2)		to [resistor] node [label={[label distance=2pt]265:{$3\Omega$}}] {} (O2);
\path[color=purple, very thick, shorten >=10pt, ->, >=stealth, bend left] (input) edge (I1);		\path[color=purple, very thick, shorten >=10pt, ->, >=stealth, bend right] (input) edge (I2);		
\path[color=purple, very thick, shorten >=10pt, ->, >=stealth, bend right] (output) edge (O1);
\path[color=purple, very thick, shorten >=10pt, ->, >=stealth, bend left] (output) edge (O2);
\end{tikzpicture}
\]
Such circuits are examples of decorated cospans.  We obtain a decorated cospan
category $\Circ$ with open circuits as morphisms.   The hypergraph
structure of $\Circ$ expresses many standard operations on circuits.   After constructing this category we describe the behavior of circuits in two ways: first in terms of Dirichlet forms, 
which describe the power consumed by a circuit as a function of potentials and currents, and 
second in terms of Lagrangian subspaces, which describe the physically allowed potentials and
currents.  This gives two further decorated cospan categories, related by hypergraph functors
\[  
\xymatrix{
  \Circ  \ar[r]^-{A} 
  & \Dirich\Cospan \ar[r]^{B} 
  & \Lag\Cospan. 
}
\]

\subsection{Open circuits}
\label{ssec:open}

We defined a circuit with boundary in Definition \ref{def_circuit_2}; we now introduce `open circuits', with inputs and outputs, to serve as the morphisms of a category.  We fix a field $\F$ with a set of positive elements $\F^+$.

\begin{definition} 
An \define{open passive linear circuit}, or \define{open circuit} for short, 
is a cospan of finite sets $X
\stackrel{i}{\longrightarrow} N \stackrel{o}{\longleftarrow} Y$ together with a
passive linear circuit whose set of nodes is $N$.
\end{definition}

This suggests that open circuits should be morphisms in a decorated cospan category.
Indeed, we now prove that the map taking a finite set $N$ to the set of passive linear circuits 
with set $N$ of nodes is a lax symmetric monoidal functor. This allows us to apply
Lemma \ref{lemma:decorated_cospans} to construct a category of circuits.

\begin{lemma}
\label{lemma:Circuit}
Define the functor
\[
  \Circuit\maps (\Fin\Set,+) \longrightarrow (\mathrm{Set},\times)
\]
on objects to take any finite set $N$ to the set
$\Circuit(N)$ of passive linear circuits $(N,E,s,t,Z)$ with $N$ as their
set of nodes. On
morphisms let it take a function $f\maps N \to M$ to the function that pushes
passive linear circuit structures on a set $N$ forward onto the set $M$:
\begin{align*}
  \Circuit(f)\maps \Circuit(N) &\longrightarrow
  \Circuit(M) \\
  (N,E,s,t,Z) &\longmapsto (M,E,f \circ s, f \circ t, Z).
\end{align*}
We obtain a lax symmetric monoidal functor by equipping this functor with the
natural transformation 
\begin{align*}
  \rho_{N,M}\maps \Circuit(N) \times \Circuit(M)
  &\longrightarrow \Circuit(N+M) \\
  \big( (N,E,s,t,Z), (M,F,s',t',Z') \big) &\longmapsto
  \big(N+M,E+F,s+s',t+t',[Z,Z']\big),
\end{align*}
together with the unit map
\begin{align*}
  \rho_1\maps 1 &\longrightarrow \Circuit(\varnothing) \\
  \bullet &\longmapsto
  (\varnothing,\varnothing,\varnothing,\varnothing,\varnothing),
\end{align*}
where we use $\varnothing$ to denote both the empty set and the unique function
of the appropriate codomain with domain the empty set. 
\end{lemma}

\begin{proof}
As $\Circuit(f)$ simply acts by post-composition for each $f$,
$\Circuit$ is indeed functorial. The naturality of $\rho$, as well as
the coherence laws for lax symmetric monoidal functors, follow from the
universal property of the coproduct. 
\end{proof}

Making use of Lemmas \ref{lemma:decorated_cospans} and \ref{lemma:dccsarehypergraph}, we obtain a hypergraph category whose morphisms are open circuits:

\begin{definition}
   Define the category $\Circ$ to be the decorated cospan category $\Circuit\Cospan$.
\end{definition}

\begin{corollary}
The category $\Circ$ is a hypergraph category.
\end{corollary}

The different structures of the category $\Circ$ capture different operations that can
be performed with circuits.  Composition expresses the fact that we can
connect the outputs of one circuit to the inputs of the next, while the monoidal
structure models the placement of circuits side-by-side. The symmetric
monoidal structure lets us reorder input and output wires, and the hypergraph structure 
let us join or split and start or end wires:

\[
  \xymatrixrowsep{1pt}
  \xymatrixcolsep{30pt}
  \xymatrix{
    \mult{.075\textwidth} & \unit{.075\textwidth} &
    \comult{.075\textwidth} & \counit{.075\textwidth} 
    \\
    \mu_1 \colon 1+1 \to 1 & \eta_1 \colon 0 \to 1 &
    \delta_1 \colon 1 \to 1+1 & \epsilon_1 \colon 1 \to 0
  }
\]

\subsection{The Dirichlet cospan semantics}
\label{ssec:DirichCospan}

As shown in Sec.\  \ref{sec:generalized}, any passive linear circuit gives 
a Dirichlet form, its extended power functional.   We now construct a category where 
the morphisms are cospans of finite sets decorated by Dirichlet forms, and a functor 
from $\Circ$ to this category.   This provides our first semantics for open circuits.  

In what follows we fix a field $\F$ with a set of positive elements $\F^+$.
Consider a cospan of finite sets $X \to N \leftarrow Y$ together with a
Dirichlet form $Q_N$ on the apex $N$. We call this a \define{Dirichlet cospan}.
To compose such cospans, say when given another cospan $Y \to M \leftarrow Z$
decorated by Dirichlet form $Q_M$, we decorate the composite cospan $X \to
N+_YM \leftarrow Z$ with the Dirichlet form
\[
\begin{array}{rccl}
  \Big({j_N}_\ast Q_N+ {j_M}_\ast Q_M\Big)\maps & \F^{N+_YM} &\longrightarrow& \F \\
  &\phi &\longmapsto& Q_N(\phi\circ j_N)+Q_M(\phi\circ j_M),
\end{array}
\]
where $j_N$ and $j_M$ are the maps that include $N$ and $M$ into the pushout $N +_{Y} M$.
Interpreted in terms of extended power functionals, this says that the
power consumed by the interconnected circuit is the sum of the power consumed by each
part.   This is formalized as follows:

\begin{lemma}
\label{lemma:Dirich}
There exists a unique lax symmetric monoidal functor
\[
  (\Dirich, \delta) \maps (\Fin\Set,+) \longrightarrow (\mathrm{Set},\times)
\]
that is given as follows.   The functor $\Dirich$ maps any finite set $X$ to the set $\Dirich(X)$ 
of Dirichlet forms $Q\maps \F^X \to \F$ on $X$, and it maps any function $f\maps X \to Y$ 
between finite sets to the function
\begin{align*}
  \Dirich(f)\maps \Dirich(X) &\longrightarrow \Dirich(Y)\\
  Q &\longmapsto f_{\ast}Q
\end{align*}
where $(f_* Q)(\phi) = Q(\phi \circ f)$ for any $\phi \in \F^Y$.   To make $\Dirich$ lax
symmetric monoidal, we equip it with the natural transformation
\begin{align*}
  \delta_{N,M}\maps \Dirich(N) \times \Dirich(M) &\longrightarrow
  \Dirich(N+M) \\
  (Q_N,Q_M) &\longmapsto {j_N}_\ast Q_N + {j_M}_\ast Q_M
\end{align*}
and also the map
\begin{align*}
  \delta_1\maps 1 &\longrightarrow \Dirich(\varnothing)\\
  \bullet &\longmapsto 0.
\end{align*}
Here the sum of two Dirichlet forms is given pointwise by the addition in
$\F$, and $0$ denotes the unique Dirichlet form on the empty set.
\end{lemma}

\begin{proof}
  As composition of functions is associative and has identities,
  $\Dirich$ is a functor.  The naturality of the $\delta_{N,M}$ follows
  from the universal property of the coproduct in $\Fin\Set$, while the symmetric
  monoidal coherence axioms follow from the associativity, unitality, and
  commutativity of addition in $\F$.
\end{proof}

We thus obtain a decorated cospan category $\Dirich\Cospan$ for which a morphism 
is a Dirichlet cospan.  Next we construct a functor $A \maps \Circ \to \Dirich\Cospan$ 
sending any open circuit to a Dirichlet cospan whose Dirichlet form is the extended 
power functional of that circuit.  For this we need the following lemma.

\begin{lemma} 
\label{lemma:A}
  The collection of maps
\begin{align*}
  \alpha_N\maps \Circuit(N) &\longrightarrow \Dirich(N) \\
  (N,E,s,t,Z) &\longmapsto \left(\phi \in \F^N \mapsto \frac{1}{2} \sum_{e \in E}
  \frac{1}{Z(e)}\big(\phi(s(e))-\phi(t(e))\big)^2\right)
\end{align*}
defines a monoidal natural transformation
\[
  \alpha \maps (\Circuit,\rho) \Longrightarrow
  (\Dirich,\delta).
\]
\end{lemma}

\begin{proof}
Naturality requires that the square
\[
  \xymatrix{
    \Circuit(N) \ar[r]^{\alpha_N} \ar[d]_{\Circuit(f)} &
    \Dirich(N) \ar[d]^{\Dirich(f)}  \\
    \Circuit(M) \ar[r]_{\alpha_M} & \Dirich(M)
  }
\]
commutes. Let $(N,E,s,t,r)$ be an $\F^+$-graph on $N$ and $f\maps N \to M$ be a
function $N$ to $M$. Then both $\Dirich(f) \circ \alpha_N$ and $\alpha_M
\circ \Circuit(f)$ map $(N,E,s,t,r)$ to the Dirichlet form
\begin{align*}
  \F^M &\longrightarrow \F;\\
  \psi &\longmapsto \frac{1}{2} \sum_{e \in E}\frac{1}{Z(e)}
  \big(\psi(f(s(e)))-\psi(f(t(e)))\big)^2.
\end{align*}
Thus both methods of constructing an extended power functional on a set of nodes $M$ from
a circuit on $N$ and a function $f\maps N \to M$ produce the same power functional.

To show that $\alpha$ is a monoidal natural transformation, we must check that
\[
\begin{aligned}
\xymatrix{
  \Circuit(N) \times \Circuit(M) \ar[r]^{\alpha_N \times
  \alpha_M} \ar[d]_{\rho_{N,M}} & \Dirich(N) \times \Dirich(M)
  \ar[d]^{\delta_{N,M}}  \\
  \Circuit(N+M) \ar[r]^{\alpha_{N+M}} & \Dirich(N+M)
}
\end{aligned}
\quad \mbox{and} \quad
\begin{aligned}
\xymatrixcolsep{1pt}
\xymatrix{
  & 1 \ar[dl]_{\rho_\varnothing} \ar[dr]^{\delta_\varnothing}\\
\Circuit(\varnothing)  \ar[rr]^{\alpha_\varnothing} &&
\Dirich(\varnothing)
}
\end{aligned}
\]
commute. It is readily observed that both paths around the square lead to taking
two graphs and summing their corresponding Dirichlet forms, and that the
triangle commutes immediately as all objects in it are the one element set.
\end{proof}

\begin{theorem}
\label{thm:A}   
The lax symmetric monoidal functor $\Dirich \maps (\Fin\Set, +) \to (\Set,\times)$ 
in Lemma \ref{lemma:Dirich} defines 
a decorated cospan category $\Dirich\Cospan$, and the monoidal natural transformation 
$\alpha \maps \Circuit \Rightarrow \Dirich$ in Lemma \ref{lemma:A} 
defines a strict hypergraph functor
\[
  A \maps \Circ \longrightarrow \Dirich\Cospan.
\]
\end{theorem}

\begin{proof} 
The first part follows from Lemmas \ref{lemma:decorated_cospans} and \ref{lemma:Dirich}, 
while the second follows from Lemmas \ref{lemma:decoratedfunctors} and \ref{lemma:Dirich}.
\end{proof}

Note that $A$ is not a faithful functor.  For example, applying $A$ to a circuit 
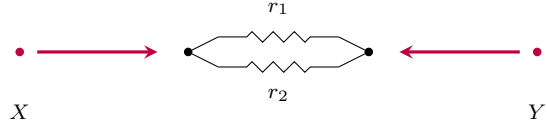
\[
\begin{tikzpicture}[circuit ee IEC, set resistor graphic=var resistor IEC
	graphic,scale=.8]
	\node[circle,draw,inner sep=1pt,fill=purple,color=purple]         (x) at
	(-2.8,0) {};
	\node at (-2.8,-1) {\footnotesize $X$};
	\node[circle,draw,inner sep=1pt,fill]         (A) at (0,0) {};
	\node[circle,draw,inner sep=1pt,fill]         (B) at (3,0) {};
	\node[circle,draw,inner sep=1pt,fill=purple,color=purple]         (y1) at
	(5.8,0) {};
	\node at (5.8,-1) {\footnotesize $Y$};
	\coordinate         (ua) at (.5,.25) {};
	\coordinate         (ub) at (2.5,.25) {};
	\coordinate         (la) at (.5,-.25) {};
	\coordinate         (lb) at (2.5,-.25) {};
	\path (A) edge (ua);
	\path (A) edge (la);
	\path (B) edge (ub);
	\path (B) edge (lb);
	\path (ua) edge  [resistor, circuit symbol unit=5pt, circuit symbol
	size=width {5} height 1.5] node[label={[label
	distance=1pt]90:{\footnotesize $r_1$}}] {} (ub);
	\path (la) edge  [resistor, circuit symbol unit=5pt, circuit symbol
	size=width {5} height 1.5] node[label={[label
	distance=1pt]270:{\footnotesize $r_2$}}] {} (lb);
	\path[color=purple, very thick, shorten >=10pt, shorten <=5pt, ->, >=stealth] (x) edge (A);
	\path[color=purple, very thick, shorten >=10pt, shorten <=5pt, ->, >=stealth] (y1) edge (B);
      \end{tikzpicture}
    \]
with two parallel edges of resistance $r_1$ and $r_2$ respectively, we obtain
the same result as for the circuit
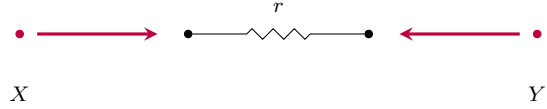
\[
\begin{tikzpicture}[circuit ee IEC, set resistor graphic=var resistor IEC
	graphic,scale=.8]
	\node[circle,draw,inner sep=1pt,fill=purple,color=purple]         (x) at
	(-2.8,0) {};
	\node at (-2.8,-1) {\footnotesize $X$};
	\node[circle,draw,inner sep=1pt,fill]         (A) at (0,0) {};
	\node[circle,draw,inner sep=1pt,fill]         (B) at (3,0) {};
	\node[circle,draw,inner sep=1pt,fill=purple,color=purple]         (y) at
	(5.8,0) {};
	\node at (5.8,-1) {\footnotesize $Y$};
	\path (A) edge  [resistor, circuit symbol unit=5pt, circuit symbol
	size=width {5} height 1.5] node[label={[label
	distance=1pt]90:{\footnotesize $r$}}] {} (B);
	\path[color=purple, very thick, shorten >=10pt, shorten <=5pt, ->, >=stealth] (x) edge (A);
	\path[color=purple, very thick, shorten >=10pt, shorten <=5pt, ->, >=stealth] (y) edge (B);
      \end{tikzpicture}
    \]
with just a single edge with resistance
\[
  r = \frac{1}{\tfrac{1}{r_1} + \tfrac{1}{r_2}}.
\]

\subsection{Lagrangian subspaces}
\label{ssec:lagrangian_subspaces}

In Section \ref{ssec:LagCospan} we introduce our next semantics for open circuits, which
is conceptually simpler than the Dirichlet cospan semantics.
This new semantics simply specifies the space of all physically allowed potential 
and current readings at all nodes of the circuit.    In Lemma \ref{lemma:qfls} we prove that 
this space is a Lagrangian subspace of the symplectic vector space $\vect{N}$, where $N$ is 
the set of nodes of the circuit.   To explain this, we begin with a review of symplectic vector 
spaces and their Lagrangian subspaces.

To keep this review brief we omit proofs of some standard results.  See any introduction to 
symplectic vector spaces, such as Cimasoni and Turaev \cite{CT} or Piccione and Tausk 
\cite{PT}, for details.

\begin{definition}
  Given a finite-dimensional vector space $V$ over a field $\F$, a 
  \define{symplectic form}
  $\omega\maps V \times V \to \F$ on $V$ is an alternating nondegenerate bilinear
  form.  That is, a symplectic form $\omega$ is a function $V \times V \to \F$
  that is
  \begin{enumerate}[(i)]
    \item bilinear: for all $\lambda \in \F$ and all $u,v,u',v' \in V$ we have
    \begin{enumerate}
    \item $\omega(\lambda u,v) = \omega(u,\lambda v) =  \lambda \omega(u,v)$,
    \item $\omega(u+u',v) = \omega(u,v)+\omega(u',v)$,
    \item $\omega(u,v+v') = \omega(u,v)+\omega(u,v')$;
    \end{enumerate}
    \item alternating: for all $v \in V$ we have $\omega(v,v) = 0$; and
    \item nondegenerate: given $v \in V$, $\omega(u,v) = 0$ for all $u \in V$ if
      and only if $v = 0$.
  \end{enumerate} 
  A \define{symplectic vector space} $(V,\omega)$ is a finite-dimensional 
  vector space $V$ equipped with a symplectic form $\omega$.  Given symplectic vector spaces 
  $(V_1,\omega_1), (V_2, \omega_2)$, a
  \define{symplectic map} is a linear map 
  \[
    f\maps (V_1,\omega_1) \longrightarrow (V_2, \omega_2)
  \]
  such that $\omega_2(f(u),f(v)) = \omega_1(u,v)$ for all $u,v \in V_1$. A
  \define{symplectomorphism} is a symplectic map that is also an isomorphism. 
\end{definition}

An alternating form is always \define{antisymmetric}, meaning that $\omega(u,v) = 
-\omega(v,u)$ for all $u,v \in V$.  The converse is true except in characteristic 2.
A \define{symplectic basis} for a symplectic vector space $(V,\omega)$ is a
basis $\{p_1,\dots,p_n,q_1,\dots,q_n\}$ such that $\omega(p_i,p_j) =
\omega(q_i,q_j) = 0$ for all $1 \le i,j \le n$, and $\omega(p_i,q_j) =
\delta_{ij}$ for all $1 \le i,j\le n$, where $\delta_{ij}$ is the Kronecker delta,
equal to $1$ when $i =j$, and $0$ otherwise.  Every symplectic vector space
has a symplectic basis.  A symplectomorphism maps symplectic
bases to symplectic bases, and conversely, any map that takes a symplectic basis
to another symplectic basis is a symplectomorphism.  

The key example is this:

\begin{example}[The symplectic vector space generated by a finite set]
  \label{ex:symplectic_space_generated_by_set}
  Given a finite set $N$, we equip the vector space $\vect{N}$ with the
  symplectic form 
  \[
    \omega\big((\phi,i),(\phi',i')\big) = i'(\phi)-i(\phi').  
  \] 
  Let $\{\phi_n\}_{n \in N}$ be the basis of $\F^N$ consisting of the functions
  $N \to \F$ mapping $n$ to $1$ and all other elements of $N$ to $0$, and let
  $\{i_n\}_{n \in N} \subseteq {(\F^N)}^\ast$ be the dual basis. Then
  $\{(\phi_n,0),(0,i_n)\}_{n\in N}$ forms a symplectic basis for $\vect{N}$, which
  we call the \define{standard symplectic basis}.
\end{example}

There are two common ways to build new symplectic spaces from old ones: 
conjugation and summation. Given a symplectic form $\omega$ on $V$, we
may define its \define{conjugate} symplectic form $\overline\omega = - \omega$,
and write the conjugate symplectic space $(V,\overline\omega)$ as $\overline V$.
Given two symplectic vector spaces $(U, \nu),(V,\omega)$, we consider their
direct sum $U \oplus V$ a symplectic vector space with the symplectic form
$(\nu+\omega)\big((u,v),(u',v')\big) := \nu(u,u')+\omega(v,v')$, and call this the \define{sum} of the two symplectic vector spaces. Note that this is neither the product nor coproduct in 
the category of symplectic vector spaces and symplectic maps.

The symplectic form provides a notion of orthogonal complement.  Given a subspace $S$ of 
$V$, we define its \define{complement}
\[
  S^\circ = \{v \in V \mid \omega(v,s) = 0 \textrm{ for all } s \in S\}.
\]
This construction obeys the following identities, where $S$ and $T$
are subspaces of $V$:
\begin{align*}
  \dim S+ \dim S^\circ &= \dim V \\
  (S^\circ)^\circ &= S \\
  (S + T)^\circ &= S^\circ \cap T^\circ \\
  (S \cap T)^\circ &= S^\circ + T^\circ.
\end{align*}

In the symplectic vector space $\vect{N}$, the subspace $\F^N$ has the property
of being a maximal subspace such that the symplectic form restricts to the zero
form on this subspace. Subspaces with this property are known as Lagrangian
subspaces, and they may all be realized as the image of $\F^N$ under
symplectomorphisms from $\vect{N}$ to itself.

\begin{definition} 
  Let $L$ be a linear subspace of a symplectic vector space $(V,\omega)$. We say 
  that $L$ is \define{isotropic} if $L \subseteq L^\circ$, or equivalently, $\omega(v,w) = 0$
  for all $v,w \in L$.
  A subspace is \define{Lagrangian} if it is a maximal isotropic subspace.
\end{definition}

Lagrangian subspaces are also known as Lagrangian correspondences or canonical
relations.    They have various characterizations:

\begin{lemma} 
\label{lemma:lagrangian_characterization} 
  Given a subspace $L \subseteq V$ of a symplectic vector space $(V,\omega)$, the
  following are equivalent: 
  \begin{enumerate}[(i)] 
    \item $L$ is Lagrangian.  
    \item $L = L^\circ$.
    \item $L$ is isotropic and $\dim L = \frac12 \dim V$.
  \end{enumerate} 
\end{lemma}

From this proposition it follows easily that the direct sum of two Lagrangian
subspaces is Lagrangian in the sum of their ambient spaces. 

We saw from Thm.\ \ref{thm:circuit_behavior_from_power} that a circuit's behavior
is determined by the differential of a quadratic form.   Now we show that there is a 
one-to-one correspondence between quadratic forms on $\F^N$ and Lagrangian subspaces of 
$\vect{N}$ obeying a certain property.

\begin{lemma} \label{lemma:qfls}
  Let $N$ be a finite set. Given a quadratic form $P$ on $\F^N$, the
  subspace 
  \[ 
    \mathrm{Graph}(dP) = \big\{(\phi,dP_\phi) \mid \phi \in \F^N\big\} \subseteq \vect{N},
  \] 
  where $dP_\phi \in {(\F^N)}^\ast$ is the formal differential of $P$ at $\phi
  \in \F^N$, is Lagrangian. Moreover, this construction gives a one-to-one correspondence 
  \[ 
    \left\{\begin{array}{c} 
      \\ \mbox{Quadratic forms over $\F$ on $N$} \\ \phantom{.}
    \end{array} \right\} 
    \longleftrightarrow
    \left\{\begin{array}{c} 
      \mbox{Lagrangian subspaces of $\vect{N}$}\\
      \mbox{with trivial intersection with} \\ 
      \{0\} \oplus {(\F^N)}^\ast \subseteq \vect{N} 
    \end{array} \right\}.  
  \]
\end{lemma}

\begin{proof}
  The symplectic structure on $\vect{N}$ and our notation for it is given in
  Example \ref{ex:symplectic_space_generated_by_set}; in particular we write
  $\phi_n$ for the basis element of $\F^N$ corresponding to $n \in N$. 
  
  Our first claim is that $\mathrm{Graph}(dP)$ is Lagrangian. Since the formal
  differential $dP$ is linear in $\phi$, $\mathrm{Graph}(dP)$ is a linear
  subspace.  Moreover, for all $n,m \in N$ we have $dP_{\phi_n}(\phi_m)
  =\frac{\partial^2 P}{\partial \phi_n \partial \phi_m} = dP_{\phi_m}(\phi_n)$.
  As $dP_{(-)}(-)$ is linear in both arguments, this implies that for all
  $\phi,\psi \in \F^N$ we have
  \[
    \omega\big((\phi,dP_\phi),(\psi,dP_\psi)\big) = dP_\psi(\phi) -
    dP_\phi(\psi) = 0,
  \]
  so $\mathrm{Graph}(dP)$ is indeed Lagrangian.  Note also that $dP_0 = 0$ for
  all quadratic forms $P$, so $\mathrm{Graph}(dP)$ has trivial intersection with
  the subspace $\{0\} \oplus {(\F^N)}^\ast$ of $\vect{N}$. Thus
  $\mathrm{Graph}(dP)$ defines a function from quadratic forms to Lagrangian
  subspaces with this trivial intersection property. 

  It remains to show that this function is one-to-one. To do this, we construct
  an inverse. Suppose that $L$ is a Lagrangian subspace of $\vect{N}$
  such that $L \cap (\{0\} \oplus {(\F^N)}^\ast) = \{(0,0)\}$. We claim that for
  each $\phi \in \F^N$, there exists a unique $i_\phi \in {(\F^N)}^\ast$ such
  that $(\phi,i_\phi) \in L$, and that setting $P_L(\phi)=i_\phi(\phi)$ defines
  a quadratic form on $N$.

  Existence of such an $i_\phi$ is given by counting dimensions. By assumption,
  $i_\phi=0 \in (\F^N)^\ast$ is the unique such element for $\phi=0$. Fix some
  nonzero $\phi \in \F^N$, and write $\langle \phi\rangle$ for the subspace
  spanned by this vector. Then $L$ and $\langle \phi \rangle \oplus (\F^N)^\ast$
  are respectively $N$ and $N+1$ dimensional subspaces of the $2N$ dimensional
  vector space $\vect{N}$, and hence must intersect nontrivially.  Since $L$ and
  $\{0\}\oplus (\F^N)^\ast$ intersect trivially, this means there is a point of
  the form $(\phi,i_\phi)$ in the intersection, and hence in $L$.  For
  uniqueness, suppose we have $(\phi,i_\phi)$ and $(\phi,i_\phi')$ in $L$.  By
  linearity then $(0,i_\phi-i_\phi') \in L$, so $i_\phi=i_\phi'$. We thus can
  define a function from $\F^N $ to  $(\F^N)^\ast$ 
  sending $\phi$ to $i_\phi$.  This is linear, so it defines a bilinear map $B(\phi,\psi) = i_\phi(\psi)$ 
  on $\F^N \oplus \F^N$, and thus a quadratic form $P_L(\phi) = i_\phi(\phi)$ on $N$. 

  Finally, to check these constructions are inverses, we must check that
  $L(\phi) = dL_\phi(\phi)$, and that $d(P_L)_\phi =i_\phi$ where $(\phi,i_\phi)
  \in L$. These are straightforward computations.
\end{proof}

\subsection{Lagrangian relations}
\label{ssec:LagRel}

To develop our next semantics for open circuits we need Lagrangian relations.  
These also play an important role in black-boxing in Part \ref{part:blackbox}.   
Lagrangian relations give a way to think of certain Lagrangian subspaces, such 
as those arising from circuits, as morphisms in a category $\Lag\Rel$.   Here 
we construct this category.

Recall that a relation from the set $X$ to the set $Y$ is a subset $R$ of their product
$X \times Y$.  Given relations $R \subseteq X \times Y$ and $S
\subseteq Y \times Z$, there is a composite relation $S \circ R \subseteq X
\times Z$ given by pairs $(x,z)$ such that there exists $y \in Y$ with $(x,y)
\in R$ and $(y,z) \in S$---a direct generalization of function composition.

\begin{definition}
  Given symplectic vector spaces $V_1$ and $V_2$, a \define{Lagrangian relation} 
  $L\maps V_1 \asrelto V_2$ is a subset $L \subseteq V_1 \times V_2$ that is a Lagrangian
  subspace of $\overline{V_1} \oplus V_2$. 
\end{definition}

This is a generalization of the notion of symplectomorphism: the graph of
any symplectomorphism $f\maps V_1 \to V_2$ is a Lagrangian subspace 
$\mathrm{Graph}(f) \subseteq \overline{V_1} \oplus V_2$. More
generally, the graph of any symplectic map $f\maps V_1 \to V_2$ gives 
an isotropic subspace of $\overline{V_1} \oplus V_2$. 

The composite of two Lagrangian relations is again Lagrangian. 

\begin{lemma} \label{lemma:lagrangian_composition}
  Let $L\maps V_1 \asrelto V_2$ and $L'\maps V_2 \asrelto V_3$ be Lagrangian relations. 
  Then the composite relation $L' \circ L$ is a Lagrangian relation $V_1 \asrelto V_3$.
\end{lemma}

\begin{proof}
This is well known \cite{Weinstein}, but a self-contained proof can be found, e.g.,
in \cite[Prop.\ 6.40]{Fong16}.
\end{proof}

Composition of relations, and thus Lagrangian relations, is associative. 
Lagrangian relations also provide the identity morphisms
that were missing for Dirichlet forms in Section \ref{subsec:dirichlet_composition}.  
Namely, given a symplectic vector 
space $V$, the Lagrangian relation $\id_V \maps V \asrelto V$
\[
  \id_V = \{(v,v) \mid v \in V\} \subseteq \overline{V} \oplus V
\]
acts as an identity for composition of relations.  We thus have a category.
However, for technical reasons, it will be useful to work with an equivalent
subcategory:

\begin{definition}
Let $\Lag\Rel$ be the category with finite sets as objects and
Lagrangian relations $L \maps \vect{X} \asrelto \vect{Y}$ as morphisms from the
finite set $X$ to the finite set $Y$.
\end{definition}

In Thm.\ \ref{thm:LagCorel=LagRel} we describe an isomorphism between $\Lag\Rel$ 
and a category $\Lag\Corel$, which is a hypergraph
category.    This equips $\Lag\Rel$ itself with the structure of a hypergraph category.
In particular, $\Lag\Rel$ is symmetric monoidal with the tensor 
product of objects given by a chosen coproduct in $\Fin\Set$, and the tensor
product of morphisms given by the direct sum of Lagrangian relations.

\subsection{Symplectification}
\label{ssec:symplectification}

As the final technical step toward our second semantics for open circuits, we now
construct the symplectification functor $S \maps \Fin\Set \to \Lag\Rel$.    This
is a mathematically natural process, but as we shall see, it also expresses an important rule in 
circuit theory: ``when two wires join, set their potentials equal and add their currents''.

Suppose $f \maps X \to Y$ is a function between finite
sets.   We define the pullback map $f^\ast$ by
  \begin{align*}
    f^\ast\maps \F^Y &\longrightarrow \F^X \\
    \phi &\longmapsto \phi \circ f,
  \end{align*}
and the pushforward map $f_\ast$ by
  \begin{align*}
    f_\ast\maps (\F^X)^\ast &\longrightarrow (\F^Y)^\ast \\
    i &\longmapsto i(-\circ f).
  \end{align*}
Pullback defines a contravariant functor from $\Fin\Set$ to the category of finite-dimensional
vector spaces, while pushforward defines a covariant functor.  Note also that if $\phi \in 
\F^Y$ and $i \in (\F^X)^\ast$ we have
\[     i( f^\ast \phi) = (f_\ast i) (\phi) .\]

\begin{proposition}
\label{prop:symplectification}
There is a strong symmetric monoidal 
functor $S \maps (\Fin\Set,+) \to (\Lag\Rel, \oplus)$ that maps any finite set $X$ to itself,
and maps any function $f \maps X \to Y$ between finite sets to the Lagrangian relation
$S(f) \maps \vect{X} \asrelto \vect{Y}$ given by
\[     S(f) = \big\{ (f^\ast \phi, i, \phi, f_\ast i) \, \big \vert \; 
\phi \in \F^Y, i \in (\F^X)^\ast \big\} .\]
\end{proposition}

\begin{proof}
First we show that $S(f)$ is Lagrangian.  Its dimension is half that of $\overline{\vect{X}} \oplus
\vect{Y}$, so by Lemma \ref{lemma:lagrangian_characterization} is suffices to show that
$S(f)$ is isotropic.   To check this, choose two vectors in $S(f)$, say $v = (f^\ast \phi, i, \phi, f_\ast i)$ and $w = (f^\ast \phi', i', \phi', f_\ast i')$, and note that if $\omega$ is the symplectic
structure on $S(f)$, then
\[ 
 \omega(v,w) = -i'(f^\ast \phi) + i(f^\ast \phi') + (f_\ast i')(\phi) - (f_\ast i)(\phi') 
                     = 0 .
\]   

Next we show that $S$ is a functor.   Suppose we have functions between finite sets
$f \maps X \to Y$ and $g \maps Y \to Z$.   Then $S(f)$ is given as above, and
\[     S(g) = \big\{ (g^\ast \phi', i', \phi', g_\ast i') \, \big \vert \; 
\phi' \in \F^Z, i' \in (\F^Y)^\ast \big\} \]
so 
\[   \begin{array}{ccl}
S(g)S(f) &=& 
\big\{ (f^\ast \phi, i, \phi', g_\ast i') \, \big \vert \; 
\phi' \in \F^Z, i \in (\F^X)^\ast, \phi = g^\ast \phi',  f_\ast i = i' \big\} \\
&=& 
\big\{ (f^\ast g^\ast \phi', i, \phi', g_\ast f_\ast i) \, \big \vert \; 
\phi' \in \F^Z, i \in (\F^X)^\ast \big\} \\
&=&
\big\{ ((gf)^\ast \phi', i, \phi', (gf)_\ast i) \, \big \vert \; 
\phi' \in \F^Z, i \in (\F^X)^\ast \big\} \\
&=& S(gf) .   
\end{array}
\]
One can also check that $S$ preserves identities and that $S$ becomes
strong symmetric monoidal using the obvious natural isomorphism $\vect{X + Y} \cong
\vect{X} \oplus \vect{Y}$.
\end{proof}

\begin{example}
Suppose $X = \{1,2\}$, $Y = \{3\}$ and $f \maps X \to Y$ is the unique function.  Using
the standard symplectic bases to identify $S(X)$ with $\F^4 = \{(\phi_1, \phi_2, i_1, i_2)\}$ and
$S(Y)$ with $\F^2 = \{(\phi_3, i_3)\}$, a calculation shows that
\[    S(f) = \{(\phi_1,\phi_2,i_1,i_2,\phi_3,i_3) \big\vert \; \phi_1 = \phi_2 = \phi_3, 
i_1 + i_2 = i_3 \} .\]
This corresponds to the fact that when two perfectly conductive wires merge into one, 
Kirchoff's current law says the currents $i_1,i_2$ on the incoming wires sum to give the 
current $i_3$ on the outgoing wire, while the potentials on all three wires are equal.
\end{example}

\subsection{The Lagrangian cospan semantics}
\label{ssec:LagCospan}

We now present a second semantics for open circuits.
The first, in Thm.~\ref{thm:A}, was given by a functor $A \maps \Circ \to\Dirich\Cospan$
mapping any open circuit to the Dirichlet cospan describing its extended power functional.
In the second, we convert this Dirichlet cospan into a `Lagrangian cospan' using Lemma
\ref{lemma:qfls}, which relates Dirichlet forms and Lagrangian subspaces.   We start 
by constructing a category of Lagrangian cospans, $\Lag\Cospan$.   Then we construct a 
functor $B \maps \Dirich\Cospan \to \Lag\Cospan$.   Our second semantics for open circuits 
is then given by the composite 
\[
\xymatrix{
  \Circ  \ar[r]^-{A} 
  & \Dirich\Cospan \ar[r]^-{B} 
  & \Lag\Cospan .   \\
}
\]

We construct $\Lag\Cospan$ as a decorated cospan category using a lax symmetric monoidal functor $\Lag \maps (\Fin\Set,+) \to (\Set,\times)$.    On objects, this functor simply maps any 
finite set $X$ to the set of all Lagrangian subspaces of $\vect{X}$.  It will be useful
to regard these as Lagrangian relations $L \maps \{0\} \asrelto \vect{X}$. 

\begin{lemma}
\label{lemma:Lag}
There exists a unique lax symmetric monoidal functor
\[
  (\Lag,\lambda) \maps (\Fin\Set,+) \longrightarrow (\mathrm{Set},\times)
\]
that is given as follows.  The functor $\Lag$ maps any finite set $X$ to the set of Lagrangian
relations $\{0\} \asrelto \vect{X}$, and it maps any function $f \maps X \to Y$ between
finite sets to the function
\begin{align*}
  \Lag(f) \maps \Lag(X) &\longrightarrow \Lag(Y) \\
   L &\longmapsto S(f) \circ L.
\end{align*}  
To make $\Lag$ lax symmetric monoidal, we equip it with the natural transformation
\begin{align*}
  \lambda_{N,M}\maps \Lag(N) \times \Lag(M) &\longrightarrow
  \Lag(N+M)\\
  (L_N,L_M) &\longmapsto L_N \oplus L_M
\end{align*}
and also the map
\begin{align*}
  \lambda_1\maps 1 &\longrightarrow \Lag(\varnothing)\\
  \bullet &\longmapsto \{0\}.
\end{align*}
\end{lemma}
\begin{proof}
The functoriality of $\Lag$ follows from Prop.~\ref{prop:symplectification},
which says that symplectification is a functor $S \maps \Fin\Set \to \Lag\Rel$. 
The lax symmetric monoidality follows from properties of the direct sum of vector spaces.
\end{proof}

Using the theory of decorated cospans, we thus obtain a hypergraph category
$\Lag\Cospan$ where a morphism is a cospan of finite sets whose apex $N$
is equipped with a Lagrangian subspace of $\vect{N}$.   Next, we use theory of
decorated cospans to construct a hypergraph functor $R \maps \Dirich\Cospan
\to \Lag\Cospan$.

\begin{lemma}
\label{lemma:B}
Let
\[
  \beta\maps (\Dirich,\delta) \Longrightarrow (\Lag,\lambda)
\]
be the collection of functions
\begin{align*}
  \beta_N\maps \Dirich(N) &\longrightarrow \Lag(N) \\
  Q &\longmapsto \{(\phi,dQ_\phi) \mid \phi \in \F^N\} \subseteq \vect{N}.
\end{align*}
Then $\beta$ is a monoidal natural transformation.
\end{lemma}

\begin{proof}
Naturality requires that the square
\[
\xymatrix{
  \Dirich(N) \ar[r]^{\beta_N} \ar[d]_{\Dirich(f)} &
  \Lag(N) \ar[d]^{\Lag(f)}  \\
  \Dirich(M) \ar[r]_{\beta_M} & \Lag(M)
}
\]
commutes for every function $f\maps N \to M$. This is primarily a consequence of the
fact that the differential commutes with pullbacks.  Recall from Lemma \ref{lemma:Dirich}
that $\Dirich(f)$ maps any Dirichlet form $Q$ on $N$ to the form $f_\ast Q$,
and $\beta_M$ in turn maps this to the Lagrangian subspace 
\[
  \big\{(\psi,d(f_\ast Q)_\psi) \, \big\vert \, \psi \in \F^M\big\} \subseteq
  \F^M \oplus (\F^M)^\ast.
\]
On the other hand, $\beta_N$ maps a Dirichlet form $Q$ on $N$ to the Lagrangian
subspace
\[
\big\{(\phi,dQ_\phi) \,\big\vert\, \phi \in \F^N\big\}\subseteq
  \F^N \oplus (\F^N)^\ast, 
\]
before $\Lag(f)$ maps this to the Lagrangian subspace
\[
  \big\{(\psi, f_\ast dQ_\phi) \,\big\vert\, \psi \in \F^M, \phi =
  f^\ast(\psi)\big\} \subseteq \F^M\oplus (\F^M)^\ast.
\]
But $f_\ast dQ_{f^\ast\psi} = d(f_\ast Q)_{\psi}$, so these two processes
commute.

For $\beta$ to be monoidal, the diagrams 
\[
\begin{aligned}
\xymatrixcolsep{30pt}
\xymatrix{
  \Dirich(N) \times \Dirich(M) \ar[r]^(.55){\beta_N \times
  \beta_M} \ar[d]_{\delta_{N,M}} & \Lag(N) \times \Lag(M)
  \ar[d]^{\lambda_{N,M}}  \\
  \Dirich(N+M) \ar[r]_{\beta_{N+M}} & \Lag(N+M)
}
\end{aligned}
\quad
\mbox{and}
\quad
\begin{aligned}
\xymatrixcolsep{5pt}
  \xymatrix{
  & 1 \ar[dl]_{\delta_\varnothing} \ar[dr]^{\lambda_\varnothing}\\
\Dirich(\varnothing)  \ar[rr]_{\beta_\varnothing} &&
\Lag(\varnothing)
}
\end{aligned}
\]
must commute. They do: the Lagrangian subspace corresponding to the sum of Dirichlet
forms is equal to the sum of the Lagrangian subspaces that correspond to the
summand Dirichlet forms, while there is only a unique map $1 \to
\Lag(\varnothing)$.
\end{proof}

Summarizing, we obtain a functor $B$ sending any Dirichlet cospan to a Lagrangian cospan:

\begin{theorem}
\label{thm:R}   
The lax symmetric monoidal functor $\Lag \maps (\Fin\Set, +) \to (\Set,\times)$ 
in Lemma \ref{lemma:Lag} defines 
a decorated cospan category $\Lag\Cospan$, and the monoidal natural transformation 
$\beta \maps \Dirich \Rightarrow \Lag$ in Lemma \ref{lemma:B} 
defines a strict hypergraph functor
\[
  B\maps \Dirich\Cospan \longrightarrow \Lag\Cospan.
\]
\end{theorem}

\begin{proof} 
The first part follows from Lemma \ref{lemma:decorated_cospans},
while the second follows from Lemma \ref{lemma:decoratedfunctors}. 
\end{proof}

\part{The Black Box Functor} 
\label{part:blackbox}
In Part \ref{part:categories} we introduced a category $\Circ$ whose morphisms are open circuits,
and a category $\Lag\Rel$ whose morphisms are Lagrangian relations.   In this part we 
prove the main result of the paper, Thm.\ \ref{thm:main}, by constructing a functor that
relates these categories:
\[  
\blacksquare\maps \Circ \to \Lag\Rel .
\]
This is called the `black box functor' because it forgets the internal structure of
an open circuit and remembers only the relation it establishes between its input and 
output potentials and currents.   

The black box functor can be obtained in two ways, thanks to this commutative diagram:
\[
\xymatrix{
  \Circ = \Circuit\Cospan \ar[r]^-{A} 
  & \Dirich\Cospan \ar[r]^{B} \ar[d]_{\square_\Dirich} 
  & \Lag\Cospan \ar[d]^{\square_\Lag}    \\
  & \Dirich\Corel \ar[r]^{\widetilde{B}} 
  & \Lag\Corel \ar[r]_-{\cong}^-{C} 
  & \Lag\Rel.
}
\]
We have already seen the top line of this diagram, which involves decorated cospans. 
In Section \ref{ssec:DirichCospan} we constructed the functor $A$ mapping any open circuit
to a cospan of finite sets decorated with a Dirichlet form specifying its extended power
functional.     In Section \ref{ssec:LagCospan} we constructed a functor $B$ that converts
this Dirichlet form into a Lagrangian subspace.

The second line of the diagram involves `decorated corelations'.   
Decorated cospans describe the internal structure of a circuit, but decorated corelations
do not: they only describe its externally observable behavior.
A morphism in $\Dirich\Corel$ describes the power consumed for a given choice of input 
and output potentials.   A morphism in $\Lag\Corel$ describes the relation between
input and output potentials and currents.    Indeed, we shall prove that $\Lag\Corel$ is 
isomorphic to our earlier category $\Lag\Rel$.   Following either route in the diagram from $\Circ$ to
$\Lag\Rel$ gives the black box functor.

We begin in Section \ref{ssec:corelations} by describing corelations and how they describe
circuits made of ideal perfectly conductive wires.   In Section \ref{ssec:corelation_categories}
we recall the general theory of corelations and show that $\Fin\Corel$, the
category of finite sets and corelations between these, is a hypergraph category.  In Section \ref{ssec:deccorel} we explain decorated corelations.   In Section \ref{ssec:ideal_wires} we 
extend the previously defined symplectification functor from $\Fin\Set$ to $\Fin\Corel$.  This 
sets the stage for constructing the functor $\square_\Lag \maps
\Lag\Cospan \to \Lag\Corel$ in Section \ref{ssec:LagCorel}.    In this section we also
give the isomorphism $C \maps \Lag\Corel \to \Lag\Rel$ that completes the construction of the 
black box functor.    Section \ref{sec:main} clarifies the meaning of the black box functor
by building the commutative square above.  We then use this square to prove our main result, Thm.\ \ref{thm:main}.

\section{Decorated corelations} 
\label{sec:deccorel}
In this section we shall see that circuits made of ideal perfectly conductive 
wires are modelled by `corelations'.   We then observe
that Kirchhoff's laws follow directly from interpreting these structures in the
category of linear relations.

\subsection{Circuits of ideal wires as corelations}
\label{ssec:corelations}

In the category of sets we hold the fundamental relationship between sets to be
that of functions. These encode the idea of a deterministic process that takes
each element of one set to a unique element of the other. For the study of
networks this is less appropriate, as the relationship between terminals is not
an input-output one, but rather one of interconnection. Willems has repeatedly
emphasised the prevalence of input-output thinking as a limitation of current
techniques in control theory \cite{Wi,Wi2}.  

In particular, the direction of a function becomes irrelevant, and to describe
interconnections via the category of sets we must develop an understanding
of how to compose functions head to head and tail to tail. We have so far used
cospans and pushouts to address this.  Cospans, however, come with an apex, which
represents extraneous structure beyond the two sets we wish to specify a
relationship between.  Corelations arise from omitting this information.

\begin{definition}
\label{def:corelation}
  A \define{corelation} from a set $X$ to a set $Y$ is an equivalence
  relation on the disjoint union $X+Y$.
\end{definition}

To motivate the use of this category, let us start with a set of input
terminals $X$ and a set of output terminals $Y$.  We may connect these
terminals with ideal wires of zero impedance, whichever way we like---input to
input, output to output, input to output---producing something like this:
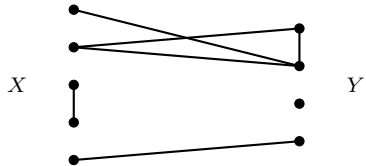
\[
  \begin{tikzpicture}[circuit ee IEC]
	\begin{pgfonlayer}{nodelayer}
		\node [contact] (0) at (-2, 1) {};
		\node [contact] (1) at (-2, 0.5) {};
		\node [contact] (2) at (-2, -0) {};
		\node [contact] (3) at (-2, -0.5) {};
		\node [contact] (4) at (-2, -1) {};
		\node [contact] (5) at (1, 0.75) {};
		\node [contact] (6) at (1, 0.25) {};
		\node [contact] (7) at (1, -0.25) {};
		\node [contact] (8) at (1, -0.75) {};
		\node [style=none] (9) at (-2.75, -0) {$X$};
		\node [style=none] (10) at (1.75, -0) {$Y$};
	\end{pgfonlayer}
	\begin{pgfonlayer}{edgelayer}
		\draw [thick] (5.center) to (1.center);
		\draw [thick] (6.center) to (1.center);
		\draw [thick] (3.center) to (2.center);
		\draw [thick] (4.center) to (8.center);
		\draw [thick] (5.center) to (6.center);
		\draw [thick] (6.center) to (0.center);
	\end{pgfonlayer}
\end{tikzpicture}
\]
In doing so, we introduce a notion of equivalence on our terminals, where two 
terminals are equivalent if electrons can move from one to 
another via some sequence of wires.   The connected components of the circuit
thus give a partition of the $X+Y$, transforming the above picture into a corelation
from $X$ to $Y$:
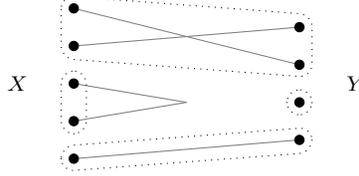
\[
  \begin{tikzpicture}[circuit ee IEC]
	\begin{pgfonlayer}{nodelayer}
		\node [contact, outer sep=5pt] (0) at (-2, 1) {};
		\node [contact, outer sep=5pt] (1) at (-2, 0.5) {};
		\node [contact, outer sep=5pt] (2) at (-2, -0) {};
		\node [contact, outer sep=5pt] (3) at (-2, -0.5) {};
		\node [contact, outer sep=5pt] (4) at (-2, -1) {};
		\node [contact, outer sep=5pt] (5) at (1, 0.75) {};
		\node [contact, outer sep=5pt] (6) at (1, 0.25) {};
		\node [contact, outer sep=5pt] (7) at (1, -0.25) {};
		\node [contact, outer sep=5pt] (8) at (1, -0.75) {};
		\node [style=none] (9) at (-2.75, -0) {$X$};
		\node [style=none] (10) at (1.75, -0) {$Y$};
		\node [style=none] (11) at (-0.5, 0.625) {};
		\node [style=none] (12) at (-0.5, -0.25) {};
		\node [style=none] (13) at (-0.5, -0.875) {};
	\end{pgfonlayer}
	\begin{pgfonlayer}{edgelayer}
		\draw [color=gray] (0.center) to (11.center);
		\draw [color=gray] (1.center) to (11.center);
		\draw [color=gray] (5.center) to (11.center);
		\draw [color=gray] (6.center) to (11.center);
		\draw [color=gray] (2.center) to (12.center);
		\draw [color=gray] (12.center) to (3.center);
		\draw [color=gray] (4.center) to (13.center);
		\draw [color=gray] (13.center) to (8.center);
		\draw [rounded corners=5pt, dotted] 
   (node cs:name=0, anchor=north west) --
   (node cs:name=1, anchor=south west) --
   (node cs:name=6, anchor=south east) --
   (node cs:name=5, anchor=north east) --
   cycle;
		\draw [rounded corners=5pt, dotted] 
   (node cs:name=2, anchor=north west) --
   (node cs:name=3, anchor=south west) --
   (node cs:name=3, anchor=south east) --
   (node cs:name=2, anchor=north east) --
   cycle;
		\draw [rounded corners=5pt, dotted] 
   (node cs:name=4, anchor=north west) --
   (node cs:name=4, anchor=south west) --
   (node cs:name=8, anchor=south east) --
   (node cs:name=8, anchor=north east) --
   cycle;
		\draw [rounded corners=5pt, dotted] 
   (node cs:name=7, anchor=north west) --
   (node cs:name=7, anchor=south west) --
   (node cs:name=7, anchor=south east) --
   (node cs:name=7, anchor=north east) --
   cycle;
	\end{pgfonlayer}
\end{tikzpicture}
\]
The dotted lines indicate equivalence classes of points, while for reference the
grey lines indicate wires connecting these points.

Given another circuit of this sort, say from $Y$ to $Z$, we may combine
these circuits to create a circuit  from $X$ to $Z$:
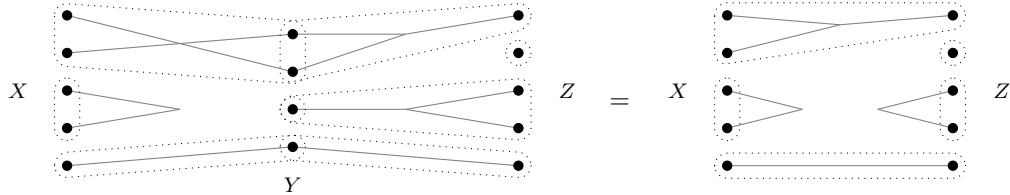
\[
  \begin{aligned}
\begin{tikzpicture}[circuit ee IEC]
	\begin{pgfonlayer}{nodelayer}
		\node [contact, outer sep=5pt] (0) at (1, 0.75) {};
		\node [contact, outer sep=5pt] (1) at (1, 0.25) {};
		\node [contact, outer sep=5pt] (2) at (1, -0.25) {};
		\node [contact, outer sep=5pt] (3) at (1, -0.75) {};
		\node [style=none] (4) at (-2.65, -0) {$X$};
		\node [style=none] (5) at (4.65, -0) {$Z$};
		\node [contact, outer sep=5pt] (6) at (-2, 1) {};
		\node [contact, outer sep=5pt] (7) at (-2, -0.5) {};
		\node [contact, outer sep=5pt] (8) at (-2, 0.5) {};
		\node [contact, outer sep=5pt] (9) at (-2, -0) {};
		\node [contact, outer sep=5pt] (10) at (-2, -1) {};
		\node [contact, outer sep=5pt] (11) at (4, -0) {};
		\node [contact, outer sep=5pt] (12) at (4, -1) {};
		\node [contact, outer sep=5pt] (13) at (4, -0.5) {};
		\node [contact, outer sep=5pt] (14) at (4, 0.5) {};
		\node [style=none] (15) at (-0.5, 0.625) {};
		\node [style=none] (16) at (-0.5, -0.25) {};
		\node [style=none] (17) at (-0.5, -0.875) {};
		\node [style=none] (18) at (2.5, -0.875) {};
		\node [contact, outer sep=5pt] (19) at (4, 1) {};
		\node [style=none] (20) at (1, -1.25) {$Y$};
		\node [style=none] (21) at (2.5, 0.75) {};
		\node [style=none] (22) at (2.5, -0.25) {};
	\end{pgfonlayer}
	\begin{pgfonlayer}{edgelayer}
		\draw [color=gray] (6.center) to (15.center);
		\draw [color=gray] (8.center) to (15.center);
		\draw [color=gray] (0.center) to (15.center);
		\draw [color=gray] (1.center) to (15.center);
		\draw [color=gray] (9.center) to (16.center);
		\draw [color=gray] (7.center) to (16.center);
		\draw [color=gray] (10.center) to (17.center);
		\draw [color=gray] (17.center) to (3.center);
		\draw [color=gray] (3.center) to (18.center);
		\draw [color=gray] (18.center) to (12.center);
		\draw [color=gray] (0.center) to (21.center);
		\draw [color=gray] (1.center) to (21.center);
		\draw [color=gray] (21.center) to (19.center);
		\draw [color=gray] (2.center) to (22.center);
		\draw [color=gray] (22.center) to (11.center);
		\draw [color=gray] (22.center) to (13.center);
		\draw [rounded corners=5pt, dotted] 
   (node cs:name=6, anchor=north west) --
   (node cs:name=8, anchor=south west) --
   (node cs:name=1, anchor=south east) --
   (node cs:name=0, anchor=north east) --
   cycle;
		\draw [rounded corners=5pt, dotted] 
   (node cs:name=9, anchor=north west) --
   (node cs:name=7, anchor=south west) --
   (node cs:name=7, anchor=south east) --
   (node cs:name=9, anchor=north east) --
   cycle;
		\draw [rounded corners=5pt, dotted] 
   (node cs:name=10, anchor=north west) --
   (node cs:name=10, anchor=south west) --
   (node cs:name=3, anchor=south east) --
   (node cs:name=3, anchor=north east) --
   cycle;
		\draw [rounded corners=5pt, dotted] 
   (node cs:name=2, anchor=north west) --
   (node cs:name=2, anchor=south west) --
   (node cs:name=2, anchor=south east) --
   (node cs:name=2, anchor=north east) --
   cycle;
		\draw [rounded corners=5pt, dotted] 
   (node cs:name=0, anchor=north west) --
   (node cs:name=1, anchor=south west) --
   (node cs:name=19, anchor=south east) --
   (node cs:name=19, anchor=north east) --
   cycle;
		\draw [rounded corners=5pt, dotted] 
   (node cs:name=2, anchor=north west) --
   (node cs:name=2, anchor=south west) --
   (node cs:name=13, anchor=south east) --
   (node cs:name=11, anchor=north east) --
   cycle;
		\draw [rounded corners=5pt, dotted] 
   (node cs:name=3, anchor=north west) --
   (node cs:name=3, anchor=south west) --
   (node cs:name=12, anchor=south east) --
   (node cs:name=12, anchor=north east) --
   cycle;
		\draw [rounded corners=5pt, dotted] 
   (node cs:name=14, anchor=north west) --
   (node cs:name=14, anchor=south west) --
   (node cs:name=14, anchor=south east) --
   (node cs:name=14, anchor=north east) --
   cycle;
	\end{pgfonlayer}
\end{tikzpicture}
\end{aligned}
\quad
  =
\quad
\begin{aligned}
\begin{tikzpicture}[circuit ee IEC]
	\begin{pgfonlayer}{nodelayer}
		\node [style=none] (0) at (-2.65, -0) {$X$};
		\node [style=none] (1) at (1.65, -0) {$Z$};
		\node [style=none] (20) at (1, -1.25) {$\phantom{Y}$};
		\node [contact, outer sep=5pt] (2) at (-2, 1) {};
		\node [contact, outer sep=5pt] (3) at (-2, -0.5) {};
		\node [contact, outer sep=5pt] (4) at (-2, 0.5) {};
		\node [contact, outer sep=5pt] (5) at (-2, -0) {};
		\node [contact, outer sep=5pt] (6) at (-2, -1) {};
		\node [contact, outer sep=5pt] (7) at (1, -0) {};
		\node [contact, outer sep=5pt] (8) at (1, -1) {};
		\node [contact, outer sep=5pt] (9) at (1, -0.5) {};
		\node [contact, outer sep=5pt] (10) at (1, 0.5) {};
		\node [style=none] (11) at (-0.5, 0.875) {};
		\node [style=none] (12) at (-1, -0.25) {};
		\node [contact, outer sep=5pt] (13) at (1, 1) {};
		\node [style=none] (14) at (0, -0.25) {};
	\end{pgfonlayer}
	\begin{pgfonlayer}{edgelayer}
		\draw [color=gray] (2.center) to (11.center);
		\draw [color=gray] (4.center) to (11.center);
		\draw [color=gray] (5.center) to (12.center);
		\draw [color=gray] (3.center) to (12.center);
		\draw [color=gray] (14.center) to (7.center);
		\draw [color=gray] (14.center) to (9.center);
		\draw [color=gray] (6.center) to (8.center);
		\draw [color=gray] (11.center) to (13.center);
		\draw [rounded corners=5pt, dotted] 
   (node cs:name=2, anchor=north west) --
   (node cs:name=4, anchor=south west) --
   (node cs:name=13, anchor=south east) --
   (node cs:name=13, anchor=north east) --
   cycle;
		\draw [rounded corners=5pt, dotted] 
   (node cs:name=5, anchor=north west) --
   (node cs:name=3, anchor=south west) --
   (node cs:name=3, anchor=south east) --
   (node cs:name=5, anchor=north east) --
   cycle;
		\draw [rounded corners=5pt, dotted] 
   (node cs:name=6, anchor=north west) --
   (node cs:name=6, anchor=south west) --
   (node cs:name=8, anchor=south east) --
   (node cs:name=8, anchor=north east) --
   cycle;
		\draw [rounded corners=5pt, dotted] 
   (node cs:name=10, anchor=north west) --
   (node cs:name=10, anchor=south west) --
   (node cs:name=10, anchor=south east) --
   (node cs:name=10, anchor=north east) --
   cycle;
		\draw [rounded corners=5pt, dotted] 
   (node cs:name=7, anchor=north west) --
   (node cs:name=9, anchor=south west) --
   (node cs:name=9, anchor=south east) --
   (node cs:name=7, anchor=north east) --
   cycle;
	\end{pgfonlayer}
\end{tikzpicture}
\end{aligned}
\]
This is given by composition of corelations: we take the transitive closure of
the two equivalence relations, and then restrict it to an equivalence relation
on $X+Z$.   This composition gives a category $\Fin\Corel$ with finite sets as
objects and corelations as morphisms.   In fact $\Fin\Corel$ is a hypergraph
category---but this is just a special case of a more general result in the next
section.

\subsection{Corelation categories}
\label{ssec:corelation_categories}

There are two equivalent ways to think of a relation $R \maps X \asrelto Y$ between
sets.  One is as a subset of $X \times Y$, and another is as an 
isomorphism class of jointly monic spans $X \leftarrow R \to Y$, where
\define{jointly monic} means that the resulting map $R \to X \times Y$ is monic.
Dually, there are two ways to think of a corelation from $X$ to $Y$.  One is as an
equivalence relation on the set $X + Y$.   Another  is as an isomorphism class 
of \define{jointly epic} cospans $X \rightarrow C \leftarrow Y$: that is, 
cospans for which the associated map $X + Y \to C$ is epic.  The points of $C$ 
correspond to equivalence classes of the equivalence relation on $X + Y$.

This second approach to corelations can be developed more generally in any
category with a suitable factorization system.

\begin{definition}
  A \define{factorization system} $(\E,\M)$ in a category
  $\C$ is a pair of subcategories $\E$ and $\M$ of $\mathcal
  C$ such that:
  \begin{enumerate}[(i)]
    \item $\E$ and $\M$ contain all isomorphisms of $\mathcal
      C$.
    \item  Every morphism $f$ of $\C$ admits a factorization $f=me$ with $e \in \E$ and $m \in \M$.
\item  Given morphisms $f,f'$ of $\C$ with factorizations $f = me$, $f' = m'e'$ of the above sort, for every $u$, $v$ such that $vf = f'uu$, there exists a unique morphism $s$ such that
  \[
    \xymatrixcolsep{3pc}
    \xymatrixrowsep{3pc}
    \xymatrix{
      \ar[r]^e \ar[d]_u & \ar[r]^m \ar@{.>}[d]^{\exists! s} &  \ar[d]^v \\
       \ar[r]_{e'}& \ar[r]_{m'} & 
    }
  \]
  commutes.
  \end{enumerate}
\end{definition}

\begin{definition}
\label{def:EM-corelation}
  Let $\C$ be a category with finite colimits, and let $(\E,
  \M)$ be a factorization system on $\C$. A \define{corelation} from $X \in \C$ to $Y \in \C$
  is an isomorphism class of cospans $X
  \stackrel{i}\longrightarrow N \stackrel{o}\longleftarrow Y$ in $\C$ that is \define{jointly
  $\E$-like}: the copairing $[i,o]\maps X+Y \to N$ lies in $\E$.
\end{definition}

In this situation we can convert any cospan
$X \stackrel{i}\longrightarrow N \stackrel{o}\longleftarrow Y$ into a corelation 
as follows.   First factor the copairing $[i,o] \maps X + Y \to N$ as  $X+Y \stackrel{e}{\longrightarrow} \overline{N} \stackrel{m}{\longrightarrow} N$ with $e \in \E$ and 
$m \in \M$.   Then form the corelation coming from the cospan 
\[
\xymatrix{
  X \ar[r]^-{e j_X} 
  & \overline{N} 
  & Y .\ar[l]_-{e j_Y}  
}
\]
We call this cospan the \define{$\E$-part} of the cospan we started with, and
$m\maps \overline{N} \to N$ the \define{$\M$-part}.

If we compose corelations as cospans, the result is typically not a corelation.  Thus, 
we compose corelations by taking the $\E$-part of their composite
cospan. This composition rule turns out to give a category when the factorization system is `co-stable', the notion dual to stability under pullbacks.

\begin{definition}
  Let $\C$ be a category with finite colimits, and let $(\E,
  \M)$ be a factorization system on $\C$.   We say
  that a factorization system $(\E,\M)$ is \define{co-stable} if for every pushout square
  \[
    \xymatrixcolsep{3pc}
    \xymatrixrowsep{3pc}
    \xymatrix{
      \ar[r]^j & \\
      \ar[u] \ar[r]_m &  \ar[u]
    }
  \]
  such that $m \in \M$, we also have that $j \in \M$. 
\end{definition}

\begin{lemma} \label{lemma:corelcat}
  Let $\C$ be a category with finite colimits and a co-stable factorization
  system $(\E,\M)$. Then there exists a category $\Corel(\C)$ with the objects
  of $\C$ as objects, $(\E,\mathcal M)$-corelations as morphisms, and
  composition given as above. Moreover, there exists a unique hypergraph structure on
  $\Corel(\C)$ such that the map taking a cospan to its $\E$-part defines a
  hypergraph functor $\square\maps \Cospan(\C) \to \Corel(\C)$.
\end{lemma}

\begin{proof}
This is Thm.\ 3.1 and Cor.\ 4.5 of \cite{Fong18}. Note that the fact that
$\square$ is a hypergraph functor immediately implies that the hypergraph
structure on each object of $\Corel(\C)$ must be the image of the hypergraph
structure of that same object considered inside $\Cospan(\C)$.
\end{proof}

We need three examples of this construction:

\begin{example}
\label{ex:Corel}
If $\C = \Fin\Set$, its epi-mono factorization system is co-stable and the resulting
category $\Corel(\Fin\Set)$ is isomorphic to the category $\Fin\Corel$ explained in the
previous section.     The category $\Fin\Corel$ thus becomes a hypergraph category.  In
Section \ref{ssec:corelations} we explained how morphisms in $\Fin\Corel$ describe circuits 
of ideal wires.
\end{example}

\begin{example}
\label{ex:LinCorel}
If $\C = \Fin\Vect$ is the category of finite-dimensional vector spaces over a field
$\F$, then its epi-mono factorization is co-stable, and we obtain a category 
$\Corel(\Fin\Vect)$, which we call $\Lin\Corel$.    This is isomorphic to the category 
$\Lin\Rel$ where objects are finite-dimensional vector spaces and morphisms are 
linear relations \cite[Sec.\ 7.2]{Fong18}.  A morphism in $\Lin\Corel$ is an isomorphism 
class of cospans $V \stackrel{i}\longrightarrow S \stackrel{o}\longleftarrow W$ in $\Fin\Vect$ 
such that $[i,o] \maps V \oplus W \to S$ is onto.   There is an isomorphism $\Lin\Corel 
\cong \Lin\Rel$ that sends any such morphism to the isomorphism class of spans 
$V \leftarrow \ker(i-o) \rightarrow W$, where 
\[  \ker(i-o) = \{(v,w) \in V\oplus W | \; i(v) - o(w) = 0\}   \]
and the legs of the span come from the projections of $V\oplus W$ onto $V$ and $W$.
\end{example}

\begin{example}
\label{ex:LinRel}
If $\C = \Fin\Vect^\opp$, then its epi-mono factorization is also 
co-stable, and the resulting category $\Corel(\Fin\Vect^\opp)$ is also isomorphic to $\Lin\Rel$,
since by duality a corelation in $\Fin\Vect^\opp$ is the same as a relation in $\Fin\Vect$.   
One must be careful, however, because this isomorphism $\Lin\Rel \cong \Corel(\Fin\Vect^{\opp})$ 
gives $\Lin\Rel$ a different hypergraph structure than the isomorphism $\Lin\Rel \cong \Corel(\Fin\Vect)$ in the previous example.
\end{example}

There is also an easy way to get hypergraph functors between corelation categories.

\begin{lemma} \label{lemma:corelfunctors}
  Let $\C$, $\C'$ have finite colimits and co-stable
  factorization systems $(\E, \M)$ and $(\E', \M')$, respectively.  
  Further let $A\maps \C \to \C'$ be a functor that
  preserves finite colimits and such that the image of $\M$ lies in
  $\M'$.  Then there exists a unique hypergraph functor $\Corel(A) \maps \Corel(\C) \to
  \Corel(\C')$ sending each object $X$ in $\Corel(\C)$ to $F(X)$
  in $\Corel(\C')$ and each corelation 
 \[ X \stackrel{i}{\longrightarrow} N \stackrel{o}{\longleftarrow} Y \]
to the $\E'$-part
\[
\xymatrix{
  A(X) \ar[rr]^-{e' j_{A(X)}} 
  && \overline{A(N)} 
  && A(Y) \ar[ll]_-{e' j_{A(Y)}}  
}
\]
of its image in $\C'$.   
\end{lemma}

\begin{proof}   This is Prop.\ 4.1 of \cite{Fong18}.
\end{proof}

\subsection{Decorated corelation categories}
\label{ssec:deccorel}

In Lemma \ref{lemma:decorated_cospans} we described a systematic procedure
for constructing decorated cospan categories.    This has an analogue for corelations.
In what follows we assume $\C$ has finite colimits and a co-stable factorization system $(\E, \M)$.

\begin{definition}
Define the category $(\C;\M^\opp)$ to have objects those of $\C$, with morphisms
from $X$ to $Y$ being isomorphism classes of cospans $X \stackrel{f}\longrightarrow N \stackrel{m}\longleftarrow Y$ with $f$ in $\C$ and $m$ in $\M$, and composition given by
pushout. 
\end{definition}

The category $(\C;\M^\opp)$ becomes symmetric monoidal using the coproduct $+$ in $\C$.

\begin{lemma} 
\label{lemma:deccorel}
Let 
\[
(F,\varphi)\maps (\C;\M^\opp,+)\longrightarrow (\Set, \times)
\]
be a lax symmetric monoidal functor. There is a category
$F\Corel$, the category of \define{$F$-decorated corelations}, with objects being
those of $\C$ and morphisms from $X$ to $Y$ being equivalence classes of pairs
\[
    (X \stackrel{i}\longrightarrow N \stackrel{o}\longleftarrow Y,\enspace s)
\]
consisting of a corelation with an element $s \in F(N)$
called the \define{decoration}.   The equivalence relation arises from isomorphism of cospans.
The composite of morphisms in $F\Corel$ is the composite of their corelations decorated by 
the image of the pair of decorations under the map 
  \[
   FN \times FM
    \xrightarrow{\varphi_{N,M}} F(N+M)
    \xrightarrow{F\big(\xrightarrow{[j_N,j_M]}\xleftarrow{m}\big)} F(\overline{N+_YM}),
  \]
 where $m\maps \overline{N+_YM} \to N+_YM$ is the $\M$-part of the composite
 cospan.

Moreover, abusing notation to write $F$ also for the restriction of $F$ to the subcategory
$\C$ of $\C;\M^\opp$, there exists a functor 
\[
\square_F \maps F\Cospan \to F\Corel
\]
mapping each object $X$ to itself and each $F$-decorated cospan $(X
\xrightarrow{i} N \xleftarrow{o} Y, s)$ to the $F$-decorated corelation $(X
\xrightarrow{ej_X} \overline{N} \xleftarrow{ej_Y} Y, Fm^\opp(s))$, where $[i,o]$
factorizes as $m \circ e$, and $m^\opp$ is the morphism
$N\xrightarrow{1}N\xleftarrow{m}\overline{N}$ in $\C;\M^\opp$.

The category $F\Corel$ can be equipped with the structure of a hypergraph category 
in a unique way such that $\square_F $ is a hypergraph functor.
\end{lemma}

\begin{proof}
This is Thm.~5.8 and Cor.~6.2 of \cite{Fong18}.
\end{proof}

\begin{lemma} 
\label{lemma:deccorelfunctor}
  Let $\mathcal C$, $\mathcal C'$ have finite colimits and respective costable
  factorisation systems $(\mathcal E, \mathcal M)$, $(\mathcal E', \mathcal
  M')$, and suppose that we have symmetric lax monoidal functors
\[
  F \maps (\mathcal C;\mathcal M^\opp,+) \longrightarrow (\Set, \times),
\]
\[
  G\maps (\mathcal C';\mathcal M'^\opp,+) \longrightarrow (\Set, \times).
\]
Suppose $A\maps \mathcal C \to \mathcal C'$ be a functor that preserves
finite colimits and such that the image of $\mathcal M$ lies in $\mathcal M'$.
This functor $A$ extends canonically 
to a symmetric monoidal functor from $\mathcal C;\mathcal M^\opp$
to $\mathcal C';\mathcal M'^\opp$, which we again call $A$.

Suppose we have a monoidal natural transformation $\theta$:
\[
  \xy
   (-20,8)*+{ \mathcal C; \mathcal M^\opp}="1";
   (-20,-8)*+{\mathcal C'; \mathcal {M'}^\opp}="2";
   (10,0)*+{\Set}="3";
        {\ar^{F} "1";"3"};
        {\ar_{G} "2";"3"};
        {\ar^{} "1";"2"};
        {\ar@{=>}^{\theta} (-13,-3)*{}; (-7,2)*{}};
\endxy
\]
Then we may define a hypergraph functor $T\maps F\mathrm{Corel} \to
G\mathrm{Corel}$ sending each object $X \in F\mathrm{Corel}$ to $AX \in
G\mathrm{Corel}$ and each decorated corelation 
  \[
    \left(
    \begin{aligned}
      \xymatrix{
	& N \\  
	X \ar[ur]^{i} && Y \ar[ul]_{o}
      }
    \end{aligned}
    ,
    \qquad
    \begin{aligned}
      \xymatrix{
	FN \\
	1 \ar[u]_{s}
      }
    \end{aligned}
    \right)
    \quad \mbox{to}
    \quad
    \left(
    \begin{aligned}
      \xymatrix{
	& \overline{AN} \\  
	AX \ar[ur]^{e'\circ\iota_{AX}} && AY \ar[ul]_{e'\circ\iota_{AY}}
      }
    \end{aligned}
    ,
    \qquad
    \begin{aligned}
      \xymatrixrowsep{1.5ex}
      \xymatrix{
	G\overline{AN} \\
	GAN \ar[u]_{Gm_{AN}^\opp}\\
        FN \ar[u]_{\theta_N}\\
	1 \ar[u]_{s}
      }
    \end{aligned}
    \right).
  \]  
\end{lemma}
\begin{proof}
See Prop. 6.2 of \cite{Fong18}.
\end{proof}

\section{Constructing the black box functor}
\label{sec:blackbox}

We now complete our construction of the black 
box functor $\blacksquare \maps \Circ \to \Lag\Rel$ by composing the Lagrangian cospan 
semantics $BA \maps \Circ \to \Lag\Cospan$ with two more functors:
\[
\xymatrix{
  \Circ = \Circuit\Cospan \ar[r]^-{A}
  & \Dirich\Cospan \ar[r]^{B}  
  & \Lag\Cospan \ar[d]^{\square_\Lag}   
   \\
  & 
  & \Lag\Corel \ar[r]_-{\cong}^-{C}
  & \Lag\Rel.
}
\]
The first, $\square_\Lag \maps \Lag\Cospan \to \Lag\Corel$, 
does the real work of black-boxing.   This is an example of the construction in Lemma \ref{lemma:deccorel}.  However, to apply this lemma we first need to construct the category of Lagrangian corelations, $\Lag\Corel$, which we do in Sections \ref{ssec:ideal_wires} and \ref{ssec:LagCorel}.   In Section \ref{ssec:LagCorel} we also construct the isomorphism $C\maps\Lag\Corel \to \Lag\Rel$ that completes the black box functor.   In Section \ref{ssec:blackbox} we assemble the black box functor and describe what it does.  In brief, it assigns to any open circuit the Lagrangian relation between its input and output potentials and currents.

\subsection{The semantics of ideal wires}
\label{ssec:ideal_wires}

To construct the category of Lagrangian corelations as a decorated corelation
category using Lemma \ref{lemma:deccorel} we now extend symplectification from
$\Fin\Set$ to $\Fin\Corel$, obtaining a lax symmetric monoidal functor 
\[   S \maps (\Fin\Corel, +) \to  (\Set,\times)  .\]  
Far from being a mere technical device, this functor describes how potentials
and currents behave in circuits made of ideal wires. Just as we built the
symplectification functor of Section \ref{sec:circsemantics} using pullback and
pushforward maps, we build this extension in two analogous parts, one for
potentials and the other for currents.

The first part is a functor $\Phi$ mapping each corelation to the linear
relation it imposes between potentials at inputs and outputs.  This functor
expresses the fact that the potential is constant on any connected component of
a circuit made of ideal wires.

\begin{proposition}
\label{prop:Phi}
  Define the functor 
  \[ 
    \Phi\maps \Fin\Corel \longrightarrow \Lin\Rel
  \] 
  on objects by sending a finite set $X$ to the vector space $\F^X$, and on
  morphisms by sending a corelation $e\maps X \to Y$ to the linear subspace
  $\Phi(e)$ of $\F^X \oplus \F^Y$ comprising functions $\phi\maps X+Y \to \F$
  that are constant on each set $e^{-1}(n)$ for $n \in N$.  The functor $\Phi$ is strong 
  symmetric monoidal, with coherence maps the usual natural isomorphisms 
  $\F^X \oplus \F^Y \cong \F^{X + Y}$ and $\{0\} \cong \F^\varnothing$.   
\end{proposition}

\begin{proof}
Consider the free vector space functor $\Psi\maps \Fin\Set \to
\Fin\Vect^\opp$ mapping a finite set $X$ to the vector space $\F^X$, and the function
$f\maps X \to Y$ to the linear map $f^\ast\maps \F^Y \to \F^X$ sending
$\phi\maps Y \to \F$ to $\phi \circ f\maps X \to \F$. The functor $\Psi$ has a
left adjoint, and hence preserves colimits.  Also note that it sends injective
functions to surjective linear maps, which are monomorphisms in
$\Fin\Vect^\opp$. Applying Lemma~\ref{lemma:corelfunctors}, we obtain a
strong symmetric monoidal functor from $\Fin\Corel$ to $\Corel(\Fin\Vect^\opp)$.  Composing 
this with the isomorphism $\Corel(\Fin\Vect^\opp) \cong \Lin\Rel$ from Example \ref{ex:LinRel}
we obtain a strong symmetric monoidal functor $\Phi \maps \Fin\Corel \to \Lin\Rel$.

It is easily checked that $\Phi$ maps any corelation $e= [i,o]\maps X+Y \to
N$ to the linear relation $\mathrm{im}(e^*) \subseteq \F^X \oplus \F^Y$, which
consists of all functions $X+Y \to \F$ that are constant on each fiber of $e$.
\end{proof}

The second part is a functor $I$ mapping each corelation to the linear relation it imposes between
currents at inputs and outputs.   This functor expresses Kirchhoff's current
law: the sum of currents flowing into each node must equal the sum of currents flowing out.  

\begin{proposition}
\label{prop:I}
  Define the functor
  \[
    I\maps \Fin\Corel \longrightarrow \Lin\Rel
  \]
  as follows. On objects send a finite set $X$ to the vector space
  $(\F^{X})^\ast$, the free vector space on $X$, with basis given by symbols 
  $dx$, one for each element $x \in X$.   On morphisms send a
  corelation $e=[i,o]\maps X +Y \to N$ to the linear relation 
  \[
    I(e) := \left\{\left(\sum_{x \in X} a_xdx,\sum_{y \in Y} b_ydy\right) 
    \middle|
    \sum_{i(x)=n} a_x = \sum_{o(y)=n} b_y \right\}
    \subseteq  (\F^{X})^\ast \oplus (\F^{Y})^\ast.
  \]
  This functor $\Phi$ is strong symmetric monoidal functor, with coherence maps the natural
  isomorphisms $(\F^X)^\ast \oplus (\F^Y)^\ast \to (\F^{X+Y})^\ast$ and $\{0\}
  \to (\F^\varnothing)^\ast$.      
\end{proposition}
\begin{proof}
  This proposition follows in the same manner as the previous one, but
  applying Lemma~\ref{lemma:corelfunctors} to the free vector
  space functor $J\maps \Fin\Set \to \Fin\Vect$.  This can be seen as mapping any 
  finite set $X$ to the vector space $(\F^X)^\ast$, which has a basis given by symbols
  $dx$, one for each element of $x$.
   Given a function $f\maps X \to Y$, $J$ maps $f$ to the linear map sending each
   basis element $dx$ to the basis element $d(f(x))$.  The functor $J$ preserves colimits 
   and maps monos to monos, so the hypotheses of Lemma~\ref{lemma:corelfunctors} are 
   satisfied. We thus obtain a strong symmetric monoidal functor from $\Fin\Set$ to 
   $\Corel(\Fin\Vect)$.   Composing this with the isomorphism $\Corel(\Fin\Vect)^\opp \cong
   \Lin\Rel$ from Example \ref{ex:LinCorel} we obtain a strong symmetric monoidal functor
   $I \maps \Fin\Corel \to \Lin\Rel$.

  By definition this functor $I$ maps a corelation $e=[i,o]\maps X+Y \to N$
  to the kernel of the linear map $J(i) - J(o) \maps (\F^X)^\ast \oplus
  (\F^Y)^\ast \to (\F^N)^\ast$. This consists of all linear combinations
  $\left(\sum_{x \in X} a_xdx,\sum_{y \in Y} b_ydy\right)$ such that
  $\sum_{x \in X} a_x d(i(x)) =  \sum_{y \in Y} b_y d(o(y))$, proving the
  above formula for $I(e)$.
\end{proof}

We have now defined two functors that describe the behavior of 
circuits made of ideal wires: one for potentials and one for currents.  
Combining these, we obtain a functor describing the behavior of both 
potentials and currents.   The relation between potentials and currents at
inputs and outputs is not merely a linear relation: it is a Lagrangian relation.

\begin{theorem} 
\label{thm:symplectification}
  We may define a strong symmetric monoidal functor $S$
  \[
    S\maps (\Fin\Corel,+) \longrightarrow (\Lag\Rel,\oplus)
  \]
  sending any finite set $X$ to the symplectic vector space $\F^X \oplus
  (\F^X)^\ast$ and any corelation $e=[i,o]\maps X+Y \to N$ to the Lagrangian
  relation
  \[
    S(e) = \Phi(e) \oplus I(e) \subseteq \overline{\F^X \oplus
    (\F^X)^\ast}\oplus \F^Y \oplus (\F^Y)^\ast.
  \]
\end{theorem}

Note that $\Phi(e)\oplus I(e)$ is more properly a subspace of $\F^X
\oplus \F^Y \oplus (\F^X)^\ast \oplus (\F^Y)^\ast$. In the above we consider it
instead as a subspace of the symplectic vector space $\overline{\F^X \oplus
(\F^X)^\ast}\oplus \F^Y \oplus (\F^Y)^\ast$, whose underlying vector space is
canonically isomorphic to $\F^X \oplus \F^Y \oplus (\F^X)^\ast \oplus
(\F^Y)^\ast$.

\begin{proof}
  As the pointwise tensor product in $\Lin\Rel$ of strong symmetric monoidal
  functors $\Phi$ and $I$, $S$ is itself a strong symmetric monoidal functor
  $(\Fin\Corel,+) \to (\Lin\Rel,\oplus)$. 
  
  It remains to show that the image of each corelation $e$ is Lagrangian.   We
  prove this using Lemma \ref{lemma:lagrangian_characterization}: a Lagrangian
  subspace is an isotropic subspace of dimension half that of the symplectic
  vector space. By compactness, we may assume without loss of generality that
  $X=\varnothing$. Consider then some element $(\phi_Y,i_Y) \in S(e)$. To prove
  isotropy, note that (i) for each $n \in N$ there exists $\phi_n \in \F$ such
  that $\phi_n=dy(\phi_Y)$ for all $y \in Y$ such that $o(y)= n$, and (ii) that
  $\sum_{o(y)=n}\lambda_y =0$.  Then $Se$ is isotropic as, for all pairs
  $(\phi_Y,i_Y)$, $(\phi_Y',i_Y') \in Se$ we have
  \[
    \omega\big((\phi_Y,i_Y),(\phi_Y',i_Y')\big)
    =  \sum_{y \in Y} \lambda_y' dy(\phi_Y)
    -\sum_{y \in Y} \lambda_y dy(\phi_Y')
    = \sum_{n \in N}\sum_{o(y)=n} \lambda_y'\phi_n - \sum_{n \in
    N}\sum_{o(y)=n} \lambda_y\phi'_n
    = 0.
  \]
  Observing that $Se$ has dimension
  \[
    \dim(\Phi(e))+ \dim(I(e)) = |N|+(|Y|-|N|) = |Y| = \tfrac12 \dim\big(\F^Y \oplus
    (\F^Y)^\ast\big)
  \]
  then proves the proposition.  
\end{proof}

It is not difficult to show that $S$ extends the strong symmetric monoidal functor of the same 
name previously constructed in Prop.\ \ref{prop:symplectification},
if we think of a function $f\colon X \to Y$ as a cospan of the form
$X \xrightarrow{f} Y \xleftarrow{1_Y} Y$.

We next use symplectification of cospans to construct the category of Lagrangian corelations.

\subsection{Lagrangian corelations}
\label{ssec:LagCorel}

We now construct a decorated corelation category with `Lagrangian corelations'
as morphisms.  This turns out to be isomorphic to our previous category with
Lagrangian relations as morphisms, but the new outlook lets us finish
constructing the black box functor.

To proceed, define the functor
\[
  \Lag\maps (\Cospan(\Fin\Set),+) \longrightarrow (\mathrm{Set},\times)
\]
to be the composite 
\[
\xymatrix{
      \Cospan(\Fin\Set) \ar[r]^-{\square}
  & \Fin\Corel \ar[r]^-{S}  
  & \Lag\Rel \ar[rr]^-{\Lag\Rel(\emptyset,-)}   
  && \Set
}
\]
where $\square$ is defined as in Lemma \ref{lemma:deccorel}, $S$ is defined
as in Thm.\ \ref{thm:symplectification}, and $\Lag\Rel(\emptyset,-)$ arises from the 
hom-functor of $\Lag\Rel$.    All these factors are strong symmetric monoidal
except for the last, which is lax symmetric monoidal because the empty set
is the unit for the tensor product in $\Lag\Rel$.   Thus, we obtain a lax symmetric
monoidal functor
\[   
  (\Lag,\lambda) \maps (\Cospan(\Fin\Set),+) \longrightarrow (\mathrm{Set},\times).
\]
A calculation show that this extends the same-named functor from Lemma \ref{lemma:Lag}.

\begin{theorem}
\label{thm:LagCorel} 
There is a hypergraph category $\Lag\Corel$ obtained by applying
Lemma~\ref{lemma:deccorel} to the lax symmetric monoidal functor
$\Lag \maps  (\Cospan(\Fin\Set),+) \longrightarrow (\mathrm{Set},\times)$. 
The category $\Lag\Corel$ has finite sets as objects, and for morphisms from $X$ to $Y$
it has Lagrangian subspaces of 
\[
  \vect{X+Y} \cong \vect{X} \oplus \vect{Y}
\]
Composition in $\Lag\Corel$ is enacted by the symplectification of the
cospan 
\[
  X+Y+Y+Z \xrightarrow{1_X+[1_Y,1_Y]+1_Z} X+Y+Z
  \xleftarrow{\iota_X+\iota_Z} X+Z.
\]
This relation relates 
\[
  (\phi_X,\phi_Y,\phi_Y',\phi_Z',i_X,i_Y,i_Y',i_Z') \in 
  \F^{X+Y+Y+Z}\oplus (\F^{X+Y+Y+Z})^\ast
\]
with 
\[
  (\phi_X,\phi_Z',i_X,i_Z') \in \F^{X+Z}\oplus (\F^{X+Z})^\ast
\]
if and only if two conditions hold: (i) $\phi_Y = \phi_Y'$ and (ii) $i_Y +i_Y' =
0$. 
The identity morphism on $X$ is the Lagrangian subspace $\{(\phi,-i,\phi,i) \mid
\phi \in \F^X, i \in (\F^X)^\ast\}$.

Lemma~\ref{lemma:deccorel} also provides a hypergraph functor
\[
 \square_\Lag \maps \Lag\Cospan \longrightarrow \Lag\Corel.
\]
\end{theorem}

\begin{proof}
To apply Lemma \ref{lemma:deccorel}, note that $\Cospan(\Fin\Set) = \C;\M^\opp$
where $\C = \M = \Fin\Set$.
\end{proof}

We see therefore that when composing two Lagrangian corelations the potentials
$\phi_Y$ and $\phi'_Y$ must agree, but on the
currents we have $i_Y=-i_Y'$. We intepret this to mean that the
morphisms of $\Lag\Corel$ record only the currents \emph{out} of the
nodes of the circuit. Thus if we were to interconnect two nodes, the current out
of the first must be equal to the negative of the current out---that is, the
current \emph{in}---of the second node. It is remarkable that this physical fact
is embedded so deeply in the underlying mathematics.

\begin{theorem}
\label{thm:LagCorel=LagRel}
There is an isomorphism of categories
\[
C \maps \Lag\Corel \longrightarrow \Lag\Rel 
\]
sending any finite set $X$ to itself and any morphism $L \subseteq \vect{X}
\oplus \vect{Y}$ to
\[
 CL = \{(\phi_X,-i_X,\phi_Y,i_Y) \mid (\phi_X,i_X,\phi_Y,i_Y) \in L\}.
\]
\end{theorem}

Equivalently, $CL$ is the image of $L$ under the symplectomorphism 
\[
\overline{1_X} \oplus 1_Y \maps \vect{X} \oplus \vect{Y}
\stackrel{\sim}{\longrightarrow} \overline{\vect{X}} \oplus \vect{Y}.
\]

\begin{proof}
  This proposed functor $C$ is identity on objects and fully faithful. Thus to prove
  the theorem we just need to check that it is indeed a functor. It is
  straightforward to observe that $C$ preserves identities.  To check that $C$
  preserves composition, suppose that we have Lagrangian corelations $L\colon X
  \to Y$ and $M\colon Y \to Z$. Then we can compute
  \begin{align*}
    C(M \circ L) 
    &= \bigg\{ (\phi_X,-i_X,\phi_Z,i_Z) \, \bigg\vert \,
    \begin{array}{c} 
    \textrm{there exists } \phi_Y \in \F^Y, i_Y \in (\F^Y)^\ast \textrm{ such
    that } \\
    (\phi_X,i_X,\phi_Y,i_Y) \in L, \quad (\phi_Y,-i_Y,\phi_Z,i_Z) \in M
    \end{array} 
    \bigg \} \\
    &= \bigg\{ (\phi_X,i_X,\phi_Z,i_Z) \, \bigg\vert \,
    \begin{array}{c} 
    \textrm{there exists } \phi_Y \in \F^Y, i_Y \in (\F^Y)^\ast \textrm{ such
    that } \\
    (\phi_X,-i_X,\phi_Y,i_Y) \in L, \quad (\phi_Y,-i_Y,\phi_Z,i_Z) \in M
    \end{array} 
    \bigg \} = CM \circ CL,
  \end{align*}
  proving functoriality.
\end{proof}

We use this isomorphism to transfer the hypergraph structure on $\Lag\Corel$ given
in Thm.\ \ref{thm:LagCorel} to $\Lag\Rel$.  They then become isomorphic as hypergraph
categories. As an aside, it is then not difficult to check that the
symplectification functor $S$ becomes a hypergraph functor.

\subsection{The black box functor}
\label{ssec:blackbox}

The black box functor maps any open circuit to the Lagrangian relation it imposes between
input and output potentials and currents.   In detail:

\begin{definition}
The \define{black box functor} $\blacksquare \maps \Circ \to \Lag\Rel$ 
is the composite
\[
\xymatrix{
  \Circ = \Circuit\Cospan \ar[r]^-{A}
  & \Dirich\Cospan \ar[r]^{B}  
  & \Lag\Cospan \ar[d]^{\square_\Lag}   
   \\
  & 
  & \Lag\Corel \ar[r]_-{\cong}^-{C} 
  & \Lag\Rel
}
\]
\end{definition}

\begin{theorem}
\label{thm:blackbox}
The black box functor is a hypergraph functor.    On objects it maps any 
finite set $X$ to the symplectic vector space $\blacksquare(X) = \vect{X}$.   
On morphisms it maps any open
  circuit $\Gamma$ with underlying cospan of finite sets $X \stackrel{i}\to N \stackrel{o}\leftarrow Y$ and passive linear circuit $(N,E,s,t,Z)$ to the Lagrangian relation
  \[
    \blacksquare(\Gamma) = (\overline{1_X}\oplus 1_Y) \circ S[i,o]^\opp
    \circ \mathrm{Graph}(dP)
    \subseteq \overline{\F^X \oplus (\F^X)^\ast} \oplus \F^Y \oplus (\F^Y)^\ast,
  \]
  where $P$ is the extended power functional of $\Gamma$.
  The coherence maps are given by the natural isomorphisms $\F^X \oplus
  (\F^X)^\ast \oplus \F^Y \oplus (\F^Y)^\ast \cong \F^{X+Y} \oplus
  (\F^{X+Y})^\ast$ and $\{0\} \cong \F^\varnothing \oplus (\F^\varnothing)^\ast$.
\end{theorem}

\begin{proof}
The black box functor is a hypergraph functor because it is the composite of
hypergraph functors $A$ (Thm.\ \ref{thm:A}), $R$ (Thm.\ \ref{thm:R}), $\square_\Lag$
(Thm.\ \ref{thm:LagCorel}), and $C$ (Thm.\ \ref{thm:LagCorel=LagRel}). The
formula for $\blacksquare(\Gamma)$ simply unpacks the definitions of these
functors. 
\end{proof}

Indeed, the formula states that the black box functor takes an open circuit $\Gamma$,
computes its extended power functional $P$ (the functor $A$), takes the graph of
the differential ($R$), restricts the resulting Lagrangian relation to the
boundary by enforcing Kirchhoff's laws on the interior ($\square_\Lag$), and
then flips the sign of the currents at the input nodes to indicate that they
are inputs ($C$).

\section{The main theorem} 
\label{sec:main}

At this point the reader might voice two concerns.   Firstly, why does the
\emph{black box} functor refer to the \emph{extended} power functional $P$ 
rather than the power functional $Q$, which is a function only of input and output
potentials?   Secondly, since this functor does not involve power minimization, how 
is it the same functor as that defined in Thm.~\ref{thm:main}? These fears are allayed
by the remarkable trinity of minimization, symplectification, and Kirchhoff's laws. 

We have seen that the symplectification functor $S$ maps a corelation to the behavior of the
corresponding circuit of ideal wires, governed by Kirchoff's laws (Section \ref{ssec:ideal_wires}).  
We have also seen that Kirchhoff's laws are closely related to the principle of minimum power 
(Thm.\ \ref{thm:circuit_behavior_from_power}).  The final aspect of this relationship
is that we can use symplectification to enact power minimization.

We use this link between symplectification and power minimization
to construct the commutative square here:
\[
\xymatrix{
  \Circ = \Circuit\Cospan \ar[r]^-{A}
  & \Dirich\Cospan \ar[r]^{B} \ar[d]_{\square_\Dirich} 
  & \Lag\Cospan \ar[d]^{\square_\Lag}   
   \\
  & \Dirich\Corel \ar[r]^{\widetilde{B}} 
  & \Lag\Corel  \ar[r]_{\cong}^-{C} & \Lag\Rel
}
\]
This shows that the black box functor, defined as the functor from $\Circ$ to
$\Lag\Rel$ obtained by the path through the top right corner of the
diagram, is the same as the functor obtained by the path through the lower left
corner, which is precisely that defined in Thm.~\ref{thm:main}.

\subsection{Dirichlet corelations}

First we construct $\Dirich\Corel$ as a decorated corelation category.
For this we must extend the functor $\Dirich \maps \Fin\Set \to \Set$ 
to the category $\Fin\Set;\mathrm{Inj}^\opp$, where
$\mathrm{Inj}$ is the category of finite sets and injections.   We know how to push
forward Dirichlet form along functions, thanks to Lemma~\ref{lemma:Dirich}; we now 
use power minimization to pull back Dirichlet forms along injections.
Thm.~\ref{thm:dirichlet_problem} implies 
that given a Dirichlet form $P$ on a finite set $N$, and an injection $m \maps Y
\to N$, there is a Dirichlet form $\min_{N \setminus \mathrm{im}(m)} P$ on $\mathrm{im}(m) \cong Y$.   
We henceforth identify this with a Dirichlet form on $Y$.   

\begin{lemma}
\label{lemma:Dirich2}
There exists a unique lax symmetric monoidal functor
\[
  (\Dirich, \delta) \maps (\Fin\Set;\mathrm{Inj}^\opp,+) \longrightarrow (\mathrm{Set},\times)
\]
that extends the functor $\Dirich$ of Lemma~\ref{lemma:Dirich} and has
\[ 
\begin{array}{rccl}  \Dirich(N\xrightarrow{1}N\xleftarrow{m} Y) \maps 
\Dirich(N) &\to& \Dirich(Y) \\
 P &\mapsto & \min_{N \setminus \mathrm{im}(m)} P .
 \end{array}
 \]
\end{lemma}

\begin{proof}
The category $\Fin\Set;\mathrm{Inj}^\opp$ is generated by the subcategories 
$\Fin\Set$ and $\mathrm{Inj}^\opp$.   We define the functor $\Dirich$ on $\Fin\Set$ as in 
Lemma~\ref{lemma:Dirich} and on morphisms in $\mathrm{Inj}^\opp$ as above.
Combining this with $\delta$ as defined in Lemma~\ref{lemma:Dirich}, 
this uniquely determines a lax symmetric monoidal functor from $\Fin\Set;\Inj^\opp$
to $\Set$.    It remains to check that this functor exists: namely, that it is well-defined on 
morphisms and preserves composition.

For this it suffices to prove that (i) for injections $m,m'$ we
have $\Dirich(\xrightarrow{i}\xleftarrow{m})
\Dirich(\xrightarrow{1}\xleftarrow{m'})=$
$\Dirich(\xrightarrow{1}\xleftarrow{mm'})$, and (ii) if a span $\xleftarrow{m}
\xrightarrow{f}$ with $m$ an injection has pushout $\xrightarrow{f'} \xleftarrow{m'}$, then we have
$\Dirich(\xrightarrow{1}\xleftarrow{m})$ $\Dirich(\xrightarrow{f}\xleftarrow{1})=$
$\Dirich(\xrightarrow{f'}\xleftarrow{1}) \Dirich(\xrightarrow{1}\xleftarrow{m'})$. 

Claim (i) is simply the observation that if we have a chain of inclusions $A
\subseteq B \subseteq C$, then for any Dirichlet form $P$ on $C$, we have
$\min_{B\setminus A} \min_{C\setminus B} P= \min_{C \setminus A} P$. This is
a consequence of Thm.~\ref{thm:dirichlet_problem}, which implies that
formal minimization of Dirichlet forms has a unique, well defined value: given
$\psi \in \F^C$, both $\min_{B\setminus A} \min_{C\setminus B} P(\psi)$ and
$\min_{C \setminus A} P(\psi)$ are equal to $P(\phi)$ for any extension $\phi$
of $\psi$ obeying the principle of minimum power for $P$.

Claim (ii) starts with a span $\xleftarrow{m} \xrightarrow{f}$ with $m \in
\Inj$ and $f \in \Fin\Set$. Without loss of generality, write this span as $A+B
\xleftarrow{j_B} B \xrightarrow{f} C$, where $j_B$ is the canonical inclusion.
Then the pushout cospan is $A+B\xrightarrow{1_A+f} A+C\xleftarrow{j_C}C$. We
must check that 
\[
\Dirich(A+B \xleftarrow{j_B} B \xrightarrow{f} C) = 
\Dirich(A+B\xrightarrow{1_A+f} A+C\xleftarrow{j_C}C).
\]
Thus, for each Dirichlet form $P$ on $A+B$, we must show that
$f_\ast \min_{A} P = \min_{A} (1_A+f)_\ast P$. 

To do this, take any $\psi \in \F^C$.  By Thm.\ \ref{thm:dirichlet_problem} we can
choose an extension $\phi \in \F^{A+C}$ 
of $\psi$ that obeys the principle of minimum power with respect to the Dirichlet form
$(1_A+f)_\ast P$, so that
\[   ((1_A + f)_\ast P)(\phi) = (\min_A (1_A + f)_\ast P)(\psi)  .\]
Since $\phi$ extends $\psi$,
we have that $\phi \circ (1_A+f) \in \F^{A+B}$ extends $\psi \circ f$.  Next,
write $(1_A+f)^\ast \maps \F^{A+C} \to \F^{A+B}$ for the pullback linear
transformation.   Using the chain rule one can show
\[
\frac{\partial P}{\partial \varphi(a)}\bigg\vert_{\varphi = \phi \circ (1_A+f)}
= \frac{\partial (1_A+f)_\ast P}{\partial\varphi(a)}\bigg\vert_{\varphi =
\phi} 
= 0.
\]
This implies that $\phi\circ (1_A+f)$ obeys the principle of minimum power with respect to $P$,
so
\[    P(\phi \circ (1_A + f)) = (\min_A P)(\psi \circ f) .\]
Using these observations and the definition of pushforward for Dirichlet forms,
we then have 
\[
(f_\ast \min_A P)(\psi) = (\min_A P)(\psi \circ f)= P(\phi\circ (1_A+f)) = (1_A+f)_\ast P(\phi) =
(\min_A((1_A+f)_\ast P))(\psi).
\]
This proves claim (ii), and so that $\Dirich$ is functorial.

For symmetric monoidal structure, we use the same lax coherence maps as in
Lemma~\ref{lemma:Dirich}. We just need to observe they are natural in
$\Inj^\opp$; that is, natural in minimization. This again follows from
Thm.~\ref{thm:dirichlet_problem}, in particular from the linearity condition
(v).
\end{proof}

In what follows we abuse language by calling a corelation from $X$ to $Y$ a
surjection $e \maps X + Y \to N$, when it is really an isomorphism class of such.

\begin{theorem}
\label{thm:DirichCorel} 
There is a hypergraph category $\Dirich\Corel$ obtained by applying
Lemma~\ref{lemma:deccorel} to the lax symmetric monoidal functor $\Dirich \maps
(\Fin\Set;\Inj^\opp,+) \longrightarrow (\mathrm{Set},\times)$.  The category
$\Dirich\Corel$ has finite sets as objects, with a morphism from $X$ to $Y$ being
a pair $(e, Q)$ where $e \maps X + Y \to N$ is a corelation from $X$ to $Y$
and $Q$ is a Dirichlet form on $N$.  

Composition of $(e,Q)\maps X \to Y$ and $(e', Q')\maps Y \to Z$ in
$\Dirich\Corel$ has two parts.  To obtain the corelation $e'\circ e\maps X+Z
\to M$, we simply compose the corelations $e\maps X+Y \to N$ and $e'\maps Y+Z
\to N'$. Note that this gives a morphism $N+N' \to N+_YN' \leftarrow M$ in
$\Fin\Set;\Inj^\opp$, where the first leg is the canonical map and the second
is the inclusion of the image of $X+Z$ into the pushout.  The composite
Dirichlet form is then the image of $Q+Q'$ under the action of this morphism. 

Lemma~\ref{lemma:deccorel} also provides a hypergraph functor
\[
 \square_\Dirich \maps \Dirich\Cospan \longrightarrow \Dirich\Corel.
\]
\end{theorem}
\begin{proof}
This follows from Lemma~\ref{lemma:deccorel} taking $\C =\Fin\Set$, 
$\E$ the subcategory of surjections between finite sets,
$\M = \Inj$, and $F = \Dirich$ described as in Lemma~\ref{lemma:Dirich2}.
\end{proof}

\subsection{Composition through power minimization}

We now construct an alternate route to the black box functor and use this to prove the main
theorem.   This alternate route involves a procedure for turning Dirichlet corelations into
Lagrangian corelations using the principle of minimum power.

\begin{lemma}\label{lem:dirichtolag}
Write $\Lag$ for the restriction of $\Lag$ to $\Fin\Set;\Inj^\opp$. The natural
transformation $\beta$ of Lemma~\ref{lemma:B} extends to a monoidal natural transformation 
\[
\xymatrix@C+3pc{
(\Fin\Set;\Inj^\opp,+) \rtwocell<5>^{\Dirich}_{\Lag}{\widetilde\beta} & (\Set,\times)
}
\]
\end{lemma}
\begin{proof}
As $\Fin\Set$ and $\Fin\Set;\Inj^\opp$ have the same objects, this extension of
$\beta$ has the same data: the map $\widetilde{\beta}_X$ is the map $\beta_X$ given in
Lemma~\ref{lemma:B} that sends a Dirichlet form $P$ to the Lagrangian relation
$\mathrm{Graph}(dP)$. Moreover, monoidality of this transformation is immediate.
We just need to show that $\widetilde{\beta}$ is also natural with respect to morphisms
in $\Inj^\opp$. More explicitly, let $m\colon Y \to N$ be
an injection, and let $P$ be a Dirichlet form on $N$. Write $Q = \min_{N
\setminus Y} P$ for the Dirichlet form on $Y$ given by
minimization over $N \setminus Y$. Then the naturality square
\[
\xymatrix{
\Dirich(N) \ar[rr]^{\Dirich(m^\opp)} \ar[d]_{\widetilde{\beta}_N} && \Dirich(Y)
\ar[d]^{\widetilde{\beta}_{Y}} \\
\Lag(N) \ar[rr]_{\Lag(m^\opp)} && \Lag(Y)
}
\]
asserts that for every $P \in \Dirich(N)$ we have the equality of Lagrangian subspaces
\[
  S(m^\opp )\circ \mathrm{Graph}(dP) = \mathrm{Graph}(dQ),
\]
where $Q = \Dirich(m^\opp)(P) = \min_{N \setminus Y} P$.

  To prove this we first compute $S(m^\opp)$.  One can do this using Thm.~\ref{thm:symplectification}, but it is quicker to note that since $S$ is a hypergraph
  functor, $S(m^\opp)$ is equal to the transpose of the relation $S(m)$, which is
  described by Prop.~\ref{prop:symplectification}.   Either way we get
  \[
    S(m^\opp) = \big\{(\phi, m_\ast i,m^\ast \phi, i) \, \big\vert
      \, \phi \in \F^{N}, i \in (\F^{Y})^\ast \big\} \subseteq
      \overline{\F^N \oplus (\F^N)^\ast} \oplus \F^{Y} \oplus
      (\F^{Y})^\ast.
  \]
  Since $\mathrm{Graph}(dP)$ consists of pairs $(\phi,dP_\phi) \in \vect{N}$, this
  implies
  \[
    S(m^\opp) \circ \mathrm{Graph}(dP) = \big\{(m^\ast \phi, i)
    \,\big\vert\, \phi \in \F^N, i \in (\F^{Y})^\ast, dP_\phi =
  m_\ast i \big\} \subseteq  \F^{Y} \oplus
      (\F^{Y})^\ast.
  \]
  We must show this Lagrangian subspace is equal to $\mathrm{Graph}(dQ)$.
  
  Consider the constraint $dP_\phi = m_\ast i$. This states that for all
  $\varphi \in \F^N$ we have $dP_\phi(\varphi) = i(\varphi\circ m)$. Letting
  $\chi_n: N \to \F$ be the function sending $n \in N$ to $1$ and all other
  elements of $N$ to $0$, we see that when $n \in N \setminus Y$ we
  must have
  \[
    \frac{dP}{d\varphi(n)}\Bigg\vert_{\varphi = \phi}  = dP_\phi(\chi_n) = 
    i(\chi_n \circ m) = i(0) = 0.
  \]
  Thus, $\phi$ is an extension of $\psi = \phi \circ m$ obeying the principle of minimum
   power.  As $m$ is injective, $\psi = \phi
  \circ m$ gives no constraint on $\psi \in \F^{Y}$. 
 
  We next observe that we can write $S(m^\opp) \circ \mathrm{Graph}(dP) =
  \mathrm{Graph}(dO)$ for some quadratic form $O$. Recall that Prop.\ 
  \ref{lemma:qfls} states that a Lagrangian subspace $L$ of $\F^{Y}
  \oplus (\F^{Y})^\ast$ is of the form $\mathrm{Graph}(dO)$ if and only
  if $L$ has trivial intersection with $\{0\} \oplus (\F^N)^\ast$. But indeed,
  if $\psi = 0$ then $0$ is an extension of $\psi$ obeying the principle of minimum
  power, so $m_\ast i =dP_0 = 0$, and hence $i = 0$. 
  
  It remains to check that $O = Q$. This is a simple computation:
  \[
    O(\psi) = dO_\psi(\psi) = dO_\psi(\tilde\psi \circ m) = m_\ast
    dQ_\psi(\tilde\psi) = dP_{\tilde\psi}(\tilde\psi) = P(\tilde\psi) = Q(\psi),
  \]
  where $\tilde\psi$ is any extension of $\psi \in \F^{Y}$ obeying the principle
  of minimum power.
\end{proof}

\begin{theorem}\label{thm:dirichtolag}
The hypergraph functor $\widetilde{B} \maps \Dirich\Corel \to
\Lag\Corel$ constructed by applying Lemma \ref{lemma:deccorelfunctor} to the 
natural transformation $\widetilde{\beta}$ in Lemma~\ref{lem:dirichtolag}
makes the following square commute:
\[
\xymatrix{
  \Dirich\Cospan \ar[r]^{B} \ar[d]_{\square_\Dirich} 
  & \Lag\Cospan \ar[d]^{\square_\Lag}   
   \\
   \Dirich\Corel \ar[r]^{\widetilde{B}} 
  & \Lag\Corel  
}
\]
\end{theorem}

\begin{proof}
The square commutes by the functoriality of the decorated corelation construction.    
Namely, by Lemma~\ref{lemma:deccorelfunctor}, each 2-cell in this diagram gives 
one of the hypergraph functors in the above square:
\[
\tikzcdset{every label/.append style = {font=\scriptsize}}
\begin{tikzcd}[row sep=20pt,column sep=35pt]
\Fin\Set \ar[dd,hookrightarrow,""{name=left}] \ar[rr,rightarrow,"1"{name=top}]
\ar[dr,"\Dirich"'{name=b1_in}, outer sep=-2pt] 
& 
& |[alias=b1_out]| \Fin\Set  \ar[dd,hookrightarrow,""{name=right}]
\ar[dl,"\Lag", outer sep=-1pt] \\
& |[alias=set]|\Set \\
\Fin\Set;\Inj^\opp \ar[rr,hookrightarrow] \ar[ur,"\Dirich"{name=b2_in},outer
sep=-2pt] 
&
& |[alias=b2_out]| \Cospan(\Fin\Set) \ar[ul,"\Lag"', outer sep=-3pt]
\arrow[Rightarrow, from=b1_in, to=b1_out, shorten <=20pt, shorten >=50pt,
"\beta"{description,pos=.35}]
\arrow[Rightarrow, from=b2_in, to=b2_out, shorten <=20pt, shorten >=35pt,
"\widetilde{\beta}"{description,pos=.42}]
\arrow[from=left, to=set, "\circlearrowright" description, draw=none]
\arrow[from=right, to=set, "\circlearrowright" description, draw=none]
\end{tikzcd}
\]
and because the composite of the bottom and left 2-cells equals the composite of the 
top and right 2-cells, the composite $\widetilde{B} \circ \square_{\Dirich}$ equals
the composite $\square_{\Lag} \circ B$.
\end{proof}

In the Introduction, lacking the machinery developed later, we gave a very
concrete description of the black box functor. The final piece of the puzzle, 
to make that description match what we have now, is as follows.

Write $\iota \maps \partial N \to N$ for the inclusion of the terminals into the
set of nodes of the circuit, and $i\rvert^{\partial N}\maps X \to \partial N$,
$o\rvert^{\partial N}\maps Y \to \partial N$ for the respective factorisations
of $i$ and $o$ through $\iota$. Note that $[i,o] = \iota \circ
[i\rvert^{\partial N}, o\rvert^{\partial N}]$. 
 In the Introduction, we
introduced the twisted symplectification $S^t$, which we can now understand as
an application of the functor $S$ followed by the isomorphism
$\overline{\idn_X}$ of a standard symplectic space with its conjugate.  Then we
have the equalities of sets, and thus Lagrangian relations:
\begin{align*}
  &\phantom{= .}(\overline{\idn_X}\oplus \idn_Y) \circ S[i,o]^\opp \circ
  \mathrm{Graph}(dP) \\
  &= (\overline{\idn_X}\oplus \idn_Y) \circ S[i\rvert^{\partial
  N},o\rvert^{\partial N}]^\opp \circ S\iota^\opp \circ \mathrm{Graph}(dP) \\
  &= (\overline{\idn_X}\oplus \idn_Y) \circ S[i\rvert^{\partial
  N},o\rvert^{\partial N}]^\opp \circ \mathrm{Graph}(dQ) \\
  &= \bigcup_{v \in \mathrm{Graph}(dQ)} S^ti\rvert^{\partial N}(v) \times
  So\rvert^{\partial N}(v)
\end{align*}
where $P$ is the extended power functional and $Q$ is the power functional.
We see now that Thm.~\ref{thm:blackbox} is a restatement of Thm.~\ref{thm:main} in 
the Introduction.

\end{document}